\documentclass[12pt]{amsart}
\usepackage[usenames]{color}
\usepackage{enumerate}
\usepackage{tikz-cd}
\usepackage{idxlayout}
\usepackage{mathdots,bm,url}
\usepackage{amsfonts,amssymb,mathscinet,a4wide}
\usepackage[all]{xy}
\usepackage[stable]{footmisc}

\author{Erez Lapid}
\address{Department of Mathematics, Weizmann Institute of Science, Rehovot 76100 Israel}
\email{erez.m.lapid@gmail.com}
\author{Zhengyu Mao}
\address{Department of Mathematics and Computer Science, Rutgers University, Newark, NJ 07102, USA}
\email{zmao@rutgers.edu}
\title[Model Transition]{Model transition for representations of metaplectic type}
\date{\today}
\thanks{Authors partially supported by U.S.--Israel Binational Science Foundation Grant \# 057/2008}
\thanks{Second named author partially supported by NSF grant DMS 1000636 and by a fellowship from the Simons Foundation}

\newcommand{\rkn}{\mathbf{n}}                          
\newcommand{\model}{\mathfrak{M}}
\newcommand{\LQ}[1]{\operatorname{LQ}(#1)}
\newcommand{\tran}{{\mathcal T}}
\newcommand{\C}{\mathbb{C}}                            
\newcommand{\R}{\mathbb{R}}                            
\newcommand{\Z}{\mathbb{Z}}                            
\newcommand{\bs}{\backslash}
\newcommand{\Hom}{\operatorname{Hom}}
\newcommand{\vol}{\operatorname{vol}}
\newcommand{\temp}{\operatorname{temp}}
\newcommand{\OO}{\mathcal{O}}                         
\newcommand{\regint}{\int^{reg}}                         
\newcommand{\evencls}{\mathcal{C}}                                
\newcommand{\oddcls}{\mathcal{C}'}                                
\newcommand{\tclass}{\mathcal{C}}
\newcommand{\oddL}{\mathcal{U}}
\newcommand{\evenL}{\mathcal{V}}
\newcommand{\smth}{\operatorname{sm}}
\newcommand{\wgt}[1]{\nu(#1)}
\newcommand{\auxst}[1]{(\#,#1)}
\newcommand{\quot}[1]{#1_*}
\newcommand{\T}{{\tilde T}}
\newcommand{\orb}{\kappa}
\newcommand{\cc}[1]{C_c^\infty(#1)}
\newcommand{\ft}[1]{\widehat{#1}^*}
\newcommand{\st}{st}
\newcommand{\elemu}{{\epsilon}}
\newcommand{\elemv}{{\epsilon'}}
\newcommand{\alg}[1]{\mathbf{#1}}
\newcommand{\factor}{\Delta}
\newcommand{\udr}{\underline{u}}
\newcommand{\id}{\operatorname{id}}
\newcommand{\spcltrs}{S}

\newcommand{\dual}[1]{#1^{\vee}}
\newcommand{\FW}{\Upsilon}
\newcommand{\Mat}{\operatorname{Mat}}
\newcommand{\GL}{\operatorname{GL}}
\newcommand{\Sp}{\operatorname{Sp}}
\newcommand{\mira}{\mathcal{P}}                     
\newcommand{\Ind}{\operatorname{Ind}}                   
\newcommand{\ind}{\operatorname{ind}}                   
\newcommand{\Levi}{M}
\newcommand{\GLnn}{{\mathbb M}}
\newcommand{\eva}[1]{{\mathcal E}[#1]}
\newcommand{\avg}{{\mathcal L}}
\newcommand{\avgt}[1]{{\mathcal L}^{#1}}
\newcommand{\val}{\operatorname{val}}
\newcommand{\csgr}{\mathcal{CSGR}}                     
\newcommand{\Irr}{\operatorname{Irr}}                  
\newcommand{\rest}{\big|}                              
\newcommand{\meta}{\operatorname{meta}}                
\newcommand{\gen}{\operatorname{gen}}                  
\newcommand{\levi}{\varrho}                            
\newcommand{\toU}{\ell}
\newcommand{\toUbar}{\overline{\ell}}
\newcommand{\startran}[1]{\breve{#1}}             
\newcommand{\modulus}{\delta}
\newcommand{\diag}{\operatorname{diag}}
\newcommand{\Whit}{\mathbb{W}}                           
\newcommand{\WhitM}{\Whit^{\psi_{N_\GLnn}}}                  
\newcommand{\WhitML}{\Whit^{\psi_{N_\Levi}}}                  
\newcommand{\spclWM}{\cc{N_\GLnn\bs\mira_{\GLnn}',\psi_{N_\GLnn}}_{\natural}}
\newcommand{\fix}{{\flat}}
\newcommand{\fixx}{{\fix\fix}}
\newcommand{\symspace}{\mathfrak{s}}                 
\newcommand{\subUbar}{\bar U^\urcorner}
\newcommand{\subsubU}{U^\triangle}
\newcommand{\subsubUbar}{\bar U^\triangle}
\newcommand{\oldsubsubUbar}{\hat{\psi}_{\subsubUbar}}
\newcommand{\newchar}{\psi_{\bar U}}                
\newcommand{\altpsi}{\Psi}
\newcommand{\chixmt}{\altpsi_{X_\GLnn,t}}         
\newcommand{\chix}{\altpsi_{X_0}}       
\newcommand{\chinu}{\altpsi_{X_0,t}}      
\newcommand{\remu}{\bar{U}_{\triangle}}        
\newcommand{\remn}{N_{\GLnn,\triangle}}     
\newcommand{\lowub}{\bar{U}}                     
\newcommand{\upub}{U}                            
\newcommand{\N}{{\mathcal N}}    
\newcommand{\M}{{\mathcal M}} 
\newcommand{\LL}{{\mathcal X}} 
\newcommand{\Nc}{{\mathcal N_0}}   
\newcommand{\fel}{\epsilon_\pi}           
\newcommand{\few}{\hat{w}}              
\newcommand{\one}{\epsilon}        
\newcommand{\wnn}{{w_0^\GLnn}}
\newcommand{\wnnM}{{w_0^\Levi}}
\newcommand{\diagelemGLn}{E_0}
\newcommand{\sprod}[2]{\left\langle#1,#2\right\rangle}
\newcommand{\abs}[1]{\left|{#1}\right|}
\newcommand{\sm}[4]{\left(\begin{smallmatrix}{#1}&{#2}\\{#3}&{#4}\end{smallmatrix}\right)}
\newcommand{\acutee}{\tilde}
\newcommand{\der}{\operatorname{der}}                     
\newcommand{\MTF}{\tilde{\Upsilon}}
\newcommand{\per}[1]{\mathfrak{P}^{#1}}
\newcommand{\soment}{T''}
\newcommand{\Imk}{T}

\newcommand{\Etwon}{\mathfrak{E}}
\newcommand{\charEtwon}{\psi_{\Etwon}}
\newcommand{\auxf}{\mathfrak{F}}
\newcommand{\auxF}{{\otimes_i\cc{\Imk_i}}}
\newcommand{\vs}{{V_{\triangle}}}
\newcommand{\zigzag}{\mathcal{Z}}
\newcommand{\charzigzag}{\psi_{\zigzag}}
\newcommand{\alltriangle}{\mathcal{A}}
\newcommand{\charalltriangle}{\psi_{\alltriangle}}
\newcommand{\radunip}{{\mathfrak{V}}}
\newcommand{\radunipsgr}{{\radunip'}}
\newcommand{\centpsi}{C_T(\psi_\radunipsgr)}
\newcommand{\W}{M^*W}
\newcommand{\tranlw}{\tran_{(H_\GLnn,1)}^{(N_\GLnn,\psi_{N_\GLnn})}}
\newcommand{\tranwl}{\tran_{(N_\GLnn,\psi_{N_\GLnn})}^{(H_\GLnn,1)}}
\newcommand{\tranlwI}{\tran_{(H_M,1)}^{(N_M,\psi_{N_M})}}
\newcommand{\tranwlI}{\tran_{(N_M,\psi_{N_M})}^{(H_M,1)}}
\newcommand{\tranWL}{\tran_{(N,\psi_N)}^{*,(H,1)}}
\newcommand{\tranLW}{\tran_{(H,1)}^{(N,\psi_N)}}
\newcommand{\tranWW}{M^*}
\newcommand{\tranLA}{\tran_{(H,1)}^{(\alltriangle,\charalltriangle)} }
\newcommand{\tranWAu}{\tran_{(N,\psi_N)}^{(\alltriangle,\charalltriangle)}}
\newcommand{\tranWAv}{\tran_{(N,\psi_N)}}
\newcommand{\tranwa}{\tran_{(N_\GLnn,\psi_{N_\GLnn})}^{(\zigzag,\charzigzag)}}
\newcommand{\tranLV}{\tran_{(H,1)}^{(\radunip, \psi_\radunip)} }
\newcommand{\tranLE}{\tran_{(H,1)}^{(\Etwon, \charEtwon)}}
\newcommand{\tranMWL}{\tran_{(N,\psi_N)}^{(H,1)}}
\newcommand{\tranY}{\tran_{\auxf}}
\newcommand{\uder}{{N^{\der_*}}}
\newcommand{\num}{{\mathfrak{a}}}

\newtheorem{theorem}{Theorem}[section]
\newtheorem{lemma}[theorem]{Lemma}
\newtheorem{proposition}[theorem]{Proposition}
\newtheorem{remark}[theorem]{Remark}
\newtheorem{definition}[theorem]{Definition}
\newtheorem{corollary}[theorem]{Corollary}

\numberwithin{equation}{section}

\makeindex

\begin{document}

\begin{abstract}
We study the interplay between different models of the same
irreducible representation of the $F$-points of a reductive group
over a local field.
\end{abstract}

\maketitle

\setcounter{tocdepth}{1}
\tableofcontents

\section{Introduction}

Let $G$ be a locally compact group.
An important notion in representation theory is that of a \emph{model}.
Broadly speaking, a model is a representation space of $G$ (usually defined
geometrically)
which contains any irreducible constituent with multiplicity one.
Consider the case where $G$ is the $F$-points of a reductive group $\alg{G}$ defined over a local field $F$.
There are many examples of models which are realized on induces spaces $\Ind_H^G\chi$ where
$\chi$ is a character of a closed subgroup $H$ of $G$ -- i.e., on the space of smooth functions
on $G$ which are left $(H,\chi)$-equivariant with the right $G$-action.
(For simplicity we assume that $H$ is unimodular.)
In this case we say that the pair $(H,\chi)$ satisfies multiplicity one.
The representations which occur in this model are called distinguished
by $(H,\chi)$. They are characterized by the property that there exists
a non-zero functional on them which is $(H,\chi)$-equivariant.
If $\pi$ is such a representation then we denote by $\model^{(H,\chi)}\pi$
its realization in the model.
The most well-known case is the \emph{Whittaker model} in which $G$ is quasi-split,
$H$ is a maximal unipotent subgroup and $\chi$ is a non-degenreate character of $H$
(\cite{MR0404534}, \cite{MR0348047}).
The importance of the Whittaker model in representations theory and automorphic forms cannot
be overestimated.
For other instances of models, including recent progress cf.
\cite{MR1394521, MR1111204, MR2680495, MR2474319, MR2553879}.

In this paper we address the following question.
Let $G$ be a reductive group defined over a $p$-adic $F$,
$H_1$, $H_2$ two closed subgroups of $G$
defined over $F$ and $\chi_i$, $i=1,2$ characters of $H_i$.
Suppose that $(H_i,\chi_i)$, $i=1,2$ satisfy multiplicity one
and let $\pi$ be an irreducible representation of $G$ which is distinguished
by both $(H_i,\chi_i)$, $i=1,2$.
Then up to a scalar there is a unique isomorphism
\[
\model^{(H_1,\chi_1)}\pi\simeq\model^{(H_2,\chi_2)}\pi.
\]
Can one explicate such an isomorphism, as well as its inverse?
More precisely, can one find a $G$-equivariant integral transform from
$\model^{(H_1,\chi_1)}\pi$ to $\model^{(H_2,\chi_2)}\pi$
with an explicit inversion formula?

The question is meaningful for any irreducible representation $\pi$ of $G$ such that
\[
\dim\Hom_{H_1}(\pi,\chi_1)=\dim\Hom_{H_2}(\pi,\chi_2)=1
\]
(whether or not $(H_1,\chi_1)$ or $(H_2,\chi_2)$ satisfy multiplicity one in general).
We will call such an explicit map $\model^{(H_1,\chi_1)}\pi\rightarrow\model^{(H_2,\chi_2)}\pi$ a \emph{model transition}.
Assume for simplicity that $\chi_1$ and $\chi_2$ coincide on $H_1\cap H_2$. A first attempt would be to consider the map
\begin{equation} \label{eq: naive}
\varphi\mapsto\int_{H_1\cap H_2\bs H_2}\varphi(hg)\chi_2(h)^{-1}\ dh
\end{equation}
which \emph{formally} defines an intertwining operator $\Ind_{H_1}^G\chi_1\rightarrow\Ind_{H_2}^G\chi_2$.
(Often, this integral needs to be regularized.)
It turns out however that it is sometimes advantageous to replace the domain of integration by
$H_1\cap H_2\bs H_2'$ where $H_2'$ is a proper subgroup of $H_2$ containing $H_1\cap H_2$.
(This reflects the fact that we secretly work within a subgroup $G'$ of $G$ containing $H_1$ and consider maps
$\Ind_{H_1}^{G'}\chi_1\rightarrow\Ind_{H_2\cap G'}^{G'}\chi_2$.)
Of course this procedure does not give rise to a map from $\Ind_{H_1}^G\chi_1$ (or a suitable $G$-invariant subspace $X$ where the regularized
integral is defined) to $\Ind_{H_2}^G\chi_2$. However, one can hope that in certain cases for (suitable) $(H_1,\chi_1)$-distinguished
representations $\pi$ we have $\model^{(H_1,\chi_1)}\pi\subset X$ and $\tran(\model^{(H_1,\chi_1)}\pi)=\model^{(H_2,\chi_2)}\pi$.

We will give several examples of this phenomenon.

The first example is for the group $M=\GL_{2n}(F)$. Let $\pi$ be a tempered representation $\pi$ of $M$.
In particular, $\pi$ has a Whittaker model $\model^{(N_M,\psi_{N_M})}\pi$, where $N_M$ is the subgroup of upper unitriangular matrices
in $M$ and $\psi_{N_M}$ a non-degenerate character of $N_M$.
Let $H_M$ be the centralizer of $E=\diag(1,-1,\dots,1,-1)$ in $M$, isomorphic to $\GL_n\times\GL_n$.
We assume that $\pi$ has a model $\model^{(H_M,1)}\pi$ with respect to $H_M$.
(It is known that, for any $\pi$, this model is unique if it exists \cite{MR1394521}.)
Note that the character $\psi_{N_M}$ is trivial on $N_M\cap H_M$.

In one direction we define $\tranwlI:\model^{(N_M,\psi_{N_M})}\pi\rightarrow \model^{(H_M,1)}\pi$ by
\[
\tranwlI (W)(g):=\int_{N_M\cap H_M\bs \mira_{2n}\cap H_M} W(pg)\, dp,\,\,W\in \model^{(N_M,\psi_{N_M})}\pi
\]
where $\mira_{2n}$ is the mirabolic subgroup of $M$. The integral converges absolutely
(even if $\pi$ is only assumed to be unitarizable and generic).
Moreover, it is known that the image of $\tranwlI$ indeed lies in $\model^{(H_M,1)}\pi$  \cite{1211.1241}.
In the other direction when $\pi$ is tempered, define $\tranlwI: \model^{(H_M,1)}\pi\rightarrow \model^{(N_M,\psi_{N_M})}\pi$ by
\[
\tranlwI(L)(g):=\regint_{N_M\cap H_M\bs N_M}L(ng)\psi_{N_M}^{-1}(n) \,dn,\,\,L\in \model^{(H_M,1)}\pi.
\]
The meaning of the regularized integral (in a more general context) will be explained in \S\ref{sec: stable}.
It is a slight extension of the stable integrals studied in \cite{LMao5}.

We show that the maps $\tranwlI$
and $\tranlwI$%
are inverse to each other
(Proposition \ref{prop: inversionM}).

This result is a version of local Rankin--Selberg unfolding.
It is closely related to the material of \cite[\S18]{1203.0039}.
See also \cite[\S4]{LMao5} for a similar result in a different setup.

An immediate consequence is that we can define $\tranwlI$ with respect to the other mirabolic subgroup
without changing the result (see Proposition~\ref{prop: sametran} which is analogous to \cite[Corollary 7.2]{MR2787356}).

Next, we will `inflate' the $\GL_{2n}$ result to the group $G=\Sp_{2n}$ (a subgroup of $\GL_{4n}$ of rank $2n$).
We identify $M$ with the Levi subgroup of the Siegel parabolic subgroup $P=M\ltimes U$ of $G$ via $x\mapsto \diag(x,x^*)$.
Let $N$ be the subgroup of upper unitriangular matrices in $G$.
Let $H$ be the centralizer of $\diag(E,-E)$ in $G$; it is isomorphic to $\Sp_n\times\Sp_n$.

With $\pi$ as before, consider the induced representation $I(\pi,\frac12)=\Ind_P^G(\pi\abs{\det\cdot}^{\frac12})$
realized on the space of smooth functions $W:U\bs G\rightarrow\C$ such that for all $g\in G$ the function
$m\mapsto\modulus_P(m)^{-\frac12}\abs{\det m}^{-\frac12}W(mg)=\abs{\det m}^{-(n+1)}W(mg)$ on $M$ belongs to $\model^{(N_M,\psi_{N_M})}\pi$.
The representation $I(\pi,\frac12)$ admits a unique irreducible quotient (the Langlands quotient) which we denote by $\LQ{\pi}$.
The representation $\LQ{\pi}$ has a unique model $\model^{(H,1)}\LQ{\pi}$ (Theorem \ref{thm: MWL}).
Meanwhile let $\psi_N$ be the extension $\psi_{N_M}$ to $N$, trivial on $U$.
Then $\LQ{\pi}$ has also a unique model $\model^{(N,\psi_N)}\LQ{\pi}$ (Proposition \ref{prop: Tadic}).
It can be realized as the image of $I(\pi,\frac12)$ under the standard intertwining operator $\tranWW$ given by
\[
\W(g)=\int_U W(wug)\,du, \,\,W\in I(\pi,\frac12)
\]
for a suitable Weyl element $w$.
We define the map
\[
\tranMWL(W):=\int_{N\cap H\bs \mira_{4n}\cap H}W(h\cdot)\ dh, \ \ W\in \model^{(N,\psi_N)}\LQ{\pi}.
\]
Here $\mira_{4n}$ is the mirabolic subgroup of $\GL_{4n}$.
In the other direction we define
\[
\tranLW:\model^{(H,1)}\LQ{\pi}\rightarrow\model^{(N,\psi_N)}\LQ{\pi},
\]
by
\[
\tranLW(L)(g):=\regint_{H\cap N\bs N} L(ng)\psi_N^{-1}(n)\,dn,\,\,L\in \model^{(H,1)}\LQ{\pi}.
\]
Once again the regularized integral is defined in \S\ref{sec: stable}.
We show that $\tranMWL$ defines a model transition from $\model^{(N,\psi_N)}\LQ{\pi}$ to $\model^{(H,1)}\LQ{\pi}$,
whose inverse is $\tranLW$ (Theorem \ref{thm: MWL}).
In particular, the integral defining $\tranMWL$ is an $H$-invariant functional.

Note that in general $\Hom_N(\sigma,\psi_N)$ is finite-dimensional (but not necessarily zero or one-dimensional)
if $\sigma$ is an irreducible representation of $G$. (We do not know whether $\dim\Hom_H(\sigma,1)\le1$ for all $\sigma$.)

One can also realize $\model^{(H,1)}\LQ{\pi}$ by the map
\[
\tranWL:  I(\pi,\frac12)\rightarrow \model^{(H,1)}\LQ{\pi}
\]
given by the absolutely convergent integral
\[
\tranWL(W)(g):= \int_{P\cap H\bs H}\int_{N_M\cap H_M\bs\mira_{2n}\cap H_M}W(phg)\abs{\det p}^{-(n+1)}\,dp\,dh.
\]
We have the following commutative diagram (Proposition \ref{P: L}):\\
\begin{tikzcd}
I(\pi,\frac12)\arrow{rd}{\fel\tranWW} \arrow{r}{\tranWL} & \model^{(H,1)}\LQ{\pi} \arrow{d}{\tranLW}\\
                                                         & \model^{(N,\psi_N)}\LQ{\pi} \arrow{u}{\tranMWL}
\end{tikzcd}
\\where $\fel\in\{\pm1\}$ is the standard epsilon factor $\epsilon(\frac12,\pi,\psi)$ (which does not depend on $\psi$
since $\pi$ has a trivial central character).

This leads us to the second theme of this paper which is the compatibility of model transitions.
More precisely, given three models with respect to pairs $(H_i,\chi_i)$, $i=1,2,3$
with transition maps we will be interested in the commutativity of the diagram\\
\begin{tikzcd}
\model^{(H_1,\chi_1)}\pi\arrow{rr} \arrow{rd}& & \model^{(H_2,\chi_2)}\pi \arrow{dl}\\
 & \model^{(H_3,\chi_2)}\pi &
\end{tikzcd}

We will study this question for the representations $\LQ{\pi}$ above.
Following Ginzburg--Rallis--Soudry consider a certain character $\psi_\radunip$ on the unipotent radical $\radunip$ of the parabolic
subgroup of $G$ with Levi part $\GL_2\times\dots\times\GL_2$ ($n$ times) (see \S\ref{sec: Etwon}).
The representation $\LQ{\pi}$ admits a model $\model^{(\radunip, \psi_\radunip)}$ for $\LQ{\pi}$
(Proposition \ref{L: uniqueEtwon} which is based on \cite{MR1671452} and Proposition \ref{prop: Tadic}).
We also define its variants
$\model^{(\alltriangle, \charalltriangle)}$ and $\model^{(\Etwon, \charEtwon)}$
where $(\alltriangle, \charalltriangle)$ and $(\Etwon, \charEtwon)$ are obtained from $(\radunip, \psi_\radunip)$
by conjugation by certain elements of the normalizer of $H$ (see \S\ref{sec: Etwon}).
(Typically, $\Hom_{\radunip}(\sigma,\psi_\radunip)$ is infinite-dimensional for a general irreducible representation $\sigma$ of $G$.)

We will show the following commutative diagram of model transitions (Propositions \ref{prop: 14} and \ref{prop: main}):\\
\begin{tikzcd}
{} & \model^{(H,1)}\LQ{\pi} \arrow{ld}{\tranLA} \arrow{dd}{\tranLW} \arrow{rd}{\tranLE} \\
\model^{(\alltriangle, \charalltriangle)}\LQ{\pi} &
&  \model^{(\Etwon, \charEtwon)}\LQ{\pi}\\
& \model^{(N,\psi_N)}\LQ{\pi}  \arrow{ul}{\tranWAu} \arrow{uu}{\tranMWL} \arrow{ur}{\tranWAv^{(\Etwon, \charEtwon)}}
\end{tikzcd}
\\ The diagonal maps $\tran_{(H_1,\chi_1)}^{(H_2,\chi_2)}$, where $(H_1,\chi_1)$ is either $(H,1)$ or $(N,\psi_N)$
and $(H_2,\chi_2)$ is either $(\alltriangle, \charalltriangle)$ or $(\Etwon, \charEtwon)$, are given by a suitable regularization of
\eqref{eq: naive}.
Actually, the map $\tranWAv^{(\Etwon, \charEtwon)}$ on (a suitable subspace of) $\Ind_N^G\psi_N$ depends on some additional choices
but its restriction to $\model^{(N,\psi_N)}\LQ{\pi}$ turns out to be independent of these choices. See Proposition \ref{prop: main}
for the precise statement.

The main tool for proving the results is repeated applications of Fourier inversion (see Appendix \ref{a: int}).
This idea is of course not new. It was used by Jacquet--Shalika in the analysis of the exterior square $L$-function
\cite{MR1044830} and plays an important role in the descent construction of Ginzburg--Rallis--Soudry \cite{MR1671452}.
In fact, this paper owes a lot to their work.
Our point of view however is slightly different as we are interested in explicit computations in the various models.
We also mention recent work of Bernstein--Reznikov \cite{1312.2898} for a closely related theme.

The integral transforms $\tranLA$ and $\tranLE$ are related by conjugation by a certain element of the normalizer of $H$ (see \S\ref{sec: Etwon}).
It follows from the above diagram that on $\model^{(N,\psi_N)}\LQ{\pi}$, there is a functional equation relating the integral transforms
$\tranWAu$ and $\tranWAv^{(\Etwon, \charEtwon)}$. A version of this functional equation is stated in Corollary~\ref{cor: fornextpaper}.


The results of this paper will be used in our work on Whittaker-Fourier coefficients of
automorphic forms on the metaplectic double cover of $\Sp_n$.
More precisely, the results (and their variants proved in \S\ref{sec: Mfactor}) will be used in proving a conjectural local identity stated in
\cite{1401.0198}. This local identity is the local counterpart (in the case of the metaplectic group) of a global conjecture which
we proposed in \cite{LMao5}. (The conjecture in \cite{LMao5} is for any quasi-split group, as well as the metaplectic group.)

The paper is organized as follows.
In \S\ref{sec: stable} we introduce a rather general procedure for regularization of oscillatory integrals
involving generic characters.
In \S\ref{sec: invGLnn} we consider $\GL_{2n}$ and the model transition between the Whittaker model and the
$(\GL_n\times\GL_n,1)$-model (for distinguished tempered representations $\pi$).
The rest of the paper will be devoted to the case of Langlands quotient $\LQ{\pi}$ on $\Sp_{2n}$ (with $\pi$ as before).
In \S\ref{sec: LQ1} we prove the uniqueness of $(\Sp_n\times\Sp_n,1)$-model for $\LQ{\pi}$ and factor the intertwining operator
explicitly through a model transition $\tranLW$ between the $(\Sp_n\times\Sp_n,1)$-model and the degenerate Whittaker model
$(N,\psi_N)$ of $\LQ{\pi}$. The inverse of $\tranLW$ is constructed in \S\ref{sec: invtrans}.
In \S\ref{sec: Etwon} we introduce the additional models $\model^{(\radunip, \psi_\radunip)}$, $\model^{(\alltriangle, \charalltriangle)}$
and $\model^{(\Etwon, \charEtwon)}$ of $\LQ{\pi}$.
In \S\ref{sec: analytic} and \S\ref{sec: lemma} we introduce the (somewhat delicate) family of operators $\tranWAv^{t,\auxf}$
(closely related to $\tranWAv^{(\Etwon, \charEtwon)}$ mentioned above).
The compatibility of the various model transitions for the $(\Sp_n\times\Sp_n,1)$, $(N,\psi_N)$ and $(\Etwon, \charEtwon)$
models is proved in \S\ref{sec: whitdesc} and \S\ref{sec: prf of prop: main}.
Additional results which are variants of the above (and which will be used in a subsequent paper) are proved in \S\ref{sec: Mfactor}.

There are also three appendices. In the first one we provide the setup for the `root exchange' procedure
in the spirit of \cite[\S7.1]{MR2848523}.
The second appendix is a certain estimate which is a close variant of a result of Waldspurger in \cite[\S3]{WaldAst3461}.
The last appendix, due to Marko Tadi\'c, establishes the uniqueness of the degenerate Whittaker models $(N,\psi_N)$ for the
representations $\LQ{\pi}$. We take this opportunity to thank Tadi\'c for his contribution.
We also thank Joseph Bernstein, Patrick Delorme, Atsushi Ichino, Andre Reznikov, Yiannis Sakellaridis, David Soudry and Akshay Venkatesh for helpful correspondence.

\subsection{Notations}
Throughout the paper let $F$ be a $p$-adic field with ring of integer $\OO$.
We fix a (continuous) non-trivial character $\psi:F\rightarrow\C^*$.
Let $q$ be the size of residue field of $F$. Thus $\abs{x}=q^{-\val(x)}$ for $x\in F$.
We take the self-dual Haar measure $dx$ on $F$ with respect to $\psi$. We take the multiplicative measure of $F^*$ to be $d^*x=dx/\abs{x}$.

We typically denote algebraic varieties (or groups) over $F$ by boldface letters (e.g., $\alg{X}$) and denote their set (or group) of $F$-points
by the corresponding plain letter (e.g., $X$). (In most cases $\alg{X}$ will be clear from the context.)

For an $\ell$-group $Q$ let $\csgr(Q)$ be the set of compact open subgroups of $Q$.
We denote the group of continuous characters of $Q$ (with the compact-open topology) by $\widehat{Q}$.
In particular, if $Q$ is abelian then $\widehat{Q}$ is the Pontryagin dual of $Q$.
Let $Z_Q$ be the center of $Q$, and $e$ the identity element of $Q$.
If $Q$ acts on a vector space $W$ and $Q'$ is a subgroup of $Q$,
we denote by $W^{Q'}$ the subspace of $Q'$-fixed points.

If $A$ and $B$ are subgroups of $Q$ and $A$ normalizes $B$, we write $\modulus_{A;B}$ for the corresponding modulus function on $A$.
In particular, we write $\modulus_A$ for the modulus function of $A$ acting on itself by conjugation.

If $Q'$ is a closed subgroup of $Q$ and $\chi$ is a character of $Q'$, we denote by $C(Q'\bs Q,\chi)$
(resp., $C^{\smth}(Q'\bs Q,\chi)$, $C_c(Q'\bs Q,\chi)$, $\cc{Q'\bs Q,\chi}$)  the spaces of
continuous (resp. $Q$-smooth,\footnote{i.e., right-invariant under an open subgroup of $Q$}, compactly supported modulo $Q'$,
smooth and compactly supported modulo $Q'$)
complex-valued left $(Q',\chi)$-equivariant functions on $Q$. When $\chi=1$ we omit it from the notation.

We use the following bracket convention for iterated integrals:
$\iint \,( \iint\,\ldots)\ldots$ implies that the inner integrals converge
as a double integral and after evaluating them,
the outer double integral is absolutely convergent.

Let $f$ be a continuous function on $Q$ and $\Psi$ a function on a subgroup $A$ of $Q$. Let $Q'$ be a subgroup of $Q$.
We denote by $\avg_{A,\Psi}\circ f(g)$ the integral $\int_{A} f(ag)\Psi(a)\ da$ and by $\avg^{Q'}_{A,\Psi}\circ f(g)$
the integral $\int_{(A\cap Q')\bs A} f(ag)\Psi(a)\ da$. Implicit in the notation is that the
integral makes sense and converges. If $\Psi\equiv1$ then we suppress it from the notation.
We denote by $\avg^{reg}_{A,\Psi}\circ f(g)$ and $\avg^{reg,Q'}_{A,\Psi}\circ f(g)$ an integral that is regularized in a certain way.
We caution the reader that the regularization depends on the context. Hopefully this will not cause confusion.

Let $I_m$ \index{$I_m$} be the identity matrix in $\GL_m$, $w_m$ the $m\times m$-matrix with ones on
the non-principal diagonal and zeros elsewhere; $g^t$ is the transpose of a matrix $g$.

We use the notation $a\ll_d b$ to mean that $\abs{a}\leq cb$ with $c$ a constant depending on $d$.

We denote by $\Irr Q$ the set of irreducible smooth (complex) representations (up to equivalence) of a group $Q$. \index{$\Irr$}

More specific notation will be introduced in \S\ref{sec: notationGL2n} and \S\ref{sec: notation}. For the convenience of the reader
there is an index of symbols at the end of the paper.

\section{Some regularized integrals} \label{sec: stable}
\subsection{Stable integral}

We extend the definition of stable integral introduced in \cite{LMao5}.
Suppose that $\alg{U}$ is a unipotent group over $F$ and $A$ is a compact group which acts continuously on $U$.
Denote by $\csgr^{A}(U)$ the set of compact open $A$-invariant subgroups of $U$.
This is a directed set. Moreover, any relatively compact subset of $U$ is contained in an element of $\csgr^{A}(U)$.

\label{sec: stint}

Fix a Haar measure $du$ on $U$.
\begin{definition} \label{def: stablenormal}
Let $f$ be a continuous function on $U$.
We say that $f$ has a \emph{$\auxst{A}$-stable integral} over $U$ if there exists $U_1\in\csgr^{A}(U)$ such that for any
$U_2\in\csgr^{A}(U)$ containing $U_1$ we have
\begin{equation} \label{eq: comvalnormal}
\int_{U_2}f(u)\ du=\int_{U_1}f(u)\ du.
\end{equation}
In this case we write $\int^{\auxst{A}}_U f(u)\ du$ for the common value \eqref{eq: comvalnormal}. \index{$\int^{\auxst{A}}$}
In other words, $\int^{\auxst{A}}_U f(u)\ du$ is the limit (if it exists) of the net $(\int_{U_1}f(u)\ du)_{U_1\in\csgr^A(U)}$
with respect to the discrete topology of $\C$.
\end{definition}

The case where $A$ is the trivial group was considered in \cite[\S2]{LMao5} under the name \emph{stable integral}.

\begin{remark}\label{rem: largerK}
It is clear that if $A'\subset A$ and $f$ has a $\auxst{A'}$-stable integral over $U$, then it also has a $\auxst{A}$-stable integral (with the same value).
Similarly, any (right or left) translate of $f$ by an element of $U$ has an $A$-stable integral, with the same value.
\end{remark}

\begin{definition}
Suppose $A$ is an $\ell$-group acting on $U$ and let $f$ be as before.
We say that the $A$-stable integral of $f$ exists if $\int^{\auxst{A'}}_U f(u)\ du$ exists for any compact open subgroup $A'\subset A$
(or equivalently, for any sufficiently small $A'\in\csgr(A)$).
We write this value, which by the previous remark does not depend on the choice of $A'$, as
$\int^{\st,A}_U f(u)\ du$. \index{$\int^{\st,A}$}
\end{definition}

\begin{remark} \label{rem: aut}
It is clear that if $\alpha$ is an automorphism of $U$ which commutes with the action of $A$
and which multiplies the Haar measure on $U$ by $\Delta_\alpha$ then
\[
\int^{\auxst{A}}_U f(u)\ du=\Delta_\alpha\int^{\auxst{A}}_U f\circ\alpha(u)\ du
\]
whenever one side is defined. Similarly for the $A$-stable integral.
\end{remark}

\subsection{A regularization}
We consider the following setting.
Suppose that $U_0$ is a unipotent subgroup of a group $Q$ and $U_1$ is a normal subgroup of $U_0$ such that
$\quot{U}=U_1\bs U_0$ is abelian.
We identify $\widehat{\quot{U}}$ with the group of characters on $U_0$ that are trivial on $U_1$.
Let $\psi_{U_0}$ be a character of $U_0$ and let $\T$ be a torus of $Q$ normalizing $U_0$ and $U_1$.
Assume that the restriction $\psi_{U_1}$ of $\psi_{U_0}$ to $U_1$ is invariant under conjugation by $\T$.
(In the cases at hand $\psi_{U_1}$ is usually trivial.)
We say that $\psi_{U_0}$ is $(\T,U_0,U_1)$-generic (or simply
$\T$-generic if $U_0$ and $U_1$ are clear from the context) if the map
\[
\orb:\T\rightarrow\widehat{\quot{U}},\ \ \ \orb(t)=\psi_{U_0}^{-1}(t\cdot t^{-1})\psi_{U_0}
\]
is open.

The prototype is the case where $U_0$ is a maximal unipotent subgroup of a quasi-split group,
$\psi_{U_0}$ is a non-degenerate character in the usual sense, $U_1=U_0^{\der}$ the derived subgroup of $U_0$ and $\T$ is a maximal torus in the Borel
subgroup containing $U_0$.

Assume that $\psi_{U_0}$ is $(\T,U_0,U_1)$-generic. For any $\phi\in C_c(\T)$ we write $\orb_*\phi$ for the function on $\widehat{\quot{U}}$ given by
\[
\orb_*\phi(\chi)=\int_{\orb^{-1}(\chi)}\phi(t)\ dt
\]
where of course the integral over the empty set is interpreted as $0$.
Note that $\orb^{-1}(\chi)$, if not empty, is a coset of $C_{\T}(\psi_{U_0})$, the stabilizer of $\psi_{U_0}$
under the conjugation action of $\T$. We take the measure on $\orb^{-1}(\chi)$ to be the translate of the Haar measure on $C_T(\psi_{U_0})$.
Since $\orb$ is open, $\orb_*:\cc{\T}\rightarrow\cc{\widehat{\quot{U}}}$.

Observe that $\orb(tt_0)=\orb(t)^{t_0}\orb(t_0)$ (where the superscript denotes the conjugation action) and therefore,
if $\phi_{t_0}=\phi(\cdot t_0)$ then
$\orb_*\phi_{t_0}(\chi)=\orb_*\phi(\chi^{t_0}\orb(t_0))$. It follows that for suitable normalization of Haar measures we have
\[
\int_{\T}\modulus_{\T;\quot{U}}(t)\phi(t)\orb^*f(t)\ dt=\int_{\widehat{\quot{U}}}\orb_*\phi(\chi)f(\chi)\ d\chi
\]
for any $f\in C(\widehat{\quot{U}})$ where $\orb^*f=f\circ\orb$. (It suffices to check this for $f\equiv1$.
Then the right-hand side defines a $(\T,\modulus_{\T;\quot{U}}^{-1})$-equivariant linear form on $\cc{\T}$.)
In particular, for any $\phi\in C_c^\infty(\T)$ the Fourier transform of $\orb_*\phi$ is given by
\begin{equation} \label{eq: FTorb}
\widehat{\orb_*\phi}(u)=\int_{\T}\modulus_{\T;\quot{U}}(t)\phi(t)\orb(t)(u)\ dt.
\end{equation}

\begin{lemma} \label{lem: genstab}
Let $U_0,U_1,\T,\psi_{U_0}$ be as above and assume that $\psi_{U_0}$ is $(\T,U_0,U_1)$-generic.
Let $U_2$ be a subgroup of $U_1$ which is stable under $\T$ and such that $\psi_{U_1}\rest_{U_2}$ is trivial.
Let $f\in C^{\smth}((\T'\ltimes U_2)\bs Q)$ for some $\T'\in\csgr(\T)$.
Suppose that $f(\cdot g)\in L^1(U_2\bs U_1)$ for all $g\in Q$.
Let
\[
\varphi_f(g)=\int_{U_2\bs U_1}f(ug)\psi_{U_1}(u)^{-1}\ du, \ \ g\in Q.
\]
Then $\varphi_f\rest_{U_0}\psi_{U_0}^{-1}$ descends to a function on $\quot{U}$ which has a $\T$-stable integral.

Moreover if $\phi\in \cc{\T'}$, let $R(\phi)f=\int_{\T'}\phi(t)f(\cdot t)\ dt$, then
\begin{equation}\label{eq: stableconv}
\int_{\quot{U}}^{\st,\T}\varphi_{R(\phi)f}(u)\psi_{U_0}^{-1}(u)\ du=
\int_{\quot{U}}\ft{\phi}(u) \varphi_f(u)\psi_{U_0}(u)^{-1}\ du
\end{equation}
with $\ft{\phi}(v)=\int_{\T'}\phi(t)\orb(t)(v)\ dt\in \cc{\quot{U}}$.
\end{lemma}

\begin{proof}
Given $\phi$, let $\T_0\in\csgr(\T')$ containing the support of $\phi$. We take any $U_c\in \csgr^{\T_0}(\quot{U})$.
Then
\[
\int_{U_c}\varphi_{R(\phi)f}(v)\psi_{U_0}^{-1}(v)\ dv
=\int_{\T_0}\int_{U_c}\big(\int_{U_2\bs U_1}\phi(t)f(uvt)\psi_{U_1}(u)^{-1}\ du\big)\psi_{U_0}^{-1}(v)\ dv \ dt
\]
which by a change of variables $(u,v)\mapsto (tut^{-1},tvt^{-1})$ becomes
\begin{multline*}
\int_{\T_0}\int_{U_c}\big(\int_{U_2\bs U_1}\phi(t) f(uv)\psi_{U_1}(u)^{-1}\ du\big)\psi_{U_0}^{-1}(tvt^{-1})\ dv \ dt\\=
\int_{\T_0}\int_{U_c}\phi(t)\varphi_f(v)\psi_{U_0}^{-1}(tvt^{-1})\ dv \ dt=
\int_{U_c}\ft{\phi}(v) \varphi_f(v)\psi_{U_0}(v)^{-1}\ dv.
\end{multline*}
Note that $\varphi_f\rest_{U_0}\psi_{U_0}^{-1}$ descends to a function on $\quot{U}$.
Since by \eqref{eq: FTorb} $\ft{\phi}$ is the Fourier transform of $\orb_*(\phi)$, it is in $\cc{\quot{U}}$.
In particular, the above integral is independent of $U_c$ as long as $U_c$ contains the support of $\ft{\phi}$.
Thus, $\varphi_{R(\phi)f}\psi_{U_0}^{-1}$ has a $\auxst{\T_0}$-stable integral.

As for any $\T_0\in\csgr(\T)$ there is $\phi\in\cc{\T'}$ with support in $\T_0$ such that $R(\phi)f=f$, we proved that $\varphi_f\rest_{U_0}\psi_{U_0}^{-1}$ has $\auxst{\T_0}$-stable integral for any $\T_0\in\csgr(\T)$. Thus it has
$\T$-stable integral. The identity in the Lemma follows from the above calculation.
\end{proof}

In the setting of the Lemma, we will write \index{$\regint$}
\[
\regint_{U_2\bs U_0}f(ng)\psi_{U_0}^{-1}(n)\ dn=
\int^{st,\T}_{\quot{U}}\big(\int_{U_2\bs U_1}f(ung)\psi_{U_1}(u)^{-1}\,du\big)\psi_{U_0}^{-1}(n)\,dn.
\]
We record some straightforward properties of this regularized integral in the following Lemma.

\begin{lemma}\label{L: stableprop} Let $f$ be as in the previous Lemma. Then
\begin{enumerate}
\item \label{item: reginv}
$\regint_{U_2\bs U_0}f(ng)\psi_{U_0}^{-1}(n)\ dn\in C^{\smth}(U_0\bs Q,\psi_{U_0})$.
\item \label{item: extendint}
If $f\rest_{U_0}\in L^1(U_2\bs U_0)$, then $\regint_{U_2\bs U_0}f(ng)\psi_{U_0}^{-1}(n)\ dn=\int_{U_2\bs U_0}f(ng)\psi_{U_0}^{-1}(n)\ dn$.
\item
If $t\in \T$ stabilizes $\psi_{U_0}$, then $\regint_{U_2\bs U_0}f(nt)\psi_{U_0}^{-1}(n)\ dn=\modulus_{\T;U_2\bs U_0}(t)\regint_{U_2\bs U_0}f(tn)\psi_{U_0}^{-1}(n)\ dn$.
\end{enumerate}
\end{lemma}

We end this section with a couple of remarks
\begin{remark}
Let $Q$ be an $\ell$-group and $R\subset Q'$ closed subgroups.
Suppose that $f\in C^{\smth}(R\bs Q,\modulus_R\modulus_Q^{-1}\rest_R)$ and that the integral
$\int_{R\bs Q} f(q)\ dq$ converges absolutely. Then the same is true for $\int_{R\bs Q'}
f(q')\modulus_Q(q')\modulus_{Q'}^{-1}(q')\ dq'$. More precisely, (assuming for simplicity that
$\modulus_Q\rest_{Q'}=\modulus_{Q'}$) if $K_0\in\csgr(Q)$ is such that $f$ is right $K_0$-invariant then
\begin{equation} \label{eq: effsgrbnd}
\int_{R\bs Q'}\abs{f(q')}\ dq'\le\vol_{Q'\bs Q}(\overline{K_0})^{-1}\int_{R\bs Q}\abs{f(q)}\ dq
\end{equation}
where $\overline{K_0}$ is the image of $K_0$ in $Q'\bs Q$.
\end{remark}

\begin{remark} \label{rem: partint}
Let $\alg{U}$ be a unipotent group over $F$ and $\alg{U'}$, $\alg{V}$ closed subgroups.
Suppose that $f\in C^{\smth}(V\bs U)$ and that the integral
$\int_{V\bs U} f(u)\ du$ converges absolutely. Then the same is true for $\int_{V'\bs U'}f(u')\ du'$
where $V'=V\cap U'$.

Indeed, by induction on $\dim\alg{U}-\dim\alg{U'}$ we may assume that $\alg{U'}$ is a maximal proper subgroup of $\alg{U}$
and hence (since $\alg{U}$ is unipotent) $\dim\alg{U}-\dim\alg{U'}=1$. If $\alg{V}\subset\alg{U'}$ then we apply the previous remark.
Otherwise, $\int_{V'\bs U'}f(u')\ du'=\int_{V\bs U}f(u)\ du$ and the statement is obvious.
\end{remark}

Finally, we have the following elementary result.

\begin{lemma}\label{L: elemC}
Let $G_0$ be an $\ell$-group and $C$, $D$ closed subgroups with a closed embedding $\iota:C\rightarrow\widehat D$.
(We do not assume that either $C$ or $D$ is abelian, or that $\iota$ is a homomorphism.)
Assume that $f\in C^{\smth}(G_0)$ and $\chi\in\widehat D$ are such that
\begin{equation} \label{eq: fcd}
f(cd)=f(c)\chi(d)\sprod{\iota(c)}d\text{ for all }c\in C, d\in D.
\end{equation}
Then $f\rest_C$ is compactly supported.
Moreover, let $K_0\in \csgr(G_0)$, $\Omega_1$ a compact subset of $\widehat D$ and $\Omega_2$ a compact subset of
the space of closed embeddings from $C$ to $\widehat D$ (with the compact-open topology). Assume that the images (in $\widehat D$) of
all $\iota\in\Omega_2$ coincide.\footnote{This guarantees that $\cup_{\iota\in\Omega_2}\iota^{-1}(\Omega)$
is precompact for any compact $\Omega\subset\widehat D$.}
Then there exists a compact set $\Omega_C$ depending on $K_0$, $\Omega_1$ and $\Omega_2$ such that for any $f\in C(G_0)^{K_0}$
satisfying \eqref{eq: fcd}, the support of $f\rest_C$ is contained in $\Omega_C$.
\end{lemma}

\section{Model transition for representations of metaplectic type} \label{sec: invGLnn}

\subsection{Notations} \label{sec: notationGL2n}
\begin{itemize}
\item Throughout the paper we fix an integer $\rkn\ge1$ (not to be confused with a running variable $n$).

\item $\GLnn$ is the group $\GL_{2\rkn}$. \index{$\GLnn$}

\item $N_\GLnn$ is the standard maximal unipotent subgroup of  $\GLnn$ consisting of upper unitriangular matrices;
$T_\GLnn$ is the maximal torus of $\GLnn$ consisting of diagonal matrices;
$B_\GLnn=T_\GLnn\ltimes N_\GLnn$ is the Borel subgroup. $K_\GLnn$ is the standard maximal compact subgroup of $\GLnn$.
\index{$B_\GLnn$, $N_\GLnn$, $T_\GLnn$} \index{$K_\GLnn$}

\item $\psi_{N_\GLnn}$ is the non-degenerate character of $N_\GLnn$ given by \index{$\psi_{N_\GLnn}$}
\[
\psi_{N_\GLnn}(u)=\psi(u_{1,2}+\dots+u_{2\rkn-1,2\rkn}).
\]

\item $N_\GLnn^{\der}$ is the derived group of $N_\GLnn$.

\item $\wnn=w_{2\rkn}\in \GLnn$ \index{$\wnn$} represents the longest Weyl element of $\GLnn$.

\item $g\mapsto g^*$ is the outer automorphism of $\GLnn$ given by $g^*=(\wnn)^{-1}\,(g^t)^{-1} \wnn$.

\item \index{$H_\GLnn$} $H_\GLnn$ is the centralizer of $E=\diag(1,-1,\dots,1,-1)$ in $\GLnn$, isomorphic to $\GL_\rkn\times\GL_\rkn$.

\item $\Irr_{\meta}\GLnn$ is the set of irreducible representations of $\GLnn$ of metaplectic type,
i.e. those which admit a nontrivial $H_\GLnn$-invariant form. Such an invariant form is unique up to a scalar \cite{MR1394521}.
Any $\pi\in \Irr_{\meta}\GLnn$ is self-dual. \index{$\Irr_{\gen}\GLnn$, $\Irr_{\temp}\GLnn$, $\Irr_{\meta}\GLnn$}

\item We write $\Irr_{\gen}\GLnn$ and $\Irr_{\temp}\GLnn$ for the sets of irreducible generic and tempered representations
of $\GLnn$ respectively.

\item $\mira_m$ is the mirabolic subgroup of $\GL_m$ consisting of the elements $g$ whose last row is $\xi_m=(0,\ldots,0,1)$.
Let $\mira_{\GLnn}=\mira_{2\rkn}$.
\index{$\mira_m$, $\mira_{\GLnn}$}

\item The Lie algebra $\mathfrak{M}$ of $\GL_m$ consists of the $m\times m$-matrices $X$ over $F$.
Let $\mathfrak{M}_{\OO}$ be the lattice of integral matrices in $\mathfrak{M}$.
For any algebraic subgroup $\alg{Q}$ of $\GL_m$ defined over $F$ let $\mathfrak{q}\subset\mathfrak{M}$ be the Lie algebra of $\alg{Q}$.
The lattice $\mathfrak{q}\cap \mathfrak{M}_{\OO}$ of $\mathfrak{q}$ gives rise to a gauge form of $\alg{Q}$
(determined up to multiplication by an element of $\OO^*$) and we use it to define a Haar measure on $Q$ by the recipe of \cite{MR0217077}.

\end{itemize}

\subsection{}
For $\pi\in\Irr_{\gen, \meta}\GLnn$ we would like to realize the $H_\GLnn$-invariant form $\per{H_{\GLnn}}$ explicitly on the Whittaker  model
$\WhitM(\pi)=\model^{(N_\GLnn,\psi_{N_\GLnn})}\pi$.
Conversely, we would like to express the Whittaker function in terms of $\per{H_{\GLnn}}$.
This is analogous to the situation of the inner product considered in \cite{LMao4}.

\begin{proposition}\cite[Lemma 2.1]{1211.1241} \label{P: LM}
Assume $\pi\in\Irr_{\gen, \meta}\GLnn$ is unitarizable. Then for any $W\in\WhitM(\pi)$ the integral
\begin{equation} \label{def: LGLnn}
\per{H_{\GLnn}} (W):=\int_{N_\GLnn\cap H_\GLnn\bs \mira_{\GLnn}\cap H_\GLnn} W(p)\, dp
\end{equation}
\index{$\per{H_{\GLnn}}$} converges and defines a nontrivial $H_\GLnn$-invariant linear form on $\WhitM(\pi)$.
Thus $\tranwl(W)(g):=\per{H_{\GLnn}} (\pi(g)W)$
\index{$\tranwl$} defines a map from $\model^{(N_\GLnn,\psi_{N_\GLnn})}\pi$ to $\model^{(H_\GLnn,1)}\pi$.
\end{proposition}

\begin{proof}
Writing (for a suitable Haar measure of $K_{\GLnn}\cap\mira_{\GLnn}\cap H_\GLnn$)
\begin{multline*}
\per{H_{\GLnn}}(W)=
\int_{\mira_\GLnn\cap B_\GLnn\cap H_\GLnn\bs \mira_{\GLnn}\cap H_\GLnn}
\int_{N_\GLnn\cap H_\GLnn\bs\mira_\GLnn\cap B_\GLnn\cap H_\GLnn}W(bp)\modulus_{\mira_\GLnn\cap B_\GLnn\cap H_\GLnn}(b)^{-1}
\modulus_{\mira_\GLnn\cap H_\GLnn}(b)\, db\, dp\\=
\int_{K_{\GLnn}\cap\mira_{\GLnn}\cap H_\GLnn}
\int_{T_\GLnn\cap\mira_\GLnn}W(tk)\modulus_{\mira_\GLnn\cap B_\GLnn\cap H_\GLnn}(t)^{-1}\modulus_{\mira_\GLnn\cap H_\GLnn}(t)\, dt\, dk
\end{multline*}
the convergence follows from the estimates of \cite[Lemma 2.1]{LMao4}
and the fact that on $\mira_{\GLnn}\cap B_\GLnn\cap H_\GLnn$ we have
\[
\modulus_{\mira_{\GLnn}\cap B_\GLnn\cap H_\GLnn}\modulus_{\mira_{\GLnn}\cap H_\GLnn}^{-1}=\modulus_{B_\GLnn}^{\frac12}\abs{\det\cdot}^{-\frac12}.
\]
Clearly, $\per{H_{\GLnn}}$ defines a $\mira_{\GLnn}\cap H_\GLnn$-invariant linear form which is non-trivial since
we can take $W\rest_{\mira_{\GLnn}}$ to be an arbitrary function in $C_c^\infty(N_\GLnn\bs\mira_{\GLnn},\psi_{N_\GLnn})$.
Lemma 2.1 of \cite{1211.1241} states that the space of $\mira_{\GLnn}\cap H_\GLnn$-invariant linear forms on $\pi$ is one-dimensional.
As $\pi$ possesses a nontrivial $H_\GLnn$-invariant linear form, $\per{H_{\GLnn}}$ must be $H_\GLnn$-invariant.
\end{proof}

We observe the following fact:

\begin{lemma}\label{L: cuspidal}
Suppose that $W\in C_c^\infty(N_\GLnn\bs\mira_{\GLnn},\psi_{N_\GLnn})$.
Then $\avg^{N_\GLnn}_{\mira_{\GLnn}\cap H_\GLnn} \circ W\in\cc{\mira_{\GLnn}\cap H_\GLnn\bs \mira_{\GLnn}}$.
In particular, if $\pi\in\Irr_{\gen, \meta}\GLnn$ and $W\in\WhitM(\pi)$
is such that $W\rest_{\mira_{\GLnn}}\in C_c^\infty(N_\GLnn\bs\mira_{\GLnn},\psi_{N_\GLnn})$
then the function $g\mapsto\per{H_{\GLnn}} (W(\cdot g))$ on $\mira_{\GLnn}$ is compactly supported
modulo $\mira_{\GLnn}\cap H_{\GLnn}$.
\end{lemma}

\begin{proof}
Let $\pi'$ be \emph{any} $H_{\GLnn}$-relatively cuspidal representation of $\GLnn$.
(See \cite[\S3]{1401.0198} for their existence.)
By the condition on $W$, there exists $W'\in \WhitM(\pi')$ such that $W'\rest_{\mira_{\GLnn}}=W$.
It is clear from the definition of $\per{H_\GLnn}$ that
$\avg^{N_\GLnn}_{\mira_{\GLnn}\cap H_\GLnn} W(g)=\avg^{N_\GLnn}_{\mira_{\GLnn}\cap H_\GLnn} W'(g)=\per{H_{\GLnn}} (\pi'(g)W')$ for all $g\in\mira_{\GLnn}$.
On the other hand, by our assumption on $\pi'$ the function $g\mapsto\per{H_{\GLnn}} (\pi'(g)W')$ is compactly supported on $H_{\GLnn}\bs\GLnn$.
The claim follows since $H_\GLnn\cap \mira_{\GLnn}\bs\mira_{\GLnn}$ is closed in $H_{\GLnn}\bs\GLnn$.
(Identifying $\mira_{\GLnn}\bs\GLnn$ with the non-zero vectors in $F^{2\rkn}$, the orbit of $H_{\GLnn}$
is given by the vanishing of all odd coordinates, hence closed.)
\end{proof}

We would like to apply Lemma~\ref{lem: genstab} to define the inverse transform of $\tranwl$, at least in the case that $\pi$ is tempered.
To that end we first provide a bound of the generalized matrix coefficient $\per{H_{\GLnn}}(\pi(g)W)$.

\subsection{A bound on generalized matrix coefficients} \label{sec: matcoefbnd}

Let $\Pi_0=\Ind_{B_\GLnn}^\GLnn 1$.
We construct an $H_\GLnn$-invariant form $L_0$ on the space of $\Pi_0$ by setting
\[
L_0(\phi)=\int_{\eta^{-1}B_\GLnn\eta\cap H_\GLnn\bs H_\GLnn}\phi(\eta h)\ dh
\]
where $\eta=\diag(\sm1{}11,\ldots,\sm1{}11)$. This is well defined.
In the case $\rkn=1$
$$L_0(\phi)=\int_{F^*}\phi(\sm1{}11 \sm{t}{}{}{1})\,d^*t$$
and the convergence follows since the integrand is $\ll\min(\abs{t},\abs{t}^{-1})^{\frac12}$.
In the general case let $P^\circ=M^\circ\ltimes U^\circ$ be the standard parabolic
subgroup of $\GLnn$ with Levi subgroup $M^\circ=\GL_2\times\cdots\times\GL_2$.
Then $\eta\in M^\circ$, $\eta^{-1} B_\GLnn\eta\cap H_\GLnn=Z_{M^\circ}\ltimes(U^\circ\cap H_\GLnn)$,
$P^\circ\cap H_\GLnn=B_{\GLnn}\cap H_\GLnn$ is a Borel subgroup of $H_\GLnn$ and
\[
\modulus_{\eta^{-1}B_\GLnn\eta\cap H_\GLnn}=\modulus_{P^\circ\cap H_\GLnn}\rest_{\eta^{-1}B_\GLnn\eta\cap H_\GLnn}=
\modulus_{P^\circ}^{\frac12}\rest_{\eta^{-1}B_\GLnn\eta\cap H_\GLnn}=
\modulus_{B_\GLnn}^{\frac12}\rest_{\eta^{-1}B_\GLnn\eta\cap H_\GLnn}.
\]
Thus,
\begin{multline} \label{eq: anthrL0}
L_0(\phi)=\int_{P^\circ\cap H_\GLnn\bs H_\GLnn}
\int_{Z_{M^\circ}\bs M^\circ\cap H_\GLnn}\phi(\eta mh)\modulus_{P^\circ\cap H_\GLnn}(m)^{-1}\ dm\ dh\\=
\int_{P^\circ\cap H_\GLnn\bs H_\GLnn}
\int_{(F^*)^\rkn}\phi(\eta\diag(t_1,1,\dots,t_{\rkn},1)h)\,\textstyle{\prod}_i\abs{t_i}^{-(\rkn+1-2i)}d^*t_i\,dh\\
= \int_{P^\circ\cap H_\GLnn\bs H_\GLnn}
\int_{F^\rkn}\phi(\udr(t_1,\dots,t_\rkn)h)\,\textstyle{\prod}_i\abs{t_i}^{-\frac12} dt_i\,dh.
\end{multline}
where $\udr(t_1,\dots,t_\rkn)=\diag(\sm{1}{}{t_1}{1},\ldots,\sm{1}{}{t_\rkn}{1})$.
Therefore, the convergence reduces to the case $\rkn=1$ considered above.

For $g\in \GLnn$ denote by $\|g\|$ the maximum of the absolute values of the entries of $g$ and $g^{-1}$, and let $\sigma(g)=\max(1,\log_q\|g\|)$.
\index{$\sigma$} We have $\sigma(g)\ge1$, $\sigma(gh)\le\sigma(g)+\sigma(h)$ and $\sigma$ is bi-$K_\GLnn$-invariant.
We define $\sigma_{H_\GLnn}(g)=\sigma(g^{-1}Eg)$ so that $\sigma_{H_\GLnn}\in C^{\smth}(H_\GLnn\bs\GLnn)$. Clearly
$\sigma_{H_\GLnn}(g)\le 2\sigma(g)$. \index{$\sigma_{H_\GLnn}$}

Let $\phi_0$ be the unramified vector in $\Pi_0$ such that $\phi_0(e)=1$.
Let $\Xi_{H_\GLnn}(m)=L_0(\Pi_0(m)\phi_0)$, so that $\Xi_{H_\GLnn}\in C^{\smth}(H_\GLnn\bs\GLnn)$.
\index{$\Xi_{H_\GLnn}$}

\begin{lemma} \label{lem: ellsigma}
For any $b\ge0$
\begin{equation} \label{eq: bnd135}
\int_{H_\GLnn\cap N_{\GLnn}\bs N_{\GLnn}^{\der}}\Xi_{H_\GLnn}(u)\sigma_{H_\GLnn}(u)^b\,du<\infty.
\end{equation}
\end{lemma}

\begin{proof}
We will use the integration formula
\[
\int_{P^\circ\cap H_\GLnn\bs H_\GLnn}f(g)\ dg=\int_{\overline{U^\circ}\cap H_\GLnn} f(u)\ du=\int_{\overline{N}_\GLnn\cap H_\GLnn} f(u)\ du.
\]
Here $\overline{U^\circ}$ and $\overline{N}_\GLnn$ are the images under transpose of $U^{\circ}$ and $N_\GLnn$ respectively. Let
$\overline{N}_\GLnn^{\der}=(N_\GLnn^{\der})^t$.
As $\wnn$ normalizes $H_\GLnn$, by \cite[Theorem~3.2]{1401.0198}\footnote{This can be checked directly at the case at hand.}
we have $L_0(\Pi_0(\wnn)v)=L_0(v)$. We can unwind the left-hand side of \eqref{eq: bnd135} to
\begin{multline}\label{bound: goal}
\int_{H_\GLnn\cap N_{\GLnn}\bs N_{\GLnn}^{\der}}L_0(\Pi_0(\wnn u)\phi_0)\sigma_{H_\GLnn}(\wnn u)^b\,du\\
=\int_{H_\GLnn\cap \overline{N}_\GLnn\bs \overline{N}_\GLnn^{\der}}\int_{\overline{N}_\GLnn\cap H_\GLnn}
\int_{F^\rkn}\phi_0(\udr(t_1,\dots,t_\rkn)\bar u_1\bar u_2)
\sigma_{H_\GLnn}(\bar u_2)^b\,\textstyle{\prod}_i\abs{t_i}^{-\frac12}dt_i\,d\bar u_1\, d\bar u_2
\\=\int_{ \overline{N}_\GLnn^{\der}}
\int_{F^\rkn}\phi_0(\udr(t_1,\dots,t_\rkn)\bar u)
\sigma_{H_\GLnn}(\bar u)^b\,\textstyle{\prod}_i\abs{t_i}^{-\frac12}dt_i\,d\bar u
\\ \ll_b\int_{ \overline{N}_\GLnn^{\der}}
\int_{F^\rkn} \phi_0(\udr(t_1,\dots,t_\rkn)\bar u)
\sigma(\bar u)^b\,\textstyle{\prod}_i\abs{t_i}^{-\frac12}dt_i\,d\bar u
\\ \ll_b\int_{ \overline{N}_\GLnn^{\der}}\int_{F^\rkn}\phi_0(\bar u\udr(t_1,\dots,t_\rkn))
\sigma(\bar u)^b\,\textstyle{\prod}_i\abs{t_i}^{-\frac12}\max(1,-\val (t_i))^b dt_i\,d\bar u.
\end{multline}
It follows from Lemma \ref{lem: mainbnd018} (with $g=e$) that the integral over $\bar u\in\overline{N}_\GLnn^{\der}$ and
$\abs{t_1},\dots,\abs{t_\rkn}\le1$ converges and that
there exists $a\ge0$ such that for any subset $\emptyset\ne I\subset\{1,\dots,\rkn\}$ and integers $n_i>0$, $i\in I$,
the integral over $\bar u\in\overline{N}_\GLnn^{\der}$,
$\val(t_i)=-n_i$, $i\in I$ and $\abs{t_i}\le1$, $i\notin I$ is $\ll_bq^{-\sum_{i\in I}n_i/2}\max_{i\in I}n_i^a$.
The lemma follows.
\end{proof}

We also need a lower bound on $\Xi_{H_\GLnn}$.
Let $H_\GLnn''$ be the centralizer of $\diag(\sm0110,\dots,\sm0110)$, which is conjugate to $H_\GLnn$
by $\kappa=\diag(\sm 111{-1},\dots,\sm 111{-1})\in M^\circ$.
Define $\Xi_{H_\GLnn''}(g)=\Xi_{H_\GLnn}(\kappa g)$.
It follows easily from \eqref{eq: anthrL0} and the relation $\sm{-2}1{}{-1}\sm 1{}11\sm{-1}{}{}1=\sm 111{-1}$
that for a suitable inessential constant $c$
\[
\Xi_{H_\GLnn''}(g)=c\int_{Z_{M^\circ}(U^\circ\cap H_\GLnn'')\bs H_\GLnn''}\phi_0(hg)\ dh
=c\int_{Z_{M^\circ}(N\cap H_\GLnn'')\bs H_\GLnn''}\phi_0(hg)\ dh
\]
where the integral converges.

Let
\[
A_0^+=\{a=\diag(a_1,\dots,a_{2\rkn})\in T_\GLnn:\abs{a_i/a_{i+1}}\le1,\ \ i=1,\dots,2\rkn-1\}.
\]
\begin{lemma} \label{lem: lwrbndXiH}
$\Xi_{H_\GLnn''}(ka)\gg\phi_0(a)$ for all $a\in A_0^+$ and $k\in K_\GLnn$.
\end{lemma}

\begin{proof}
Clearly, for a suitable choice of Haar measure we have
\[
\Xi_{H_\GLnn''}(g)\ge\int_{K_\GLnn\cap H_\GLnn''}\phi_0(k'g)\ dk'.
\]
Here we used the fact that if $H_1$ is any locally compact group, $H_2$ is a closed subgroup of $H_1$
and $K_1$ is a compact open subgroup of $H_1$ then up to normalization of measures we have
\[
\int_{H_2\bs H_1}f(h)\ dh\ge\int_{H_2\bs H_1}f(h)1_{H_2K_1}(h)\ dh=\int_{H_2\cap K_1\bs K_1}f(k)\ dk=\int_{K_1}f(k)\ dk
\]
for any non-negative function $f\in C(H_2\bs H_1,\modulus_{H_2}\modulus_{H_1}^{-1})$.

If $g\in A_0^+$ then $\phi_0(kg)\gg\phi_0(g)$ for all $k\in K_\GLnn$ by \cite[Lemma II.3.2]{MR1989693}.
The Lemma follows.
\end{proof}

\begin{lemma} \label{lem: bndbytmpr}
Let $\pi\in \Irr_{\meta,\temp}\GLnn$ and let $L$ be an $H_\GLnn$-invariant form on the space of $\pi$.
Then there exists $b\ge0$ such that for any $v\in \pi$
\[
L(\pi(g)v)\ll_{L,v}\Xi_{H_{\GLnn}}(g)\sigma_{H_\GLnn}(g)^b, \ \ \ g\in\GLnn.
\]
Consequently,
\begin{equation} \label{eq: conv183}
\int_{H_\GLnn\cap N_{\GLnn}^{\der}\bs N_{\GLnn}^{\der}}\abs{L(\pi(u)v)}\,du<\infty.
\end{equation}
\end{lemma}

\begin{proof}
The second part of the lemma follows from the first one and Lemma \ref{lem: ellsigma}.
In order to prove the first part, it is convenient to work with the centralizer $H_\GLnn'$ of $\wnn$,
which is conjugate to $H_\GLnn''$ by a permutation matrix $w$.
Define $\sigma_{H_\GLnn'}(g)=\sigma(g^{-1}\wnn g)$ and $\Xi_{H_\GLnn'}(g)=\Xi_{H_\GLnn''}(wg)$.
Note that $B_\GLnn$ is a $\theta'$-split Borel group of $\GLnn$ where $\theta'(g)=\wnn g\wnn$.
Thus, by \cite[Lemme 2]{MR2381204}, if $L'$ is an $H'_\GLnn$-invariant form on $\pi$ then
there exists a vector $\hat v\in \dual\pi$ (the dual of $\pi$) such that $L'(\pi(g)v)=\sprod{\pi(g)v}{\hat v}$ for all $g\in A_0^+$.
Since $\pi$ is tempered it follows (cf.~\cite[Lemme II.1.1]{MR1989693}) that there exists $b\ge0$ such that\footnote{In fact,
$b$ can be chosen independently of $\pi$ by \cite{MR946351}}
\[
L'(\pi(g)v)\ll_{L',v}\phi_0(g)\sigma(g)^b, \ \ \ g\in A_0^+.
\]
Let $A_1^+=\{\diag(a_1,\dots,a_\rkn,1,\dots,1):\abs{a_1}\le\dots\le\abs{a_\rkn}\le1\}$. Then
\[
L'(\pi(g)v)\ll_{L',v}\phi_0(g)\sigma_{H_\GLnn'}(g)^b, \ \ \ g\in A_1^+
\]
and by Lemma \ref{lem: lwrbndXiH},
\[
L'(\pi(g)v)\ll_{L',v}\Xi_{H_{\GLnn}'}(g)\sigma_{H_\GLnn'}(g)^b, \ \ \ g\in A_1^+.
\]
Since $\GLnn=H_\GLnn'A_1^+K_{\GLnn}$ (see e.g., \cite[Propotion 3.1]{MR2060496}) it follows that
\[
L'(\pi(g)v)\ll_{L',v}\Xi_{H_{\GLnn}'}(g)\sigma_{H_\GLnn'}(g)^b, \ \ \ g\in\GLnn.
\]
The first part of the lemma follows immediately.
\end{proof}

\subsection{Model transition for $\Irr_{\temp, \meta}\GLnn$}

Let $\pi\in \Irr_{\temp, \meta}\GLnn$ and consider $f(g)=\per{H_{\GLnn}} (\pi(g)W)\in C^{\smth}(H_\GLnn\bs \GLnn)$.
The character $\psi_{N_\GLnn}$ is $(T_\GLnn,N_\GLnn,N_\GLnn^{\der})$-generic and is trivial on
$N_\GLnn\cap H_\GLnn\subset N_\GLnn^{\der}$.
By Lemma~\ref{lem: bndbytmpr}, the conditions of Lemma~\ref{lem: genstab} are satisfied and
from Lemma~\ref{L: stableprop} part \ref{item: reginv} we get a map from $\model^{(H_\GLnn,1)}\pi$ to
$\model^{(N_\GLnn,\psi_{N_\GLnn})}\pi$:
$$f(g)\mapsto \tranlw(f)=\regint_{N_\GLnn\cap H_\GLnn\bs N_\GLnn}f(ng)\psi_{N_\GLnn}^{-1}(n) \,dn.$$
\index{$\tranlw$}
\label{sec: regintNG}

The next proposition is analogous to \cite[Lemma 4.4]{LMao5}.

\begin{proposition} \label{prop: inversionM}
Assume that $\pi\in\Irr_{\meta, \temp}\GLnn$. Then
\begin{equation}\label{eq: invLM}
W(g)=\regint_{N_\GLnn\cap H_\GLnn\bs N_\GLnn}\per{H_{\GLnn}} (\pi(ng)W)\psi_{N_\GLnn}^{-1}(n) \,dn.
\end{equation}
Namely, $\tranlw$ is the inverse transform of $\tranwl$ between models of $\pi$.
\end{proposition}

\begin{proof}
To prove \eqref{eq: invLM} it is enough to consider $g=e$. By Lemma~\ref{L: stableprop}, the right-hand side of \eqref{eq: invLM} is
a  Whittaker functional.
By uniqueness of Whittaker functional, both sides are proportional and hence
we only need to prove the identity for $W$ whose restriction to $\mira_{\GLnn}$ is compactly supported modulo $N_{\GLnn}$.
This would follow from the following functional equation and Lemma~\ref{L: stableprop} part \ref{item: extendint}.
\end{proof}

\begin{lemma}\label{L: feglnn}
If $W\in \cc{N_\GLnn\bs\mira_{\GLnn},\psi_{N_\GLnn}}$ then $\avg^{N_\GLnn}_{\mira_{\GLnn}\cap H_\GLnn} \circ W
\in\cc{\mira_{\GLnn}\cap H_\GLnn\bs \mira_{\GLnn}}$ and $\avg_{N_\GLnn,\psi_{N_\GLnn}^{-1}}^{H_\GLnn}\circ
\avg^{N_\GLnn}_{\mira_{\GLnn}\cap H_\GLnn} \circ W=W$.
\end{lemma}

\begin{proof}
The fact $\avg^{N_\GLnn}_{\mira_{\GLnn}\cap H_\GLnn} \circ W\in\cc{\mira_{\GLnn}\cap H_\GLnn\bs \mira_{\GLnn}}$ follows from Lemma~\ref{L: cuspidal}.

We identify $\mira_{\GLnn}\cap H_\GLnn$ with the product of $\GL_\rkn$ and the mirabolic subgroup of $\GL_\rkn$.
We write $\det_1$ for the character of $\mira_{\GLnn}\cap H_\GLnn$ given by the determinant of the $\GL_\rkn$ factor
and $\det_2$ for the determinant of the other factor.
For convenience we also set $\det_i=\det_{i-2}$ for $i>2$.

For any $i=1,\dots,2\rkn$ let $\mira_i$ (resp.~$N_i$) be the mirabolic subgroup of $\GL_i$
(resp.~the group of upper unitriangular matrices of $\GL_i$).
We embed $\GL_i$ (and its subgroups) in $\GLnn$ via $g\mapsto\sm g{}{}{I_{2\rkn-i}}$.
Note that $N_{j+1}=N_j\rtimes C_j$ where $C_j\simeq F^j$ is the subgroup
of unipotent matrices in $N_{j+1}$ whose upper left $j\times j$ corner is the identity matrix.
Then
\[
\avg_{N_\GLnn,\psi_{N_\GLnn}^{-1}}^{H_\GLnn}\circ
\avg^{N_\GLnn}_{\mira_{\GLnn}\cap H_\GLnn} \circ W=\avg^{H_\GLnn}_{C_1,\psi_{N_\GLnn}^{-1}}\circ \cdots \circ\avg^{H_\GLnn}_{C_{2\rkn-1},\psi_{N_\GLnn}^{-1}}\circ
\avg^{N_\GLnn}_{\mira_{\GLnn}\cap H_\GLnn} \circ W.
\]
We will prove by descending induction on $i$ that
\[
\avg^{H_\GLnn}_{C_i,\psi_{N_\GLnn}^{-1}}\circ \cdots \circ\avg^{H_\GLnn}_{C_{2\rkn-1},\psi_{N_\GLnn}^{-1}}\circ
\avg^{N_\GLnn}_{\mira_{\GLnn}\cap H_\GLnn} \circ W=\avg^{N_i}_{\mira_i\cap H_\GLnn,\prod_{i\le j<2\rkn}\abs{\det_j}^{-1}}\circ W
\]
on $\mira_{\GLnn}$. The case $i=1$ is the required identity \eqref{eq: invLM}.
The base of the induction, the case $i=2\rkn$, is a tautology.
For the induction step we prove that for any character $\chi$ of $\mira_{i+1}\cap H_\GLnn$
\begin{equation}\label{eq: iterateHM}
\avg^{H_\GLnn}_{C_i,\psi_{N_\GLnn}^{-1}}\circ \avg^{N_{i+1}}_{\mira_{i+1}\cap H_\GLnn,\chi}\circ W(g)=
\avg^{N_i}_{\mira_i\cap H_\GLnn,\chi\abs{\det_i}^{-1}}\circ W(g)
\end{equation}
for all $g\in\mira_{\GLnn}$.

As $N_{i+1}\bs \mira_{i+1}\cong N_i\bs \GL_i$ and $\modulus_{\mira_i\cap H_\GLnn}=\abs{\det_i}$ we have
\[
\avg^{N_{i+1}}_{\mira_{i+1}\cap H_\GLnn,\chi}=
\avg^{N_i}_{\GL_i\cap H_\GLnn,\chi}
=\avg^{\mira_i}_{\GL_i\cap H_\GLnn,\chi}\circ
\avg^{N_i}_{\mira_i\cap H_\GLnn,\chi\abs{\det_i}^{-1}}.
\]
Thus we can write the left-hand side of \eqref{eq: iterateHM} explicitly as
\begin{multline*}
\int_{C_i\cap H_{\GLnn}\bs C_i}\big(\int_{\mira_i\cap H_\GLnn\bs \GL_i\cap H_\GLnn}
\big(\int_{N_i\cap H_\GLnn\bs \mira_i\cap H_\GLnn} W(phng)
\\\chi(ph)\abs{{\det}_i(ph)}^{-1}\abs{{\det}_i(h)}\,dp\big)\,dh\big)\psi_{N_\GLnn}^{-1}(n)\,dn.
\end{multline*}
Let $\xi_i=(0,\ldots,0,1)\in F^i$. Then $h\mapsto \xi_i h$ identifies $\mira_i\bs \GL_i$ with $F^i\bs\{0\}$.
For $p\in \mira_i$, $h\in \GL_i$ and $n\in C_i$, we have $W(phng)=\psi(\sprod{\xi_i h}{\hat n}_i)W(phg)$
where $n\mapsto \hat n$ is the standard isomorphism between $C_i$ and $F^i$ and $\sprod{\cdot}{\cdot}_i$ is
the standard pairing on $F^i\times F^i$, which allows us to view $C_i$ as the dual of $F^i$.

The set $\{\xi_i h:h\in \mira_i\cap H_\GLnn\bs \GL_i\cap H_\GLnn\subset \mira_i\bs \GL_i\}$ is the non-zero vectors
in a subspace $V_i$ of $F^i$, and $\sprod{\cdot}{\cdot}_i$ restricts to a non-degenerate pairing between $V_i$ and
$C_i\cap H_\GLnn\bs C_i$. Also $\psi_{N_\GLnn}(n)=\psi(\sprod{\xi_i}{\hat n}_i)$ for $n\in C_i\cap H_\GLnn\bs C_i$.
Therefore the left-hand side of \eqref{eq: iterateHM} is:
\begin{multline*}
\int_{C_i\cap H_\GLnn\bs C_i}\big(\int_{\mira_i\cap H_\GLnn\bs \GL_i\cap H_\GLnn}\big(\int_{N_i\cap H_\GLnn\bs \mira_i\cap H_\GLnn}
W(phg)\psi(\sprod{\xi_i h}{\hat n}_i)
\\ \chi(ph)\abs{{\det}_i(ph)}^{-1}\,dp\big)\abs{{\det}_i(h)}\,dh\big)\psi^{-1}(\sprod{\xi_i}{\hat n}_i)\,dn.
\end{multline*}
By our support condition on $W$, the function
\[
f_g(\xi_i h)= \int_{N_i\cap H_\GLnn\bs \mira_i\cap H_\GLnn} W(phg)\chi(ph)\abs{{\det}_i(ph)}^{-1}\,dp,
\ \  h\in \mira_i\cap H_\GLnn\bs \GL_i\cap H_\GLnn
\]
extends to a Schwartz function on $V_i$. The right-hand side of \eqref{eq: iterateHM} is $f_g(\xi_i)$.
The left-hand side of \eqref{eq: iterateHM} becomes
\[
\int_{C_i\cap H_\GLnn\bs C_i}\big(\int_{V_i}f_g(\eta)\psi(\sprod{\eta}{\hat n}_i)d\eta \big)\psi^{-1}(\sprod{\xi_i}{\hat n}_i)\,dn.
\]
Thus, the relation \eqref{eq: iterateHM} follows by Fourier inversion.
\end{proof}

\subsection{A functional equation}

Let $\mira_{\GLnn}'$ be the second mirabolic subgroup of $\GLnn$ consisting of matrices whose first column is $(1,0,\dots,0)^t$.
Then if $W\in \WhitM(\pi)$ where $\pi\in\Irr_{\gen,\meta}\GLnn$ is unitarizable, the same argument as in Proposition~\ref{P: LM} shows that
\begin{equation}\label{eq: defper1}
\per{H_{\GLnn}}_1 (W):=\int_{N_\GLnn\cap H_\GLnn\bs \mira_{\GLnn}'\cap H_\GLnn} W(p)\, dp
\end{equation}
converges and defines another nontrivial $H_\GLnn$-invariant linear form on $\WhitM(\pi)$.
By uniqueness, $\per{H_{\GLnn}}_1$ is a scalar multiple of $\per{H_{\GLnn}}$.
The following functional equation is in the spirit of \cite[Corollary 7.2]{MR2787356}.
\begin{proposition}\label{prop: sametran}
For any unitarizable $\pi\in\Irr_{\gen,\meta}\GLnn$ and $W\in \WhitM(\pi)$ we have $\per{H_{\GLnn}}_1(W)=\per{H_{\GLnn}}(W)$.
\end{proposition}

\begin{proof}
We first assume that $\pi$ is tempered. Let $L_W'(g)=\per{H_{\GLnn}}_1(\pi(g)W)$.
We show that $\tranLW(L_W')=W$, which will imply that $\per{H_{\GLnn}}_1=\per{H_{\GLnn}}$ on $\WhitM(\pi)$ by
Proposition~\ref{prop: inversionM}.

Let $\dual W(g)=W(g^*)\in \model^{(N_\GLnn,\psi_{N_\GLnn}^{-1})}\dual\pi$. Then $L_W'(g)=\per{H_{\GLnn}}(\dual\pi(g^*)\dual W)=L_{\dual W}(g^*)$.
For $W$ as above, as $H_\GLnn^*=H_\GLnn$, $N_\GLnn^*=N_\GLnn$ and $\psi_{N_\GLnn}(n^*)=\psi_{N_\GLnn}^{-1}(n)$
\begin{multline*}
\int_{H_\GLnn\cap N_\GLnn\bs N_\GLnn} L_W'(n)\psi_{N_\GLnn}^{-1}(n)\ dn=
\int_{H_\GLnn\cap N_\GLnn\bs N_\GLnn} L_W'(n^*)\psi_{N_\GLnn}(n)\ dn
\\=\int_{H_\GLnn\cap N_\GLnn\bs N_\GLnn} L_{\dual W}(n)\psi_{N_\GLnn}(n)\ dn
\end{multline*}
which by Proposition~\ref{prop: inversionM} (applied to $\dual\pi$ and $\psi^{-1}$) is $\dual W(e)=W(e)$.

Using the classification of irreducible generic representation of metaplectic type (\cite{1301.0350}),
a standard argument (see \cite[\S5]{LMao4}) extends the result from the tempered representations to all unitarizable $\pi\in \Irr_{\gen,\meta}\GLnn$.
\end{proof}

\section{Model transition for Langlands quotient: I} \label{sec: LQ1}

\subsection{Notations}
\label{sec: notation}

We keep the notation of \S\ref{sec: notationGL2n} and introduce additional notation as follows.
\begin{itemize}

\item $\alg{G}=\Sp_{2\rkn}=\{g\in\GL_{4\rkn}:\, g^t\sm{}{w_{2\rkn}}{-w_{2\rkn}}{}g=\sm{}{w_{2\rkn}}{-w_{2\rkn}}{}\}$.
\index{$G$}

\item $P=\Levi\ltimes U$ is the Siegel parabolic subgroup of $G$, with its standard Levi decomposition.
$\bar P=P^t$ is the opposite parabolic of $P$, with unipotent radical $\bar U=U^t$. $\mira=\mira_{4\rkn}\cap G$.
\index{$P$, $\Levi$, $U$, $\bar P$, $\bar U$} \index{$\mira$}

\item We use the isomorphism $\levi(g)=\diag(g,g^*)$ to identify $\GLnn$ with $\Levi$.

\item $N$ is the standard maximal unipotent subgroup of $G$ consisting of upper unitriangular matrices;
$T$ is the maximal torus of $G$ consisting of diagonal matrices;
$B=T\ltimes N$ is the standard Borel subgroup of $G$. \index{$B$, $N$, $T$}

\item For a subgroup $X$ of $G$, $X_\Levi$ denotes $X\cap\Levi$. In particular $T=T_\Levi=\levi(T_\GLnn)$
and $N_\Levi=\levi(N_\GLnn)$.

\item $N_\Levi^{\der}=\levi(N_\GLnn^{\der})$. We set $\uder=N_\Levi^{\der}\ltimes U$. \index{$\uder$}

\item Let $K$ be the standard maximal compact subgroup of $G$. \index{$K$}

\item Let $\wgt{\cdot}$ be the character $\wgt{\levi(m)}=\abs{\det m}$ of $M$. \index{$\nu$}
Extend $\wgt{\cdot}$ to a left-$U$ right-$K$ invariant function on $G$ using the Iwasawa decomposition.

\item $H$ is the centralizer of $\levi(E)$ in $G$, isomorphic to $\Sp_\rkn\times\Sp_\rkn$. \index{$H$}

\item Write $H_X=H\cap X$ for any subgroup $X$ of $G$. In particular, $H_M=\levi(H_\GLnn)$
and $H_N\subset\uder$.

\item Let $x\mapsto \startran{x}$ be the twisted transpose map on the space $\Mat_{m,m}$ of $m\times m$ matrices  given by
$\startran{x}=w_m x^t w_m$. \index{$\startran{x}$}
Let $\symspace_m=\{x\in\Mat_{m,m}:\startran{x}=x\}$. \index{$\symspace_m$}

\item Define $\toU:\symspace_{2\rkn}\rightarrow U$ to be the isomorphism given by $\toU(x)=\sm{I_{2\rkn}}{x}{}{I_{2\rkn}}$.
Similarly $\toUbar(x)=\sm{I_{2\rkn}}{}x{I_{2\rkn}}$ is the isomorphism from $\symspace_{2\rkn}$ to $\bar U$.
\index{$\toU$, $\toUbar$}

\item $\wnnM=\levi(\wnn)$; \index{$\wnnM$}
$w_U=\sm{}{I_{2\rkn}}{-I_{2\rkn}}{}\in G$ represents the longest $\Levi$-reduced Weyl element of $G$.
\index{$w_U$}

\item Let $\one_{i,j}\in \Mat_{m,m}$ be the matrix with one at the $(i,j)$-entry and zeros elsewhere.
Let $\one^{\symspace_m}_{i,j}=\one_{i,j}+\startran{\one_{i,j}}\in \symspace_m$ when $i+j\not=m+1$ and
$\one^{\symspace_m}_{i,m+1-i}=\one_{i,m+1-i}$.

\item $N^\GLnn_{i,j}\subset \GLnn$ ($i\not=j$) is the one-parameter group $\{I_{2\rkn}+x\one_{i,j}:x\in F\}$. \index{$N^\GLnn_{i,j}$}
$N_{\alpha_i}$ is the simple root group $N^\GLnn_{i,i+1}$ and $N_{-\alpha_i}=N^\GLnn_{i+1,i}$. $N^{\Levi}_{\alpha_i}=\levi(N_{\alpha_i})$.

\item $\newchar$ is the character on $\bar U$ given by $\newchar(\bar u)=\psi(\bar u_{2\rkn+1,1})$. \index{$\newchar$}
$\psi_N$ is the \emph{degenerate} character on $N$ given by \index{$\psi_N$}
$\psi_N(\levi(n)u)=\psi_{N_\GLnn}(n)$ for any $n\in N_\GLnn$ and $u\in U$.
$\psi_{N_\Levi}$ is the restriction of $\psi_N$ to $\Levi$, i.e., $\psi_{N_\Levi}\circ\levi=\psi_{N_\GLnn}$.

\item $[x,y]=xyx^{-1}y^{-1}$ denotes the commutator of $x$ and $y$.

\item $\langle A_1,A_2,\ldots,A_k\rangle$ is the group generated by $A_1,\ldots,A_k\subset G$.

\item The convention of Haar measures will be as in \S\ref{sec: notationGL2n}.
Namely, we will use the lattice of $4\rkn\times4\rkn$-matrices integral matrices
to define a gauge form (up to an element of $\OO^*$) for any algebraic subgroup of $\GL_{4\rkn}$
(and in particular, any algebraic subgroup of $\alg{G}$) defined over $F$.
\end{itemize}

\subsection{}
Let $\pi\in\Irr\GLnn$. We consider $\pi$ as a representation of $\Levi$ via $\levi$.
For $s\in \C$, let $I(\pi,s)=\Ind_P^G \pi\nu^s$. \index{$I(\pi,s)$}

Note that $H_P$ is a parabolic subgroup of $H$ with Levi decomposition
$H_P=H_\Levi\ltimes H_U$. The opposite parabolic with respect to $H_\Levi$ is
$H_{\bar P}=H_\Levi\ltimes H_{\bar U}$.
We have
\begin{equation} \label{eq: modulusPH}
\modulus_P =\nu^{2\rkn+1}\ \text{ on }\Levi;
\ \ \modulus_{H_P}=\nu^{\rkn+1}\ \text{ on }H_{\Levi}.
\end{equation}

As was observed in the course of the proof of \cite[\S3.3, Theorem 2]{MR1671452},
if $\pi\in\Irr_{\meta}\GLnn$ then $I(\pi,\frac12)$ admits a non-trivial $H$-invariant functional.
More precisely, if $\ell$ is an $H_{\GLnn}$-invariant functional on $\pi$ then
\begin{equation} \label{eq: Hinv}
\varphi\rightarrow\int_{H_P\bs H}\ell(\varphi(h))\ dh\text{ defines an $H$-invariant functional on }I(\pi,\frac12).
\end{equation}
This is well defined since by \eqref{eq: modulusPH}
$\modulus_{H_P}(m)=\modulus_P^{\frac12}(m)\wgt{m}^{\frac12}$ for all $m\in H_\Levi$.

We will need the following complementary information.

\begin{proposition}\label{prop: distequi}
Let $\pi\in \Irr_{\temp}\GLnn$.
Then $I(\pi,\frac12)$ is $H$-distinguished if and only if $\pi$ is $H_\Levi$-distinguished.
In this case $I(\pi,\frac12)$ has a unique $H$-invariant form up to a constant multiple.
\end{proposition}

\begin{proof}
We argue using the geometric Lemma of Bernstein--Zelevinsky.
The orbits of $H$ on $P\bs G$, ordered by dimension, are indexed by integers $d=0,\dots,\rkn$.
Let $\eta_d$ be representatives of the orbits and let $H_{\eta_d}=H\cap\eta_d^{-1}P\eta_d$
be the stabilizer. Let $P_{\eta_d}=\eta_dH_{\eta_d}\eta_d^{-1}=P\cap\eta_dH\eta_d^{-1}$.
Correspondingly, the representation $I(\pi,\frac12)$ admits a filtration
$0=W_0\subset W_1\subset\dots\subset W_{\rkn+1}=I(\pi,\frac12)$ by $H$-invariant subspaces such that
\[
W_{d+1}/W_d=\ind_{H_{\eta_d}}^H(\pi\modulus_P^{\frac12}\nu^{\frac12}\rest_{P_{\eta_d}})^{\eta_d}.
\]
(Here $\ind$ denotes unnormalized compact induction and the superscript denotes conjugation by $\eta_d$.)
Let $Q_d=L_d\rtimes V_d$ be the standard parabolic subgroup of $G$ contained in $P$ such that
$L_d=\{\levi(\diag(a,b)):a\in\GL_{2d},b\in\GL_{2\rkn-2d}\}$.
As described in \cite[\S1.3]{MR1740991} we can choose $\eta_d$ such that $P_{\eta_d}=M_d\rtimes N_d$ where
\begin{enumerate}
\item $M_d=\{\levi(\diag(g_1,g_2,g_3)):g_1,g_2\in\GL_d,g_3\in\Sp_{\rkn-d}\}\subset L_d$.
\item The projection of $N_d$ to $\Levi$ is $V_d\cap \Levi$, the unipotent radical of $Q_d\cap \Levi$
(which is the parabolic subgroup of $\Levi$ type $(2d,2\rkn-2d)$).
\item $N_d\cap U=\toU(\{\left(\begin{smallmatrix}&&X\\&Y&\\0&&\end{smallmatrix}\right):X,Y\in\symspace_d\})$.
\end{enumerate}
It follows by a simple computation that $\modulus_{P_{\eta_d}}$ is the restriction of $\modulus_{Q_d}^{\frac12}\nu^{\frac12}$.
By a standard argument using Frobenius reciprocity (cf.~\cite[\S6]{MR2930996}) we conclude that
\[
\Hom_H(W_{d+1}/W_d,\C)=\Hom_{P_{\eta_d}}(\pi\modulus_P^{\frac12}\nu^{\frac12}\modulus_{P_{\eta_d}}^{-1},\C)=\Hom_{M_d}(J_{Q_d\cap \Levi}(\pi),\C).
\]
Here $J_{Q_d\cap \Levi}(\pi)$ is the Jacquet module of $\pi$ with respect to $Q_d\cap \Levi$.
We show below that $\Hom_{M_d}(J_{Q_d\cap \Levi}(\pi),\C)=0$ when $d<\rkn$.
This would imply that
\[
\Hom_H(I(\pi,\frac12),\C)=\Hom_H(W_{\rkn+1}/W_{\rkn},\C)=\Hom_{H_\Levi}(\pi,\C)
\]
from which the Proposition follows.

In fact, we will show that (for $d<\rkn$) $\Hom_{M_d}(\sigma_1\otimes\sigma_2,\C)=0$ for any irreducible subquotient $\sigma_1\otimes\sigma_2$ of $J_{Q_d\cap \Levi}(\pi)$.
Denote by $\omega_\rho$ the central character of an irreducible representation $\rho$.
As $\pi\in \Irr_{\temp}\GLnn$, by \cite{MR584084} $\pi=\delta_1\times\dots\times\delta_k$ where for
$i=1,\dots,k$, $\delta_i\in \Irr \GL_{d_i}$ is the unique irreducible quotient $\delta([-m_i,m_i]_{\rho_i})$ of
$\rho_i \abs{\cdot}^{-m_i}\times \rho_i \abs{\cdot}^{-m_i+1}\times\ldots\times \rho_i \abs{\cdot}^{m_i}$.
Here we use $\times$ to denote parabolic induction; $\rho_i$ are irreducible supercuspidal representations of $\GL_{d_i/(2m_i+1)}$
with $\abs{\omega_{\rho_i}}=1$; $m_i\in \frac12\Z_{\ge0}$.
Consider an irreducible subquotient $\sigma_1\otimes\sigma_2$ of $J_{Q_d\cap \Levi}(\pi)$. Then by the geometric lemma and the description
of the Jacquet modules of square-integrable representations,
there exist $l_i\in\{-m_i-1,-m_i,\ldots,m_i\}$, $i=1,\dots,k$ such that
$\sigma_1$ is a subquotient of $\delta_1'\times\dots\times\delta_k'$ and $\sigma_2$ is a subquotient of $\delta_1''\times\dots\times\delta_k''$
where $\delta_i'=\delta([l_i+1,m_i]_{\rho_i})$ and $\delta_i''=\delta([-m_i,l_i]_{\rho_i})$.
If $-m_i\le l_i< m_i$ for some $i$ then
$d>0$ and $\abs{\omega_{\sigma_1}}=\abs{\det\cdot}^\alpha$ with $\alpha>0$, in which case $\sigma_1$ is not of metaplectic type.
Otherwise $\sigma_2$ has the form $\delta_{i_1}\times\dots\times\delta_{i_l}$ with $1\le i_1<\dots<i_l\le k$.
In particular $\sigma_2$ is generic and hence does not admit a non-trivial functional invariant under $\Sp_{\rkn-d}$ by \cite{MR1078382}.
In both cases we have $\Hom_{M_d}(\sigma_1\otimes\sigma_2,\C)=0$ as required.
\end{proof}

\begin{remark}
It is conceivable that in general, any $\pi\in\Irr G$ admits at most one $H$-invariant functional up to a scalar.
However, this is an open problem except for $\rkn\le2$ (\cite{MR2678856}).
\end{remark}

\subsection{An integral transform}
For any $f\in C^\infty(G)$ and $s\in\C$ define $f_s(g)=f(g)\wgt{g}^s$, $g\in G$.

Let $\pi\in\Irr_{\gen}\GLnn$.
Let $\Ind(\WhitML(\pi))$ \index{$\Ind(\WhitML(\pi))$} be the space of $G$-smooth left $U$-invariant functions $W:G\rightarrow\C$ such that
for all $g\in G$, the function $m\mapsto\modulus_P(m)^{-\frac12}W(mg)$ on $\Levi$ belongs to $\WhitML(\pi)$.
For any $s\in\C$ we have a representation $\Ind(\WhitML(\pi),s)$ on the space $\Ind(\WhitML(\pi))$ given by
$(I(s,g)W)_s(x)=W_s(xg)$, $x,g\in G$. Then the representation $\Ind(\WhitML(\pi),s)$ is isomorphic to $I(\pi,s)$.
It will be convenient to set
\[
[W]_{s,g}(m)=\modulus_P(\levi(m))^{-\frac12}W(\levi(m)g)\wgt{g}^s, \ \ g\in G,\, m\in\GLnn,\, s\in\C
\]
so that $[W]_{s,g}\in\WhitM(\pi)$.
Thus, $[I(s,g_1)W]_{s,g_2}=[W]_{s,g_2g_1}$ and for $m\in \GLnn$, $u\in U$ and $g\in G$ we have
\begin{equation} \label{eq: W'trans}
[W]_{s,\levi(m)ug}=\modulus_P(\levi(m))^{\frac12}\abs{\det m}^s\pi(m)[W]_{s,g}.
\end{equation}

We can explicate \eqref{eq: Hinv} as follows.

\begin{lemma}\label{lem: WL}
Let $\pi\in \Irr_{\meta,\gen}\GLnn$, considered also as a representation of $\Levi$ through $\levi$.
Assume that $\pi$ is unitarizable.
Then a nontrivial $H$-invariant linear form on $\Ind(\WhitML(\pi),\frac12)$ is given by
\[
\per{H}(W):=\int_{H_P\bs H}\per{H_{\GLnn}}([W]_{\frac12,h})\ dh=
\int_{H_{\bar U}}\per{H_{\GLnn}}([W]_{\frac12,u})\ du
\]
where $\per{H_{\GLnn}}$ is defined in Proposition~\ref{P: LM}.
\index{$\per{H}$}
\end{lemma}

We define the integral transform $\tranWL:  \Ind(\WhitML(\pi),\frac12)\rightarrow C^{\smth}(H\bs G)$ by
\begin{equation}\label{eq: defL}
\tranWL(W)(g)=\per{H}(I(\frac12,g)W)=\int_{H_P\bs H}\per{H_{\GLnn}}([W]_{\frac12,hg})\ dh.
\end{equation}
\index{$\tranWL$}
We will show that if $\pi$ is tempered then $\tranWL$ factors through the Langlands quotient
$\LQ{\pi}$ of $I(\pi, \frac12)$, and hence gives rise to a model $\model^{(H,1)}$ of the Langlands quotient.
We will also define an integral transform $\tranLW$ from $\model^{(H,1)}\LQ{\pi}$ to $\model^{(N,\psi_N)}\LQ{\pi}$ as well as its inverse.
The fact that there is a unique model $\model^{(N,\psi_N)}\LQ{\pi}$ is proved in Appendix~\ref{sec: Tadic}.

\subsection{A class of functions on $H\bs G$} \label{sec: defclass}

For any $K_0\in\csgr(G)$ we denote by $\tclass(H\bs G)^{K_0}$ the space
of right-$K_0$-invariant functions $L$ on $H\bs G$ such that for all $g\in G$ the integral
\begin{equation}\label{eq: Ltclass}
\int_{H_N\bs \uder}L(ug)\ du=\int_{N_\Levi\cap H\bs N_\Levi^{\der}}\int_{H_U\bs U}L(ung)\ du\ dn
\end{equation}
converges absolutely.
We endow $\tclass(H\bs G)^{K_0}$ with the topology given by the seminorms
\[
\int_{H_N\bs \uder}\abs{L(ug)}\ du,\ \ g\in G.
\]
We define
\[
\tclass(H\bs G)=\cup_{K_0\in\csgr(G)}\tclass(H\bs G)^{K_0}
\]
with the inductive limit topology. \index{$\tclass(H\bs G)^{K_0}$, $\tclass(H\bs G)$, $\tclass^{\pm1}(H\bs G)$}

Note that $N_G(H)=H\cup\wnnM H$. Define $\tclass^{\pm1}(H\bs G)$
as the subspace of $\tclass(H\bs G)$ consisting of $f\in \tclass(H\bs G)$ such that $f(\wnnM g)=\pm f(g)$.

The character $\psi_N$ is $(T,N,\uder)$-generic and is trivial on $H_N$.
Thus, by Lemma \ref{lem: genstab} we can define for $L\in\tclass(H\bs G)$
\[
\regint_{H_N\bs N}L(ng)\psi_N^{-1}(n)\ dn=\avg^{H,reg}_{N,\psi_{N}^{-1}}\circ L(g)=
\int^{\st,T}_{N_\Levi^{\der}\bs N_\Levi}\big(\int_{H_N\bs \uder}L(ung)\,du\big)\psi_N^{-1}(n)\,dn.
\]
Note that
\begin{equation} \label{eq: HNNnHNM}
\regint_{H_N\bs N}L(ng)\psi_N^{-1}(n)\ dn=\regint_{H_{N_\Levi}\bs N_\Levi}\big(\int_{H_U\bs U}L(ung)\,du\big)\psi_N^{-1}(n)\,dn.
\end{equation}
when the left-hand side is well defined, where $\regint_{H_{N_\Levi}\bs N_\Levi}$ on the right-hand side was defined in \S\ref{sec: regintNG}.

\begin{lemma}\label{L: stablemodel}
\begin{enumerate}
\item \label{mod: dense}
The space $C_c(H\bs G)^{K_0}$ of compactly supported right $K_0$-invariant functions on $H\bs G$ is dense in $\tclass(H\bs G)^{K_0}$.
Hence $\cc{H\bs G}=\cup_{K_0}C_c(H\bs G)^{K_0}$ is dense in $\tclass(H\bs G)$.
\item The expression
\[
\regint_{H_N\bs N}L(ng)\psi_N^{-1}(n)\ dn
\]
is a continuous functional on $\tclass(H\bs G)^{K_0}$.
\item The function $g\mapsto\regint_{H_N\bs N}L(ng)\psi_N^{-1}(n)\ dn$ belongs to $C(Z_{\Levi} N\bs G,\modulus_P^{\frac12}\nu^{-1/2}\psi_N)^{K_0}$.
\footnote{Here and elsewhere we extend $\psi_N$ to $Z_{\Levi}N$, trivially on $Z_{\Levi}$.}
\end{enumerate}
\end{lemma}

\begin{proof}
The first part is clear.
The second part follows from \eqref{eq: stableconv} (applied with $\phi=1_{K_0\cap T}$).
The last part follows from Lemma~\ref{L: stableprop}.
\end{proof}

\subsection{Model transition -- inflated model}

We go back to the notation of \S\ref{sec: matcoefbnd}.
Let $\Xi_H$ be the function \index{$\Xi_H$}
\[
\Xi_H(umk)=\modulus_P(m)^{\frac12}\Xi_{H_\GLnn}(m),\ \ u\in U, m\in M, k\in K.
\]
Thus $\Xi_H(g)=L_0(\varphi(g))$ where $\varphi$ is the unramified section of $\Ind_P^G\Pi_0$.
As before let $\sigma(g)=\max(1,\log_q\|g_{i,j}\|)$, $g\in G$.

\begin{lemma} \label{lem: conv192}
For any $a\ge0$ and $s>0$
\[
\int_{H_\GLnn\cap N_{\GLnn}^{\der}\bs N_{\GLnn}^{\der}}\int_U(\Xi_H)_s(nw_Uu)(\sigma_{H_{\GLnn}}(n)+\sigma(u))^a\ dn\ du<\infty.
\]
\end{lemma}

\begin{proof}
Note that the integrand is non-negative.
By the Gindikin--Karpelevich formula, the integral $\int_U(\Xi_H)_s(mw_Uu)\ du$ converges for $s>0$ and is equal to $c(s)(\Xi_H)_{-s}(m)$
where $c(s)=\frac{\zeta_F(s)^{2\rkn}\zeta_F(2s)^{\rkn(2\rkn-1)}}{\zeta_F(s+1)^{2\rkn}\zeta_F(2s+1)^{\rkn(2\rkn-1)}}$
and as usual $\zeta_F(s)=(1-q_F^{-s})^{-1}$.
On the other hand, it follows immediately from \cite[Lemme II.3.4]{MR1989693} that for any $0<s'<s$ and $a\ge0$ we have
$(\Xi_H)_s(mw_Uu)\sigma(u)^a\ll_{s,s',a}\wgt{m}^{s-s'}(\Xi_H)_{s'}(mw_Uu)$.
Therefore the lemma follows from Lemma \ref{lem: ellsigma}.
\end{proof}

Let $\pi\in\Irr_{\gen}\GLnn$. Assume that $\pi$ is self-dual. We recall the intertwining operators
\[
M(s):\Ind(\WhitML(\pi),s)\rightarrow \Ind(\WhitML(\pi),-s)
\]
\index{$M(s)$} given by
\[
(M(s)W)_{-s}(g)=\int_UW_s(\levi(E)w_Uug)\ du
\]
in the range of convergence. (The element $\levi(E)$ is introduced in order to preserve the character $\psi_{N_\Levi}$.)
Thus,
\begin{equation} \label{eq: interell}
\ell([M(s)W]_{-s,g})=\int_U\ell([W]_{s,\levi(E)w_Uug})\ du
\end{equation}
for any $\ell$ in the smooth dual of $\pi$. In fact, \eqref{eq: interell} holds
for any functional $\ell$ provided that the right-hand side is absolutely convergent.
Indeed, both sides are unchanged if we replace $\ell$ by $\vol(K_0)^{-1}\int_{K_0}\ell\circ\pi(k)\ dk$
for sufficiently small $K_0\in\csgr(\GLnn)$.

We also set
\[
M^*W(g)=(M(\frac12)W)_{-\frac12}(g)=\int_UW_{\frac12}(\levi(E)w_Uug)\ du
\]
\index{$M^*W$} so that if $\pi$ is tempered then
$W\mapsto M^*W$ defines a $G$-intertwining map from $\Ind(\WhitML(\pi),\frac12)$ to $\model^{(N,\psi_N)}\LQ{\pi}$.

We denote by $\fel$ the sign such that $\ell\circ\pi(\wnn)=\fel\ell$ \index{$\fel$}
where $\ell$ is a non-trivial $H_\GLnn$-invariant linear form on $\pi$. (By uniqueness, $\fel$ exists
and does not depend on the choice of $\ell$.)
It is known that for $\pi\in \Irr_{\gen,\meta}\GLnn$ we have $\fel=\epsilon(\frac12,\pi,\psi)$
(the `standard' $\epsilon$-factor of $\pi$, which of course does not depend on $\psi$ since $\pi$
has a trivial central character) -- see \cite[Theorem 3.2]{1401.0198}. However, we will not use this fact.

\begin{proposition}\label{P: L}
Let $\pi\in \Irr_{\meta,\temp}\GLnn$ considered also as a representation of $\Levi$ through $\levi$.
Then $\tranWL(\Ind(\WhitML(\pi),\frac12))\subset \tclass^{\fel}(H\bs G)$ and $\tranLW\circ\tranWL=\fel M^*$
where
$$\tranLW(L)(g):=\regint_{H_N\bs N} L(ng)\psi_N^{-1}(n)\ dn.$$
\index{$\tranLW$}
\end{proposition}

\begin{proof}

Let $L_W=\tranWL(W)$.
We show that $L_W\in\tclass^{\fel}(H\bs G)$.
First, by \eqref{eq: defL} and \eqref{eq: W'trans} for any $w\in H\wnnM$ we have
\begin{multline} \label{eq: felpi'}
L_W(wg)=L_W( \wnnM g)=\int_{H_{\bar U}}\per{H_{\GLnn}}([W]_{\frac12,\bar u \wnnM g})\,d\bar u
\\=\int_{H_{\bar U}}\per{H_{\GLnn}}(\pi(\wnn)([W]_{\frac12,\bar u g}))\,d\bar u
=\fel\int_{H_{\bar U}}\per{H_{\GLnn}}([W]_{\frac12,\bar u g})\,d\bar u=\fel L_W(g).
\end{multline}
Next we show the convergence of  \eqref{eq: Ltclass} for $L=L_W$. We have
\begin{equation} \label{eq: bndLW2}
\int_{H_U\bs U}\abs{L_W(ug)}\ du=
\int_{H_U\bs U}\abs{\int_{H_U}\per{H_{\GLnn}}([W]_{\frac12,w_Uvug})\ dv}\ du
\le\int_{U}\abs{\per{H_{\GLnn}}([W]_{\frac12,w_Uug})}\ du.
\end{equation}
Therefore
\begin{multline*}
\int_{N_\Levi\cap H\bs N_\Levi^{\der}}\int_{H_U\bs U}\abs{L_W(un)}\ du\ dn\le
\int_{N_\Levi\cap H\bs N_\Levi^{\der}}\int_U\abs{\per{H_{\GLnn}}([W]_{\frac12,w_Uun})}\ du\ dn\\=
\int_U\int_{N_\Levi\cap H\bs N_\Levi^{\der}}\abs{\per{H_{\GLnn}}(\pi(n)[W]_{\frac12,w_Uu})}\ dn\ du.
\end{multline*}
We write $w_Uu=u_1m_1k_1$ where $u_1\in U$, $m_1\in M$ and $k_1\in K$.
Since $[W]_{\frac12,k}$, $k\in K$ range over a finite set of vectors in $\WhitM(\pi)$
we conclude from Lemma \ref{lem: bndbytmpr} that there exists $a\ge0$ such that for any $m\in M$
with $\wgt{m}=1$
\begin{multline} \label{eq: bndper2}
\abs{\per{H_{\GLnn}}(\pi(m)[W]_{\frac12,w_Uu})}=
\modulus_P(m_1)^{\frac12}\wgt{m_1}^{\frac12}\abs{\per{H_{\GLnn}}(\pi(mm_1)[W]_{\frac12,k_1})}\\\ll_W
\modulus_P(m_1)^{\frac12}\wgt{m_1}^{\frac12}\Xi_{H_\GLnn}(mm_1)\sigma_{H_\GLnn}(mm_1)^a\\\le
(\Xi_H)_\frac12(mw_Uu)(\sigma_{H_\GLnn}(m)+2\sigma(m_1))^a\ll
(\Xi_H)_\frac12(mw_Uu)(\sigma_{H_\GLnn}(m)+\sigma(u))^a.
\end{multline}
The convergence of
\[
\int_{H_N\bs \uder}\abs{L_W(n)}\ dn
\]
therefore follows from Lemma \ref{lem: conv192}. We get $L_W\in\tclass^{\fel}(H\bs G)$.

Similarly, since $w_U\in H\wnnM$ we have
\begin{multline*}
\fel\int_{H_U\bs U}L_W(ug)\ du=
\int_{H_U\bs U}\big(\int_{H_U}\per{H_{\GLnn}}([W]_{\frac12,w_Uvug})\ dv\big)\ du
\\=\int_U \per{H_{\GLnn}}([W]_{\frac12,w_Uug})\ du=\int_U \per{H_{\GLnn}}([W]_{\frac12,\levi(E)w_Uug})\ du=
\per{H_{\GLnn}}  ([M(\frac12)W]_{-\frac12,g})
\end{multline*}
where the last equality follows from \eqref{eq: interell}.
It follows from \eqref{eq: HNNnHNM} that $\tranLW(L_W)(g)$ equals
\begin{multline*}
\regint_{H_N\bs N} L_W(ng)\psi_N^{-1}(n)\ dn=
\regint_{N_\Levi\cap H\bs N_\Levi}\big(\int_{H_U\bs U}  L_W(ung)\ du\big)\psi_N^{-1}(n)\,dn\\=
\fel\regint_{N_\GLnn\cap H_{\GLnn}\bs N_\GLnn}\per{H_{\GLnn}}
([M(\frac12)W]_{-\frac12,\levi(n)g})\psi_{N_\GLnn}^{-1}(n)\,dn\\=
\fel\regint_{N_\GLnn\cap H_{\GLnn}\bs N_\GLnn}\per{H_{\GLnn}}
(\pi(n)[M(\frac12)W]_{-\frac12,g})\psi_{N_\GLnn}^{-1}(n)\,dn=\fel[M(\frac12)W]_{-\frac12,g}(e)=\fel M^*W(g)
\end{multline*}
where the regularized integral $\regint_{N_\Levi\cap H\bs N_\Levi}$ is as defined in \S\ref{sec: regintNG} and in the last line we used
the relation \eqref{eq: invLM}.
\end{proof}

\section{Inverse transform} \label{sec: invtrans}

Next we show that $\per{H}$ defined above (and hence $\tranWL$) factors through the Langlands quotient $\LQ{\pi}$.
Indeed we will define an inverse transform of $\tranLW$.

Let $\mira=\mira_{4\rkn}\cap G$. \index{$\mira$}

\begin{theorem}\label{thm: MWL}
Let $\pi\in\Irr_{\meta,\temp}\GLnn$. Then
\begin{enumerate}
\item The integral
$$\tranMWL(\tilde W):=\int_{H_N\bs H_\mira}\tilde W(h\cdot)\ dh$$
converges for $\tilde W\in \model^{(N,\psi_N)}\LQ{\pi}$.
\item $\LQ{\pi}$ is $H$-distinguished with a unique $H$-invariant linear form up to scalar multiple.
\item $\tranLW$ is a model transition from $\model^{(H,1)}\LQ{\pi}$ to $\model^{(N,\psi_N)}\LQ{\pi}$ with the inverse being $\tranMWL$.
\end{enumerate}
\end{theorem}

\begin{proof}
We first show convergence. Let $\tilde W=\W$. Since $H_\mira$ is unimodular and $\modulus_{H_{B\cap\mira}}=\modulus_{H_B}\rest_{H_{B\cap\mira}}$,
we can write $\int_{H_N\bs H_\mira}\tilde W(h)\ dh$ as
\[
\int_{H_{B\cap\mira}\bs H_\mira}\int_{H_N\bs H_{B\cap\mira}}\tilde W(bh)\modulus_{H_B}(b)^{-1}\ db\ dh=
\int_{K\cap H_\mira}\int_{T\cap\mira}\tilde W(tk)\modulus_{H_B}(t)^{-1}\ dt\ dk
\]
for a suitable Haar measure of $K\cap H_\mira$. We need to show the convergence of the inner integral.
Write $t\in T\cap H_\mira$ as $\levi(\diag(t_1,t_2,\ldots,t_{2\rkn}))$ with $t_1=1$. Then
$\modulus_{H_B}(t)=\prod_{i=1}^\rkn \abs{t_{2i-1}t_{2i}}^{2(\rkn+1-i)}$. On the other hand, for any $s>0$ we have
$$\tilde W(t)\ll_s \modulus_B^{\frac12}(t)\prod_{i=1}^{2\rkn}\abs{t_i}^{-\frac12+s}=\prod_{i=1}^{2\rkn}\abs{t_i}^{\frac12+s+2\rkn-i}
$$
and $\tilde W(t)$ is supported on $1\ll\abs{t_2}\ll\ldots\ll\abs{t_{2\rkn}}$.
Thus the convergence follows from \footnote{Indeed the convergence holds when $\pi$ is unitarizable since we can take $s<\frac12$
(\cite[Lemma 2.1]{LMao4}).}
\[
\int_{1\ll\abs{t_2}\ll\ldots\ll\abs{t_{2\rkn}}}\textstyle{\prod}_{i=1}^\rkn\abs{t_{2i}}^{-\frac32+s}
\textstyle{\prod}_{i=2}^\rkn\abs{t_{2i-1}}^{-\frac12+s}\ \textstyle{\prod}_{i=2}^{2\rkn} d^*t_i<\infty.
\]

We will show that for $W\in \Ind(\WhitML(\pi),\frac12)$:
\begin{equation}\label{eq: goalMWL}
\int_{H_N\bs H_\mira}\big(\regint_{H_N\bs N} L_W(nh)\psi_N^{-1}(n)\ dn\big)\ dh=L_W(e).
\end{equation}
This would imply the second and third part of the theorem.
Indeed, by Proposition \ref{P: L} and the equation above
\[
\tranMWL\circ M(\frac12)=\fel\tranMWL\circ \tranLW\circ\tranWL=\fel \tranWL
\]
on $\Ind(\WhitML(\pi),\frac12)$. In particular $\tranWL$ factors through $M(\frac12)$. This shows $\LQ{\pi}$ is $H$-distinguished.
By Proposition~\ref{prop: distequi}, the $H$-invariant linear form on $\LQ{\pi}$ has dimension one.
Thus we have a model $\model^{(H,1)}\LQ{\pi}$. The proposition follows.

In order to show \eqref{eq: goalMWL} let $\tilde T_k$ be the image of the co-character $t\mapsto\levi(\diag(I_{2\rkn-k},t,I_{k-1}))$,
$k=1,\dots,2\rkn$ and let $V^k$ (resp., $\bar V^k$)
be the unipotent group in $N$ (resp., $\bar N$), consisting of elements of the form $\diag(I_{2\rkn-k},\bar v,I_{2\rkn-k})$
where the middle $2(k-1)\times 2(k-1)$ block of $\bar v$ is the identity matrix. (It is a Heisenberg group of dimension $2k-1$.)
Then the left-hand side of \eqref{eq: goalMWL} is
\begin{equation}\label{eq: goalMWLiter}
\avg_{\tilde T_{2\rkn-1}}\circ\avg_{\bar V^{2\rkn-1}\cap H} \circ\dots\avg_{\tilde T_k,\nu^{\rkn-1-[\frac k2]}}\circ\avg_{\bar V^k\cap H}
\circ\dots\circ\avg_{\tilde T_1,\nu^{\rkn-1}}\circ\avg_{\bar V^1\cap H}\circ\avg_{N,\psi_N^{-1}}^{reg,H} \circ L_W(e).
\end{equation}

We introduce some auxiliary integrals. For $k=1,\dots,2\rkn$, let $N_k$ be the unipotent radical of the standard parabolic subgroup
with Levi part $\GL_1^{2\rkn-k}\times\Sp_k$
and let $N_k^\circ=N_k\cap\uder$.
If $L\in \tclass(H\bs G)$ then $L\in L^1((H\cap N_k^\circ)\bs N_k^\circ)$ for $k=1,\dots,\leq 2\rkn$ by Remark~\ref{rem: partint}.
Since $\psi_N\rest_{N_k}$ is $(T,N_k,N_K^\circ)$-generic we can define using Lemma~\ref{lem: genstab} the regularized integral
$$J_k(L;g)=\regint_{(H\cap N_k)\bs N_k} L(ng)\psi_N^{-1}(n)\ dn.$$
This is a continuous functional on $\tclass(H\bs G)$.
Similarly for $k=1,\dots,2\rkn-1$ let $N^{\Levi}_{\alpha_{2\rkn-k}}$ be the simple root group consisting of $\levi(I_{2\rkn}+x\one_{2\rkn-k,2\rkn-k+1})$.
Define a continuous functional on $\tclass(H\bs G)$:
$$J'_k(L;g)=\regint_{N^{\Levi}_{\alpha_{2\rkn-k}}}J_{k+1}(L;ng)\psi_N^{-1}(n)\ dn.$$
Here we use Lemma~\ref{lem: genstab} with the data $\T=\tilde T_k$, $U_0=N^{\Levi}_{\alpha_{2\rkn-k}}$, $U_1=U_2=1$,
$\psi_{U_0}=\psi_N\rest_{U_0}$.

As $J_1(L_W;\cdot)=\avg_{N,\psi_N^{-1}}^{reg,H} \circ L_W$ and $J_{2\rkn}(L_W;\cdot)=L_W$, \eqref{eq: goalMWL} follows from the following proposition.
\end{proof}

\begin{proposition}\label{prop: iterinv}
Let $k=1,\dots,2\rkn-1$. Then
\begin{enumerate}
\item For $L\in \tclass(H\bs G)$ and $g\in G$, $J_k(L;\cdot g)$ is compactly supported on $\bar V^k$ and
\begin{equation}\label{eq: iterinv1}
\int_{\bar V^k\cap H} J_k(L;\bar v g)\ d\bar v=J'_k(L;g).
\end{equation}
\item Let $\pi\in\Irr_{\temp,\meta}\GLnn$ and $W\in \Ind(\WhitML(\pi),\frac12)$. Then
\begin{equation}\label{eq: iterinv2}
\int_{\tilde T_k} J'_k(L_W;t)\wgt{t}^{\rkn-1-[\frac k2]}\ dt=J_{k+1}(L_W;e).
\end{equation}
\end{enumerate}
\end{proposition}

\begin{proof}
Let $L\in\tclass(H\bs G)^{K_0}$, $C=\bar V^k$, $D=V^{k+1}$ and $f=J_k$.
Then $[C,D]\subset D\subset N_k$ and therefore, since $f$ is $(N_k,\psi_N\rest_{N_k})$-equivariant, for $c\in C$, $d\in D$
$$f(cd)=f([c,d]dc)=\psi_N([c,d])\psi_N(d)f(c).$$
It is easy to check that $c\mapsto\psi_N([c,\cdot])$ defines a homeomorphism from $C$ to the Pontryagin dual of
$Z_DN^{\Levi}_{\alpha_{2\rkn-k}}\bs D$.
Thus we can apply Lemma~\ref{L: elemC} to conclude that $J_k(L;\bar v g)$
is supported on a compact set determined by $g$ and $K_0$. In particular the integration in \eqref{eq: iterinv1} is a continuous functional on $\tclass(H\bs G)$. Thus
we only need to prove the first part of the proposition for $L\in \cc{H\bs G}$.
The identity follows by applying Lemma~\ref{L: Fourier} for $f=L(\cdot g)$, $A=N_k$, $B=N^{\Levi}_{\alpha_{2\rkn-k}}Z_{V^{k+1}}\ltimes N_{k+1}$,
$C=\bar V^k\cap H$, $D=V^{k+1}\cap \uder$ and $\Psi(nn')=\psi_N^{-1}(n)$ for $n\in N$, $n'\in\bar N$.
We only need to note that the map $c\mapsto\psi_N([c,\cdot])$ defines a homeomorphism preserving Haar measures between
$\bar V^k\cap H$ and the Pontryagin dual of $H_D\bs D$.

We turn to the second part of the proposition.
First we need a form of Fourier inversion formula:
\begin{lemma}\label{L: genftinv}
Let $\varphi$ be a smooth function on $F$.
Suppose that there exists $\Omega_0\in \csgr(F^*)$ such that $\varphi(x\alpha)=\varphi(x)$ for all $x\in F$, $\alpha\in\Omega_0$.
Assume that $\int_F\abs{\varphi(x)}\max(1,\abs{x})^{-1}\ dx<\infty$.
Then
\begin{equation}\label{eq: stintF}
I:=\int_F\big(\lim_{c\rightarrow\infty}\int_{\abs{x}<c} \varphi(x)\psi^{-1}(xt)\ dx)\ dt
\end{equation}
is well defined and equal to $\varphi(0)$.
\end{lemma}

\begin{proof}
We first observe that for any $t\ne0$ the limit in $c$ (with respect to the discrete topology) exists.
Indeed, $\int_{\abs{x}=b}\varphi(x)\psi^{-1}(xt)=0$ whenever $b\gg_{\Omega_0}\abs{t}^{-1}$.

Consider
$$B(c):=\int_{F}\abs{\int_{\abs{x}=c} \varphi(x)\psi^{-1}(xt)\ dx}\ dt$$
for $c\ge1$. The above integrand in $t$ is supported on $\abs{t}\ll_{\Omega_0} c^{-1}$.
Thus $B(c)\ll_{\Omega_0} c^{-1} \int_{\abs{x}=c} \abs{\varphi(x)}\ dx$ and hence $\sum_{c\ge1}B(c)<\infty$
by assumption.

Let
$$I_c:=\int_{F}\big(\int_{\abs{x}\leq c} \varphi(x)\psi^{-1}(xt)\ dx\big)\ dt.$$
Clearly $I_c$ is absolutely convergent and equals $\varphi(0)$ by Fourier inversion. On the other hand, for any $c_0\ge1$
$$\int_F\abs{\lim_{c\rightarrow\infty}\int_{\abs{x}<c} \varphi(x)\psi^{-1}(xt)\ dx}\ dt\leq
\int_F\abs{\int_{\abs{x}<c_0} \varphi(x)\psi^{-1}(xt)\ dx}\ dt+\sum_{c\ge c_0}B(c)<\infty.$$
Thus, \eqref{eq: stintF} is absolutely convergent and moreover
$$\abs{I-\varphi(0)}=\abs{I-I_{c_0}}\leq\sum_{c\ge c_0}B(c)\rightarrow0\text{ as }c_0\rightarrow\infty.$$
The lemma follows.
\end{proof}

Let $h_k(x):=J_{k+1}(L_W;\lambda_{2\rkn-k}(x))$ where $\lambda_i(x)=\levi(I_{2\rkn}+x\one_{i,i+1})$.
\begin{lemma}\label{L: invinfprep}
There is $\Omega_0\in\csgr(F^*)$ such that $h_k(\alpha x)=h_k(x)$ when $\alpha\in \Omega_0$.
Moreover, for some $d\ge0$ we have $h_k(x)\ll_W\abs{x}^{-\frac12}(\log\abs{x})^d$ as $\abs{x}\rightarrow \infty$.
\end{lemma}

\begin{proof}
Let $\Omega_0$ be such that $L_W(\cdot t)=L_W(\cdot)$ for $t\in \tilde T_k$ of the form $\levi(\diag(I_{2\rkn-k},\alpha,I_{k-1}))$ with $\alpha\in \Omega_0$.
Then since $\tilde T_k$ stabilizes $(N_{k+1},\psi_N\rest_{N_{k+1}})$,
\begin{multline*}
h_k(\alpha x)=J_{k+1}(L_W;t^{-1}\lambda_{2\rkn-k}(x)t)=
\regint_{(H\cap N_{k+1})\bs N_{k+1}} L_W(nt^{-1}\lambda_{2\rkn-k}(x))\psi_N^{-1}(n)\ dn
\\= \regint_{(H\cap N_{k+1})\bs N_{k+1}} L_W(n\lambda_{2\rkn-k}(x))\psi_N^{-1}(n)\ dn
= h_k(x)
\end{multline*}
by a change of variable $n\mapsto t^{-1}nt$.

Next we show the bound. Recall
\begin{multline*}
h_k(x)=\regint_{(H\cap N_{k+1})\bs N_{k+1}} L_W(n\lambda_{2\rkn-k}(x))\psi_N^{-1}(n)\ dn
\\=\int^{st,\T}_{N_{k+1}^\circ\bs N_{k+1}}\big(\int_{(H\cap N_{k+1}^\circ)\bs N_{k+1}^\circ}L_W(nv\lambda_{2\rkn-k}(x))\,dn\big)\psi_N(v)^{-1}\,dv.
\end{multline*}
Here we can choose  $\T=\prod_{i=k+2}^{2\rkn}\tilde T_i$. In particular $\T$ commutes with $\lambda_{2\rkn-k}(x)$.
Thus by \eqref{eq: stableconv}, the integration of $v$ can be taken over a compact set which is independent of $x$.
Moreover $v$ and $\lambda_{2\rkn-k}(x)$ commute. Thus we may ignore the integration over $v$ and it suffices to bound
\[
\int_{(H\cap N_{k+1}^\circ)\bs N_{k+1}^\circ}\abs{L_W(n\lambda_{2\rkn-k}(x))}\,dn.
\]
We show that
\begin{equation} \label{eq: fnlbndU}
\int_{(H\cap N_{k+1}^\circ)\bs N_{k+1}^\circ}\abs{L_W(n\lambda_{2\rkn-k}(x))}\,dn
\ll_{W}\int_{(H\cap\uder)\bs \uder }\abs{L_W(u\lambda_{2\rkn-k}(x))}\,du.
\end{equation}
We argue as in Remark \ref{rem: partint}. Consider a sequence of subgroups
\[
V'_0=N^{\Levi}_{\alpha_{2\rkn-k}}\rtimes N_{k+1}^\circ\subset V'_1\subset\dots\subset V'_m=N^{\Levi}_{\alpha_{2\rkn-k}}\rtimes \uder
\]
such that $\dim V'_i-\dim V'_{i-1}=1$ for all $i=1,\dots,m$ and let $V_i=V'_i\cap\uder$. We show that for all $i$
\[
\int_{H_{V_{i-1}}\bs V_{i-1}}\abs{L_W(v\lambda_{2\rkn-k}(x))}\,dv\ll_{W}\int_{H_{V_i}\bs V_i}\abs{L_W(v\lambda_{2\rkn-k}(x))}\,dv.
\]

Altogether we get \eqref{eq: fnlbndU}.
If $H_{V_{i-1}}=H_{V_i}$, this follows from \eqref{eq: effsgrbnd} and the fact that
\[
\vol_{V_{i-1}\bs V_i}(V_{i-1}(\lambda_{2\rkn-k}(x)^{-1}K_0\lambda_{2\rkn-k}(x)\cap V_i))=
\vol_{V_{i-1}\bs V_i}(V_{i-1}\lambda_{2\rkn-k}(x)^{-1}(K_0\cap V_i)\lambda_{2\rkn-k}(x))
\]
is independent of $x$ for any $K_0\in\csgr(G)$.
On the other hand, if $H_{V_{i-1}}\subsetneq H_{V_i}$ then
\[
\int_{H_{V_{i-1}}\bs V_{i-1}}\abs{L_W(v\lambda_{2\rkn-k}(x))}\,dv=\int_{H_{V_i}\bs V_i}\abs{L_W(v\lambda_{2\rkn-k}(x))}\,dv.
\]

From the bound given in the proof of Proposition~\ref{P: L}
(combining \eqref{eq: bndLW2} and \eqref{eq: bndper2}) the right-hand side of \eqref{eq: fnlbndU} is:
$$\ll_W \int_{(H_\GLnn\cap  N_\GLnn^{\der})\bs N_\GLnn^{\der}}
\int_U (\Xi_H)_\frac12(\levi(n)\lambda_k(x)w_Uu)(\sigma_{H_\GLnn}(n)+\log_+\abs{x}+\sigma(u))^{A_1}\ du\ dn$$
for suitable $A_1\ge0$ where $\log_+\abs{x}=\max(0,\log\abs{x})$.
This by the argument of Lemma~\ref{lem: conv192} is
$$\ll \int_{(H_\GLnn\cap  N_\GLnn^{\der})\bs N_\GLnn^{\der}}
\Xi_{H_\GLnn}(n\lambda_k'(x)))(\sigma_{H_\GLnn}(n)+\log_+\abs{x})^{A_1}\ dn$$
where $\lambda'_k(x)=\levi^{-1}(\lambda_k(x))\in N_\GLnn$.
As in \eqref{bound: goal}, we can unwind the above integral using the definition of $\Xi_{H_\GLnn}$ to get
\begin{multline*}
\int_{(H_\GLnn\cap  N_\GLnn^{\der})\bs N_\GLnn^{\der}}L_0(\Pi_0(\wnn n\lambda'_k(x))\phi_0)(\sigma_{H_\GLnn}(n)+\log_+\abs{x})^{A_1}\,dn
\\=\int_{\overline{N}_\GLnn^{\der}}
\int_{F^\rkn}\phi_0(\udr(t_1,\dots,t_\rkn)\bar u \lambda''_k(x))
(\sigma_{H_\GLnn}(\bar u)+\log_+\abs{x})^{A_1}\textstyle{\prod}_{i=1}^\rkn \abs{t_i}^{-\frac12}\ dt_i\,d\bar u
\end{multline*}
where $\lambda''_k(x)=\wnn\lambda'_k(x)\wnn\in \overline{N}_\GLnn$.
We integrate the above over $\abs{x}<c$ and separate the two cases $k$ even and $k$ odd.
When $k$ is even, the same argument as in the proof of Lemma \ref{lem: ellsigma} shows that the integration of the above expression over
$\abs{x}<c$ is $\ll(\log c)^{A_2}$ for some $A_2>0$.
Thus $h_k(x)\ll\abs{x}^{-1}(\log\abs{x})^{A_2}$.
When $k$ is odd, setting $i_k=\rkn-\frac{k-1}{2}$ we use the change of variable
$\bar u\mapsto \lambda''_k(x)\bar u \lambda''_k(x)^{-1}$ to bound the integral over $\abs{x}<c$ by
\[
\int_{\overline{N}_\GLnn^{\der}}\int_{\abs{x}<c}\int_{F^\rkn}
\phi_0(\udr(t_1,\dots,t_{i_k}+x,\dots,t_\rkn)\bar u)
(\sigma(\bar u)+\log_+\abs{x})^{A_1}\ \textstyle{\prod}_{i=1}^\rkn \abs{t_i}^{-\frac12}dt_i\,dx\,d\bar u.
\]
Separate the integration over $t_{i_k}$ into two parts: $\abs{t_{i_k}}<c$ and $\abs{t_{i_k}}\geq c$.
By a change of variables the first part is
\begin{multline*}
\int_{\overline{N}_\GLnn^{\der}}
\int_{\abs{x}<c}\int_{F^\rkn, \abs{t_{i_k}}<c}\abs{x}^{-\frac12}\big(\textstyle{\prod}_{i=1, i\not=i_k}^\rkn \abs{t_i}^{-\frac12}\big)
\phi_0(\udr(t_1,\dots,t_\rkn)\bar u)
\\(\sigma(\bar u)+\log_+\abs{x+t_i})^{A_1}\ \textstyle{\prod}_idt_i\,dx\,d\bar u
\ll c^{\frac12}(\log c)^{A_1}\\\int_{\overline{N}_\GLnn^{\der}}
\int_{F^\rkn, \abs{t_{i_k}}<c}\big(\textstyle{\prod}_{i=1, i\not=i_k}^\rkn \abs{t_i}^{-\frac12}\big)
\phi_0(\udr(t_1,\dots,t_\rkn)\bar u)
(\sigma(\bar u)+\log_+\abs{t_i})^{A_1}\ \textstyle{\prod}_idt_i\,d\bar u
\end{multline*}
which once again as in Lemma \ref{lem: ellsigma} is $\ll c^{\frac12}(\log c)^{A_3}$ for some $A_3>0$.
On the other hand, the contribution of $\abs{t_{i_k}}\geq c$ is
\begin{multline*}
\int_{\overline{N}_\GLnn^{\der}}\int_{\abs{x}<c}\int_{F^\rkn, \abs{t_{i_k}}\geq c}
\phi_0(\udr(t_1,\dots,t_\rkn)\bar u)
(\sigma(\bar u)+\log_+\abs{x})^{A_1}\textstyle{\prod}_{i=1}^\rkn \abs{t_i}^{-\frac12}dt_i\,dx\,d\bar u
\\\ll c(\log c)^{A_1}\int_{\overline{N}_\GLnn^{\der}}\int_{F^\rkn, \abs{t_{i_k}}\geq c}
\phi_0(\udr(t_1,\dots,t_\rkn)\bar u)
\sigma(\bar u)^{A_1}\textstyle{\prod}_{i=1}^\rkn \abs{t_i}^{-\frac12}dt_i\,d\bar u
\end{multline*}
which again by the argument in Lemma \ref{lem: ellsigma} is $\ll c^{\frac12}(\log c)^{A_4}$ for some $A_4>0$.
The lemma follows.
\end{proof}

We can now finish the proof of \eqref{eq: iterinv2}.
Assume that $L_W$ is right invariant under $K_0\in\csgr(G)$.
Let $f_k(g)=J_{k+1}(L_W;g)\wgt{g}^{\rkn-1-[\frac k2]}$. Then $f_k\in C(\tilde T_k\bs G)^{K_0}$. Thus
\begin{multline*}
\int_{\tilde T_k} J'_k(L_W;t)\wgt{t}^{\rkn-1-[\frac k2]}\ dt=\int_{\tilde T_k}\big(\regint_F f_k(\lambda_{2\rkn-k}(x)t)\psi^{-1}(x)\ dx\big)\ dt
\\=\int_{F}\big(\regint_F f_k(\lambda_{2\rkn-k}(x))\psi^{-1}(xt)\ dx\big)\ dt.
\end{multline*}
Here the regularized integration over $F$ is a stable integral which equals $\lim_{c\rightarrow\infty}\int_{\abs{\cdot}<c}$.
Let $h_k(x)=f_k(\lambda_{2\rkn-k}(x))$. Then  by Lemma~\ref{L: invinfprep}, $h_k$ satisfies the conditions in Lemma~\ref{L: genftinv}.
By Lemma~\ref{L: genftinv}, the above integral is well defined and equal to $h_k(0)=f_k(e)=J_{k+1}(L_W;e)$ as required.
\end{proof}

\section{Model transition for Langlands quotient: II} \label{sec: Etwon}

Motivated by the results of Ginzburg--Rallis--Sodury \cite{MR1671452} we will consider another model for the Langlands quotient.

Let $\radunip$ \index{$\radunip$, $\psi_{\radunip}$} be the unipotent radical of the standard parabolic subgroup of $G$ with Levi subgroup $\GL(2)\times\dots\times\GL(2)$
($\rkn$ times) and let $\psi_{\radunip}$ be a character of $\radunip$ of the form
\begin{equation}\label{eq: defpsiradunip}
\psi_{\radunip}(u)=\psi( (a_1 u_{1,4}+a_2 u_{2,3})+\ldots+(a_{2\rkn-3}u_{2\rkn-3,2\rkn}+a_{2\rkn-2}u_{2\rkn-2,2\rkn-1})+a_{2\rkn-1}u_{2\rkn-1,2\rkn+1})
\end{equation}
where all $a_i\not=0$. Note that $T$ acts transitively on the characters of this form.
The restriction $\psi_{\radunipsgr}$ of $\psi_{\radunip}$ to $\radunipsgr:=\radunip\cap \uder$ is given by
\[
\psi_{\radunipsgr}(u)=\psi( a_1 u_{1,4}+a_3 u_{3,6}+\ldots+a_{2\rkn-3}u_{2\rkn-3,2\rkn}+a_{2\rkn-1}u_{2\rkn-1,2\rkn+1}).
\]
Let $\centpsi$ be the stabilizer of $\psi_{\radunipsgr}$ under the conjugation action of $T$ on the characters of $\radunipsgr$.
Explicitly,
\[
\centpsi=\{\levi(\diag(t_1,\dots,t_{2\rkn})):t_{2i-1}=t_{2i+2},i=1,\dots,\rkn-1, t_{2\rkn-1}=t_{2\rkn}^{-1}\}.
\]
The character $\psi_{\radunip}$ is $(\centpsi,\radunip,\radunipsgr)$-generic and trivial on $H_\radunip$ (which is contained in $\radunipsgr$).
Using Lemma \ref{lem: genstab} we can therefore define
\begin{multline*}
\tranLV(L)(g):=\avg^{H,reg}_{\radunip,\psi_\radunip^{-1}}\circ L(g)=\regint_{H_\radunip\bs\radunip}L(ng)\psi_\radunip^{-1}(n)\ dn
\\=\int^{\st,\centpsi}_{\radunipsgr\bs\radunip}\big(\int_{H_\radunip\bs\radunipsgr}L(ung)\psi_{\radunipsgr}(u)^{-1}\,du\big)\psi_\radunip^{-1}(n)\,dn
\end{multline*}
\index{$\tranLV$} for any $L\in \tclass(H\bs G)$.

The following is a consequence of the results of Ginzburg--Rallis--Soudry and the uniqueness of $(N,\psi_N)$-model of $\LQ{\pi}$
(Proposition \ref{prop: Tadic}).

\begin{lemma}\label{L: uniqueEtwon}
If $\pi\in \Irr_{\meta,\temp}\GLnn$, then $\LQ{\pi}$ has a unique $\model^{(\radunip, \psi_\radunip)}$ model.
\end{lemma}

\begin{proof}
\cite[Theorem 5.7]{MR1954940} and its proof apply to $\LQ{\pi}$.
Therefore, as explained in [ibid.,\S5.8], \cite[Theorem 4.3]{MR1671452} also applies to $\LQ{\pi}$.
Hence, the lemma follows from Proposition \ref{prop: Tadic}.
\end{proof}

We get immediately from Lemmas~\ref{lem: genstab}, \ref{L: stableprop} and Proposition~\ref{P: L}:
\begin{lemma}
$\tranLV$ defines a continuous functional on $\tclass(H\bs G)$.
When $\pi\in\Irr_{\temp,\meta}\GLnn$, $\tranLV$ defines an intertwining map from $\model^{(H,1)}\LQ{\pi}$ to $\model^{(\radunip, \psi_\radunip)}\LQ{\pi}$.
\end{lemma}
Later on we will show that $\tranLV$ is non-zero (Proposition \ref{prop: 14}).

We introduce some variants of the above Lemma.

Let $\subsubU$ (resp., $\subsubUbar$) denote the image under $\toU$ (resp., $\toUbar$) \index{$\subsubU$, $\subsubUbar$}
of the space of strictly upper triangular matrices in $\symspace_{2n}$.

Let $\alltriangle$ be the unipotent subgroup $\langle N_\Levi,\subsubU,\subsubUbar\rangle$ of $G$. \index{$\alltriangle$, $\charalltriangle$}
For instance, for $\rkn=2$, $\alltriangle$ consists of the matrices in $G$ of the form
\[
\left(\begin{smallmatrix}
1&*&*&*& &*&*&*\\
 &1&*&*& & &*&*\\
 & &1&*& & & &*\\
 & & &1& & & & \\
 &*&*&*&1&*&*&*\\
 & &*&*& &1&*&*\\
 & & &*& & &1&*\\
 & & & & & & &1\\
\end{smallmatrix}\right).
\]
Clearly $\alltriangle=(N\cap\alltriangle)\rtimes(\bar U\cap\alltriangle)$. Let $\charalltriangle$ be the character on $\alltriangle$ given by
\[
\charalltriangle(n\bar u)=\psi_N(n),\,n\in N\cap\alltriangle,\bar u\in \bar U\cap\alltriangle.
\]
To verify that $\charalltriangle$ is indeed a character
let $\alpha$ be the Weyl element in $G$ such that $\alpha_{i,2i-1}=1$, $i=1,\dots,2\rkn$,
$\alpha_{2\rkn+i,2i}=-1$, $i=1,\dots,\rkn$, $\alpha_{2\rkn+i,2i}=1$, $i=\rkn+1,\dots,2\rkn$.
Let $\kappa$ be the Weyl element in $\GLnn$ such that for $i\leq 2\rkn$, $\kappa_{i,i}=\kappa_{i+1,i+1}=1$ when $i\equiv 1$ mod $4$ and
$\kappa_{i,i+1}=\kappa_{i+1,i}=1$ if $i\equiv 3$ mod $4$.
Let $x=\alpha\levi(\kappa)$. Then $x\in H$, $x^{-1}\alltriangle x=\alpha^{-1}\alltriangle\alpha=\radunip$
and the function $\psi_{\radunip}$ on $\radunip$ given by
$\psi_{\radunip}(x^{-1}ux)=\charalltriangle(u)$, $u\in\alltriangle$ is a character of the form \eqref{eq: defpsiradunip}
with $a_i=\pm1$.
We can define
\[
\tranLA(L)(g):=\regint_{H_\alltriangle\bs\alltriangle}L(ug)\charalltriangle(u)^{-1}\ du:=\regint_{H_\radunip\bs\radunip}L(ux^{-1}g)\psi_{\radunip}(u)^{-1}\ du
\]
for any $L\in\tclass(H\bs G)$. We get:
\begin{corollary}
$\tranLA$ defines a continuous functional on $\tclass(H\bs G)$. When $\pi\in\Irr_{\temp,\meta}\GLnn$, $\LQ{\pi}$ has
a unique $\model^{(\alltriangle, \charalltriangle)}$ model. $\tranLA$ defines an intertwining map from
$\model^{(H,1)}\LQ{\pi}$ to $\model^{(\alltriangle, \charalltriangle)}\LQ{\pi}$.
\end{corollary}

Let
\[
\few:=\diag(I_\rkn,\sm{}{I_\rkn}{-I_\rkn}{},I_\rkn)\levi(\sm{}{I_\rkn}{w_\rkn}{})=
\left(\begin{smallmatrix}&I_\rkn&&\\&&&w_\rkn\\-w_\rkn&&&\\&&I_\rkn&\end{smallmatrix}\right).
\]
Then $\few$ normalizes $H$.
Define $\Etwon=\few \alltriangle \few^{-1}$ and $\charEtwon=\charalltriangle(\few^{-1}\cdot\few)$.
Then $\bar U\subset\Etwon\subset\bar P$. Write $\Etwon=\levi(\zigzag)\ltimes\bar U$ with $\zigzag\subset\GLnn$.
\index{$\Etwon$, $\charEtwon$} \index{$\zigzag$, $\charzigzag$}
For instance, when $\rkn=4$ $\zigzag$ consists of the matrices of the form
\[
\left(\begin{smallmatrix}
1&*&*&*&*&*&*& \\
 &1&*&*&*&*&&\\
 & &1&*&*&&&\\
 & &&1&&&&\\
 &&&&1&&&\\
 &&&*&*&1&&\\
 &&*&*&*&*&1&\\
 &*&*&*&*&*&*&1
\end{smallmatrix}\right).
\]
Explicitly, $\charEtwon$ is given by
\begin{equation} \label{eq:charEtwon}
\charEtwon(\levi(m)\bar u)=\charzigzag(m)\newchar^{-1}(\bar u),\ \ m\in\zigzag, \bar u\in\bar U.
\end{equation}
where $\newchar(\bar u)=\psi(\bar u_{2\rkn+1,1})$ and
\begin{equation}\label{eq: defpsie}
\charzigzag(m)=\psi(m_{1,2}+\ldots+m_{\rkn-1,\rkn}-m_{\rkn+2,\rkn+1}-\ldots-m_{2\rkn,2\rkn-1}),\ \ m\in\zigzag.
\end{equation}
We can define
\begin{equation}\label{eq: defetwonreg}
\tranLE(L)(g):=\regint_{H_\Etwon\bs\Etwon}L(ug)\charEtwon(u)^{-1}\ du=
\regint_{H_\alltriangle\bs\alltriangle}L(\few u \few^{-1}g)\charalltriangle(u)^{-1}\ du
\end{equation}
\index{$\tranLE$} for any $L\in\tclass(H\bs G)$.
\begin{corollary} \label{cor: uniqueEtwon}
$\tranLE$ defines a continuous functional on $\tclass(H\bs G)$. When $\pi\in\Irr_{\temp,\meta}\GLnn$, $\LQ{\pi}$ has
a unique $\model^{(\Etwon, \charEtwon)}$ model. $\tranLE$ defines an intertwining map from
$\model^{(H,1)}\LQ{\pi}$ to $\model^{(\Etwon, \charEtwon)}\LQ{\pi}$.
\end{corollary}

Note that $\few\in H$ if and only if $\rkn$ is even.
Thus for $L\in\tclass^{\pm1}(H\bs G)$, $L(\few g)=(\pm1)^\rkn L(g)$, and
\begin{equation}\label{eq: relationetwon}
(\pm1)^\rkn\regint_{H_\Etwon\bs\Etwon}L(u\few)\charEtwon(u)^{-1}\ du=
\regint_{H_\alltriangle\bs\alltriangle}L( u )\charalltriangle(u)^{-1}\ du.
\end{equation}

\section{A family of integral transforms} \label{sec: analytic}

Our next goal is to describe explicit an model transition between $\model^{N,\psi_N}\LQ{\pi}$ and
$\model^{(\Etwon, \charEtwon)}\LQ{\pi}$ which will coincide with $\tranLE\circ\tranMWL$ but which is given
more directly. Basically we want to integrate over $\Etwon\cap N\bs\Etwon$, i.e. over $(\zigzag\cap\overline{N_\GLnn})\ltimes\bar U$.
This is a little more involved than the previous cases because of convergence issues.

We fix $K_0\in\csgr(G)$ and $Z_{\Levi}'=\levi(Z_{\GLnn}')\in\csgr(Z_\Levi)$ throughout. We will define
a family of integral transforms from $C^{\smth}(Z_{\Levi}'N\bs G,\psi_N)$ to $C^{\smth}(\Etwon\bs G, \charEtwon)$.
Eventually we will show in \S\ref{sec: prf of prop: main} that these transforms all equal to $\tranLE\circ\tranMWL$ when restricted to $\model^{N,\psi_N}\LQ{\pi}$.

\subsection{Regularization of an integration over $\bar U$}\label{sec: regU}

Let $\bar U^{\dag}\subset \bar U$ be the subgroup consisting of $\toUbar(v)$ where the last row of $v\in\symspace_{2\rkn}$
(and hence also the first column) is $0$. \index{$\bar U^{\dag}$}
Let $\bar U_{1,1}$ be the one-parameter subgroup $\{\toUbar(\lambda\one^{\symspace_{2\rkn}}_{1,1}):\lambda\in F\}$ of $\bar U$.
\index{$\bar U_{1,1}$}

The character $\newchar$ is clearly $(Z_\Levi,\bar U^{\dag}\bar U_{1,1},\bar U^{\dag})$-generic
and trivial on $\bar U^{\dag}$.
Using Lemma \ref{lem: genstab} we define
\[
\regint_{\bar U^{\dag}\bar U_{1,1}}f(\bar v)\newchar(\bar v)\,d\bar v=
\int_{\bar U^{\dag}\bs (\bar U^{\dag}\bar U_{1,1})}^{\st,Z_\Levi}\big(\int_{\bar U^{\dag}}f(\bar u\bar v)\,d\bar u\big)\newchar(\bar v)\,d\bar v
\]
for any $f\in C^{\smth}(Z_{\Levi}'\bs G)$ such that the integral $\int_{\bar U^{\dag}}\abs{f(\bar ug)}\ d\bar u$ converges
for all $g\in G$.
If furthermore, the integral
\begin{equation} \label{def: regint}
\int_{(\bar U_{1,1}\bar U^{\dag})\bs \bar U}\big(\regint_{\bar U^{\dag}\bar U_{1,1}}f(\bar v\bar u)\newchar(\bar v)\,d\bar v\big)
\newchar(\bar u)\ d\bar u
\end{equation}
converges, we denote it by $\regint_{\bar U} f(\bar v)\newchar(\bar v)\,d\bar v$.

It is clear that if $\regint_{\bar U} f(\bar v)\newchar(\bar v)\,d\bar v$ is defined then for any $\bar u\in\bar U$
\begin{equation} \label{eq: regUtrans}
\regint_{\bar U} f(\bar v\bar u)\newchar(\bar v)\,d\bar v\text{ is defined and equal to }\newchar^{-1}(\bar u)\regint_{\bar U} f(\bar v)\newchar(\bar v)\,d\bar v.
\end{equation}
Let $\GLnn^\fix$ be the stabilizer of the character $\newchar$ of $\bar U$ in $\GLnn$, i.e.:
\begin{equation}\label{eq: defM1}
\GLnn^\fix=\{m\in \GLnn: \newchar(\levi(m)\bar v\levi(m)^{-1})=\newchar(\bar v)\text{  for all }\bar v\in\bar U\}.
\end{equation}
Explicitly, $\GLnn^\fix$ consists of the matrices in $\GLnn$ whose first (resp. last) column is $(z,0,\dots,0)^t$ (resp. $(0,\dots,0,z^{-1})^t$)
for some $z\in F^*$. Let $B_\GLnn^\fix=B_\GLnn\cap\GLnn^\fix$. \index{$\GLnn^\fix$, $B_\GLnn^\fix$}
Note that $B_\GLnn$ normalizes $\bar U^{\dag}$ and $\bar U_{1,1}\bar U^{\dag}$.
By Remark \ref{rem: aut} we have
\begin{equation} \label{eq: conjborelevi}
\regint_{\bar U} f(\levi(m)\bar v\levi(m)^{-1})\newchar(\bar v)\,d\bar v=\modulus_P(\levi(m))\regint_{\bar U} f(\bar v)\newchar(\bar v)\,d\bar v
\end{equation}
for any $m\in B_\GLnn^\fix$, provided that $\regint_{\bar U} f(\bar v)\newchar(\bar v)\,d\bar v$ is well defined.

\begin{lemma}\label{L: Ureg}
Let $W\in C(Z_{\Levi}'N\bs G,\psi_N)^{K_0}$ and let $\Omega_{T_\GLnn}$ be a compact subset of $T_\GLnn$.
Then for any $t\in\Omega_{T_\GLnn}$
\[
\regint_{\bar U}W(\levi(t)\bar v)\newchar(\bar v)\,d \bar v
\]
is defined and
\[
\regint_{\bar U}W(\levi(t)\bar v)\newchar(\bar v)\,d \bar v=
\int_{\Omega_{\bar U}}W(\levi(t)\bar v)\newchar(\bar v)\,d \bar v
\]
for some $\Omega_{\bar U}\in\csgr(\bar U)$ depending only on $K_0$, $\Omega_{T_\GLnn}$ and $Z_{\Levi}'$.
\end{lemma}

\begin{proof}
We first show for any $g\in G$ the function $\bar v\mapsto W(\levi(t)\bar{v}g)$ on $\bar U^{\dag}$ is compactly supported
with the support contained in a compact set depending on $gK_0g^{-1}$ and $\Omega_{T_\GLnn}$ (for $t\in\Omega_{T_\GLnn}$).

Indeed, suppose that $W(\levi(t)\bar{v}g)\ne0$ and let $\bar{v}=nak$ be the Iwasawa decomposition, with $a=\levi(\diag(a_1,\ldots,a_{2\rkn}))$.
Then $\abs{a_1}=1$ since $\bar v\in \bar U^{\dag}$. From the support condition of Whittaker functions (on $\GLnn$),
we get that $\abs{t_ia_i}^{-1}\ll_{gK_0g^{-1}}\abs{t_1}^{-1}$, $i=2,3,\dots,2\rkn$.
On the other hand, writing $\bar{v}=\toUbar(x)$, the entries of $x$ are bounded by $\prod_{i=1}^{2\rkn}\abs{a_i}^{-1}$.

Thus
\[
\phi(g):=\regint_{\bar U^{\dag}\bar U_{1,1}}W(\levi(t)\bar{v}g)\newchar(\bar v)\,d\bar v
\]
is well defined. Moreover, we can write
\[
\phi(g)=\int_{\Omega'}W(\levi(t)\bar{v}g)\newchar(\bar v)\,d\bar v
\]
where $\Omega'\in\csgr(\bar U^{\dag}\bar U_{1,1})$ depends only on $\Omega_{T_\GLnn}$, $Z_{\Levi}'$ and $gK_0g^{-1}$.

We claim that $\phi$ is compactly supported on $(\bar U_{1,1}\bar U^{\dag})\bs \bar U\cong F^{2\rkn-1}$. %
Let $v'$ be a column vector in $F^{2\rkn-1}$ and let $\beta(v')=\levi(\sm{I_{2\rkn-1}}{v'}{}{1})$.
For $\bar u\in \bar U$ we have $[\bar u,\beta(v')]\in \bar U_{1,1}\bar U^{\dag}$ and
\[
\newchar([\bar u,\beta(v')])=\psi_{v'}(\bar u):=\psi(\bar u_{4\rkn,1}v'_1+\dots+\bar u_{4\rkn,2\rkn-1}v'_{2\rkn-1}).
\]
Note that $N_{\Levi}$ stabilizes the pair $(\bar U_{1,1}\bar U^{\dag},\newchar\rest_{\bar U_{1,1}\bar U^{\dag}})$ under conjugation. Thus $\phi(g)$ is left $(N_{\Levi},\psi_{N_{\Levi}}(\levi(t)\cdot\levi(t)^{-1}))$-equivariant.
Clearly $\phi(g)$ is also left $(\bar U_{1,1}\bar U^{\dag},\newchar^{-1})$-equivariant. We get
\[
\phi(\bar u \beta(v'))=\psi_{v'}(\bar u)\phi(\beta(v')\bar u)
=\psi_{N_{\Levi}}(\levi(t)\beta(v')\levi(t)^{-1})\psi_{v'}(\bar u)\phi(\bar u).
\]
Let $C$ be the image under $\toUbar$ of the subspace of $\symspace_{2\rkn}$ consisting of matrices $v$ such that
$v_{i,j}=0$ unless either $i>j=1$ or $2\rkn=i>j$. Thus $\bar U=\bar U_{1,1}\times\bar U^{\dag}\times C$ and
by Lemma~\ref{L: elemC} we get that $\phi\rest_C$ is compactly supported and the support is bounded in terms of $K_0$ and $\Omega_{T_\GLnn}$.
The lemma follows.
\end{proof}

\begin{corollary}[of proof] \label{cor: suppW}
The restriction of any $W\in C^{\smth}(N\bs G,\psi_N)$ to $\bar U^{\dag}$ is compactly supported.
Moreover, for any $K_0\in\csgr(G)$ there exists $\bar U^{\dag}_c\in\csgr(\bar U^{\dag})$ such that
the support of $W\rest_{\bar U^{\dag}}$ is contained in $\bar U^{\dag}_c$ for any $W\in C(N\bs G,\psi_N)^{K_0}$.
\end{corollary}

\subsection{}
Following Lemma \ref{L: Ureg}, for any $W\in C^{\smth}(Z_{\Levi}'N\bs G,\psi_{N})$, $t\in T_{\GLnn}$ and $m\in \GLnn$ define
\begin{equation}\label{eq: regu}
\FW_W(t,m):=
\regint_{\bar U}W\left(\levi(t)\bar v\levi(m) \right)\newchar(\bar v)\ d\bar v.
\end{equation}
\index{$\FW_W$}

We observe some properties of $\FW_W(t,m)$.
Let $N_\GLnn^\fix=N_\GLnn\cap \GLnn^\fix$ and $T_\GLnn^\fix=T_\GLnn\cap\GLnn^\fix$. \index{$N_\GLnn^\fix$, $T_\GLnn^\fix$}
\begin{lemma}\label{L: FWproperty}
For any $W\in C^{\smth}(Z_{\Levi}'N\bs G,\psi_N)$ and any $m\in\GLnn$, $t\in T_\GLnn$ we have:
\begin{enumerate}
\item $\FW_W(t,nm)=\psi_{N_\GLnn}(tnt^{-1})\FW_W(t,m)$ for any $n\in N_\GLnn^\fix$.
\item $\FW_W(tt',m)=\abs{\det t'}^{2\rkn+1}\FW_W(t,t'm)$ for any $t'\in T_\GLnn^\fix$.
\item $\FW_{W_{\bar u}}(t,m)=\FW_W(t,m)\newchar(\bar u)^{-1}$ if $m\in\GLnn^\fix$, $\bar u\in \bar U$ and $W_{\bar u}(\cdot)=W(\cdot \bar u)$.
\end{enumerate}
\end{lemma}

Note that we also have $\FW_W(zt,m)=\FW_W(t,m)$ for any $z\in Z_{\GLnn}'$ but we won't use this fact.

\begin{proof}
Let $n\in N_\GLnn^\fix$.
From \eqref{eq: regu}, \eqref{eq: conjborelevi} and the equivariance of $W$, we get
\begin{multline*}
\FW_W(t,nm)=\regint_{\bar U}W\left(\levi(t)\bar v\levi(nm) \right)\newchar(\bar v)\ d\bar v
=\regint_{\bar U}W\left(\levi(tn)\bar v\levi(m) \right)\newchar(\bar v)\ d\bar v\\=
\psi_{N_\GLnn}(tnt^{-1})\regint_{\bar U}W\left(\levi(t)\bar v\levi(m) \right)\newchar(\bar v)\ d\bar v=
\psi_{N_\GLnn}(tnt^{-1})\FW_W(t,m).
\end{multline*}
Hence, the first claim.
For the second claim, we get from \eqref{eq: conjborelevi}:
\begin{multline*}
\FW_W(t,t'm)=\regint_{\bar U}W\left(\levi(t)\bar v\levi(t'm) \right)\newchar(\bar v)\ d\bar v\\=
\abs{\det t'}^{-(2\rkn+1)}\regint_{\bar U}W\left(\levi(tt')\bar v\levi(m) \right)\newchar(\bar v)\ d\bar v=\abs{\det t'}^{-(2\rkn+1)}\FW_W(tt',m).
\end{multline*}
The last claim follows from \eqref{eq: regUtrans}. Indeed, for any $m\in\GLnn$
$$
\FW_{W_{\bar u}}(t,m)=\regint_{\bar U}W\left(\levi(t)\bar v\levi(m)\bar u \right)\newchar(\bar v)\ d\bar v=\newchar^{-1}(\levi(m)\bar u\levi(m^{-1}))
\FW_W(t,m). \qedhere
$$
\end{proof}

\subsection{Some classes of functions}
Motivated by the above properties of $\FW_W$, we define some classes of function on $T_\GLnn\times \GLnn$.

We first introduce some notations and auxiliary unipotent subgroups of $\GLnn$.

Let $\alg{X_\GLnn}$ be the closed subvariety $\alg{N_\GLnn}\alg{N}_\GLnn^t$ of $\GLnn$.
\index{$X_\GLnn$}
Thus, any element of $X_\GLnn$ can be written uniquely as $n \bar n$ with $n\in N_\GLnn$ and $\bar n\in N_\GLnn^t$.
For $t\in T_\GLnn$, define a function $\chixmt$ on $X_\GLnn$ by \index{$\chixmt$}
\begin{equation}\label{eq: defchixmt}
\chixmt(n\bar n)=\psi_{N_\GLnn}(tnt^{-1})\psi_{N_\GLnn}^{-1} (\bar n^t),\,\, n\in N_\GLnn, \,\bar n\in N_\GLnn^t.
\end{equation}

For $k=0,\ldots,4\rkn-1$, let $\LL^k$ be the unipotent subgroup of $N_\GLnn\subset\GLnn$ generated by
the commuting one-parameter root subgroups $\langle N^\GLnn_{i,j}:i<j, i+j=k+1\rangle$, that is,
\[
\LL^k=\{I_{2\rkn}+u:\, u_{i,j}=0\text{   if   }i+j\not=k+1 \text{  or  }i\geq j\}\simeq F^{\min([\frac{k}{2}],2\rkn-[\frac{k+1}{2}])}.
\]
Let $\overline\LL^k=\wnn \LL^k\wnn\subset N_\GLnn^t$ so that
\[
\overline\LL^k=\{I_{2\rkn}+u:\,  u_{i,j}=0\text{   if   }i+j\not=4\rkn+1-k \text{  or  }i\leq j\}.
\]
\index{$\LL^k$, $\overline\LL^k$, $\hat\LL^k$}
Note that $\LL^0=\overline\LL^0=\LL^1=\overline\LL^1=1$. It is clear $\chixmt\rest_{\LL^k}=\chixmt\rest_{\overline\LL^k}\equiv1$ when $k$ is odd.
Define
$$\hat\LL^k=\langle \LL^0,\dots, \LL^{4\rkn-k-1},\overline\LL^0,\dots,\overline\LL^{k-1}\rangle.$$

We illustrate the groups in the case $\rkn=4$. In the template below:
\begin{equation} \label{eq: template1}
\left(\begin{smallmatrix}
1&2&3&4&5&6&7& 8\\
 &1&4&5&6&7&8&9\\
 & &1&6&7&8&9&10\\
 & &&1&8&9&10&11\\
 & & & &1&10&11&12\\
 & & &7&6&1&12& 13\\
 & &7&6&5&4&1&14\\
 &7&6&5&4&3&2&1
\end{smallmatrix}\right)
\end{equation}
the numbers $>1$ above and below the principal diagonal mark the non-constant coordinates of $\LL^k$ and $\overline\LL^k$ respectively.
(The groups $\overline\LL^k$ with $k\ge2\rkn$ will not be used.)
Meanwhile for $k\le 8$,
each group $\hat\LL^k$ corresponds to the numbers $<k$ below the principal diagonal and the numbers $<16-k$ above it.

Let $\hat\M^k=\hat\LL^k\cap \GLnn^\fix$. \index{$\hat\M^k$} We observe that $\hat\LL^k\subset X_\GLnn$ and
\begin{equation}\label{eq: specialMk}
\hat\LL^0=N_\GLnn;\,\,\hat\LL^{2\rkn}=\hat\M^{2\rkn}=\zigzag;\,\,\hat\M^2=N_\GLnn^\fix; \,\,\chixmt\rest_{\zigzag}=\charzigzag \ \
\text{for }t\in\soment
\end{equation}
where $\soment=\{\diag(t_1,\dots,t_{2\rkn}):t_1=\dots=t_{\rkn}\}$. \index{$\soment$}
For $i=1,\dots,\rkn$ let $\Imk_i$ be the one-dimensional torus \index{$\Imk_i$}
\[
\Imk_i:=\{\diag(\overbrace{z^{-1},\ldots,z^{-1}}^{2\rkn-i},\overbrace{z,\ldots,z}^i):z\in F^*\}\subset T_\GLnn^\fix\cap\soment.
\]
Note that
as algebraic groups (but not at the level of $F$-points) we have $\alg{\soment}=\alg{Z_{\GLnn}}\cdot\prod_{i=1}^{\rkn}\alg{\Imk_i}$.

For any
$k=1,\dots,\rkn$ (resp. $k=1,\dots,\rkn-1$) consider the classes
$\evencls_k(T_\GLnn\times\GLnn)^{K_0}$ (resp. $\oddcls_k(T_\GLnn\times\GLnn)^{K_0}$) of functions $\phi$ on
$T_\GLnn\times \GLnn$ satisfying the following properties
for any $t\in T_\GLnn$, $m\in\GLnn$: \index{$\evencls_k(T_\GLnn\times\GLnn)^{K_0}$, $\oddcls_k(T_\GLnn\times\GLnn)^{K_0}$}
\begin{subequations}
\begin{gather}
\label{p: K0inv} \phi(t,mk)=\phi(t,m)\text{ for any }k\in \levi^{-1}(K_0\cap M),\\
\label{p: equim} \phi(t,nm)=\chixmt(n)\phi(t,m)\text{ for any }n\in \hat\M^{2k}\text{ (resp. $\hat\M^{2k+1}$)},\\
\label{p: switcht} \phi(tt',m)=\abs{\det t'}^{2\rkn+1}\modulus_{T_\GLnn;\overline\LL^1\cdots\overline\LL^{2k-1}}^{-1}(t')
\phi(t,t'm)\text{ for any }t'\in\prod_{i=k}^{\rkn-1}\Imk_i
\\\nonumber(\text{resp., }\phi(tt',m)=\abs{\det t'}^{2\rkn+1}\modulus_{T_\GLnn;\overline\LL^1\cdots\overline\LL^{2k}}^{-1}(t')
\phi(t,t'm)\text{ for any }t'\in\prod_{i=k+1}^{\rkn-1}\Imk_i).
\end{gather}
\end{subequations}
For instance, by Lemma~\ref{L: FWproperty}, if $W\in C^{\smth}(Z_{\Levi}'N\bs G,\psi_N)^{K_0}$ then
$\FW_W(t,m)\in\evencls_1(T_\GLnn\times \GLnn)^{K_0}$.

\subsection{Definition of $\tranWAv^{t,\auxf}(W)$}
Define
\begin{equation}\label{eq: defvs}
\vs=\zigzag\cap\overline{N_\GLnn}=\overline\LL^{2\rkn-1}\rtimes(\overline\LL^{2\rkn-2}\rtimes(\cdots\rtimes\overline\LL^1)).
\end{equation}

The group $\vs$ corresponds to the lower unitriangular part in the template \eqref{eq: template1}.
Define $\psi_{\vs}=\charzigzag\rest_{\vs}^{-1}$, i.e. \index{$\vs$, $\psi_{\vs}$}
$$\psi_{\vs}(\bar v)=\psi_{N_\GLnn}(\bar v^t).$$

It will be also useful to set \index{$\oddL^k$, $\evenL^k$}
\begin{gather}
\oddL^k=\overline\LL^{2k+1}, \evenL^k=\overline\LL^{2k}, \ \ k=0,\dots,\rkn-1.
\end{gather}
Note that $\dim\oddL^k=\dim\evenL^k=k$ and in particular $\oddL^0=\evenL^0=1$.  Also $\psi_{\vs}\rest_{\oddL^k}=1$.

Denote the simple roots of $\GLnn$ by $\alpha_1,\dots,\alpha_{2\rkn-1}$, so that the corresponding one-parameter root groups are
$N_{\alpha_i}=N^\GLnn_{i,i+1}$ and $N_{-\alpha_i}=N^\GLnn_{i+1,i}$.
Write $\evenL^k=\evenL_\circ^k\times N_{-\alpha_{2\rkn-k}}$ where
\[
\evenL_\circ^k=\{I_{2\rkn}+u:\,  u_{i,j}=0\text{   if   }i+j\not=4\rkn+1-2k \text{  or  }i\leq j+1\}.
\]
For a two-variable function such as $\phi(t,m)$ there are two possible meanings for the operation $\avg$, namely either in the $m$ variable or
in the $t$ variable. We use $\avg$ to denote the integration in the $m$ variable.

For any subtorus $T'$ of $T_{\GLnn}$ and $f'\in\cc{T'}$ we will write
\[
\avgt{T'}[f']\circ\phi:=\int_{T'}f'(t')\phi(\cdot t';\cdot)\ dt'
\]
(if defined).
The following Lemma will be proved in the next section.
\begin{lemma}\label{L: Mreg2}
Fix a compact subset $\Omega_{T_\GLnn}$ of $T_\GLnn$ and $k=1,\dots,\rkn-1$.
\begin{enumerate}
\item Let $\phi(t,m)\in\evencls_k(T_\GLnn\times \GLnn)^{K_0}$ and $f_k\in \cc{\Imk_k}$. Then
\begin{enumerate}
\item $\avgt{\Imk_k}[f_k]\circ\phi$ is a convergent integral. Moreover $\avgt{\Imk_k}[f_k]\circ\phi\in\evencls_k(T_\GLnn\times \GLnn)^{K_0}$.
\item $\avg^{reg}_{\evenL^k, \psi_{\vs}}\circ\avgt{\Imk_k}[f_k]\circ\phi$
is defined in the sense of Lemma~\ref{lem: genstab}, with $\T=\Imk_k$, $U_0=\evenL^k$, $U_1=\evenL_\circ^k$, $U_2=1$ and $\psi_{U_0}=\psi_{\vs}^{-1}$. Moreover $\avg^{reg}_{\evenL^k, \psi_{\vs}}\circ\avgt{\Imk_k}[f_k]\circ\phi$ belongs to $\oddcls_k(T_\GLnn\times \GLnn)^{K_0}$, and there exists $\Omega\in \csgr(\evenL^k)$
depending on $\Omega_{T_\GLnn}$, $K_0$ and $f_k$ such that
\[
\avg^{reg}_{\evenL^k, \psi_{\vs}}\circ\avgt{\Imk_k}[f_k]\circ \phi(t,e)=
\avg_{\Omega, \psi_{\vs}}\circ\avgt{\Imk_k}[f_k]\circ \phi(t,e)
\]
for all $t\in\Omega_{T_\GLnn}$.
\end{enumerate}
\item Let $\phi(t,m)\in\oddcls_k(T_\GLnn\times \GLnn)^{K_0}$. Then for all $t\in\Omega_{T_\GLnn}$, as a function of $m\in\oddL^k$, $\phi(t,m)$ is supported on
some $\Omega\in \csgr(\oddL^k)$ determined by $\Omega_{T_\GLnn}$ and $K_0$. Moreover
$\avg_{\oddL^k}\circ \phi$  belongs to $\evencls_{k+1}(T_\GLnn\times \GLnn)^{K_0}$.
\end{enumerate}
\end{lemma}

For $t=\diag(t_1,\dots,t_{2\rkn})\in T_\GLnn$, set
$$\factor(t)=\abs{t_1}^{-\rkn}\modulus_B^{\frac12}(\levi(t)).$$
\index{$\factor$}
Let $\spcltrs$ be the torus $\prod_{i=1}^{\rkn-1}\Imk_i$.
We can then define the following integral transforms:

\begin{definition}\label{def: tranWAv}
Identify $\cc{\spcltrs}$ with $\auxF$ by $(f_1\otimes\dots\otimes f_{\rkn-1})(t_1\dots t_{\rkn-1})=\prod_if_i(t_i)$.
For $\auxf=f_1\otimes\ldots\otimes f_{\rkn-1}\in\auxF$ we define $\tranY$ on $\evencls_1(T_\GLnn\times \GLnn)^{K_0}$ by
\[
\tranY\circ\phi=\avg_{\oddL^{\rkn-1}}\circ\avg^{reg}_{\evenL^{\rkn-1},\psi_{\vs}}\circ\avgt{\Imk_{\rkn-1}}[f_{\rkn-1}]\circ\cdots
\circ\avg_{\oddL^1}\circ\avg^{reg}_{\evenL^1,\psi_{\vs}}\circ\avgt{\Imk_1}[f_1]\circ\phi
\]
(and extend $\tranY$ to $\auxF$ by multi-linearity).
\index{$\tranY$, $\tranWAv^{t,\auxf}$} Also define for $W\in C^{\smth}(Z_{\Levi}'N\bs G,\psi_{N})$
\[
\tranWAv^{t,\auxf}(W)(g):=\factor(t)^{-1} \tranY\circ\FW_{W_g}(t,e)
\]
where $W_g(\cdot)=W(\cdot g)$.
\end{definition}

\begin{corollary} \label{cor: abcde}
Notations as in Lemma~\ref{L: Mreg2}.
\begin{enumerate}
\item There exists $\Omega\in\csgr(\vs)$ depending only on $K_0$, $\Omega_{T_\GLnn}$ and $\auxf$, such that for all $t\in \Omega_{T_\GLnn}$ we have
\[
\tranY\circ\phi(t;e)=\avg_{\Omega,\psi_{\vs}}\circ \avgt{\spcltrs}[\auxf]\circ\phi(t;e).
\]
\item There exists $\Omega'\in\csgr(\levi(\vs)\ltimes\bar U)$
depending only on $K_0$, $\Omega_{T_\GLnn}$ and $\auxf$ such that
\[
\tranY\circ\FW_{W_g}(t,e)=\int_{\spcltrs}\int_{\Omega'}\auxf(t')W(\levi(tt')vg)\charEtwon^{-1}(v)\ d v \ dt'
\]
for all $t\in\Omega_{T_\GLnn}$.
\item We have
\[
\tranY\circ\FW_{W_g}(t,e)=\int_{\spcltrs}\int_{\levi(\vs)\ltimes\bar U}
\auxf(t')W(\levi(tt')vg)\charEtwon^{-1}(v)\ d v \ dt'
\]
if the right-hand side is absolutely convergent.
\item If $t\in\soment$ then $\tranWAv^{t,\auxf}(W)\in C^{\smth}(\Etwon\bs G, \charEtwon)$.
\end{enumerate}
\end{corollary}

\begin{proof}
All parts except the last one are immediate from Lemmas~\ref{L: Ureg} and \ref{L: Mreg2}. It follows from Lemma~\ref{L: Mreg2} and
\eqref{eq: specialMk}  that
$\tranWAv^{t,\auxf}(W)\in C^{\smth}(\levi(\zigzag)\bs G, \charzigzag\circ\levi^{-1})$, while from Lemma~\ref{L: FWproperty},
$\tranWAv^{t,\auxf}(W)\in C^{\smth}(\bar U\bs G, \newchar^{-1})$.
\end{proof}

Recall that by Lemma~\ref{L: stablemodel}, for any $L\in\tclass(H\bs G)$,
we have $\tranLW(L)\in C^{\smth}(Z_{\Levi}N\bs G,\modulus_P^{\frac12}\nu^{-1/2}\psi_{N})$.
We will prove the following functional equation in \S\ref{sec: prf of prop: main}:
\begin{proposition}\label{prop: main}
Let $c(\auxf)=\int_{\spcltrs}\auxf(t')\factor(t')\ dt'$.
Then for any $L\in\tclass(H\bs G)$ we have
\begin{equation}\label{eq: equalt}
\tranWAv^{t,\auxf}\circ\tranLW(L)=c(\auxf)\tranLE(L)
\end{equation}
for all $t\in \soment$ and $\auxf\in \auxF$.
Consequently, by Theorem \ref{thm: MWL} we have $\tranWAv^{t,\auxf}=c(\auxf)\tranLE\circ\tranMWL$ on $\model^{(N,\psi_N)}\LQ{\pi}$.
\end{proposition}

\begin{remark}\label{rem: compactonly}
Note that by Corollary \ref{cor: uniqueEtwon}, $\tranLE\circ\tranMWL$ and $\tranWAv^{t,\auxf}$ (for $t\in\soment$) are proportional
on $\model^{(N,\psi_N)}\LQ{\pi}$ with the constant of proportionality depending a priori on $t$, $\auxf$ and $\pi$.

Also note that for fixed $t$ and $\auxf$, $L\mapsto\tranWAv^{t,\auxf}\circ\tranLW(L)(e)$ is a continuous functional on $\tclass(H\bs G)$,
being equal to a (finite) linear combination (depending on $K_0$) of the values of the function $\int_{H_N\bs \uder}L(u\cdot)\ du$.
On the other hand  $\tranLE$ is also a continuous functional on $\tclass(H\bs G)$.
Thus we only need to prove the identity \eqref{eq: equalt} for $L\in \cc{H\bs G}$ by Lemma~\ref{L: stablemodel} part \eqref{mod: dense}.
\end{remark}

\section{Proof of Lemma \ref{L: Mreg2}}\label{sec: lemma}
\subsection{}\label{sec: defLk}
We first introduce a few more families of unipotent groups.

For $k=0,\dots,4\rkn-1$,
define
\[
\acutee\LL^k=\langle \LL^0,\dots, \LL^{4\rkn-k-1},\overline\LL^0,\dots,\overline\LL^{k}\rangle,
\]
so that $\hat\LL^k=\acutee\LL^k\cap\acutee\LL^{k-1}$.
Also set
\[
\check\LL^k=\hat\LL^k\cap\hat\LL^{k+1}=\acutee\LL^{k+1}\cap\acutee\LL^{k-1}=\langle \LL^0,\dots, \LL^{4\rkn-k-2},\overline\LL^0,\dots,\overline\LL^{k-1}\rangle.
\]
\index{$\acutee\LL^k$, $\check\LL^k$}
For instance, in the case $\rkn=4$ and $k\le 8$ in the diagram \eqref{eq: template1} above
$\check\LL^k$ corresponds to the numbers $<k$ below the principal diagonal and the numbers $<15-k$ above it;
$\acutee\LL^k$ (for $k<8$) corresponds to the numbers $\le k$ below the principal diagonal and the numbers $<16-k$ above it.

Observe that $\check\LL^k\subset\hat\LL^k\subset\acutee\LL^k$, $\acutee\LL^0=\acutee\LL^1=N_\GLnn$ and
\begin{equation} \label{eq: k4n-k}
\acutee\LL^k=\wnn \acutee\LL^{4\rkn-1-k}\wnn, \ \ k=0,\dots,4\rkn-1.
\end{equation}
Note that $\acutee\LL^k$ is a maximal unipotent subgroup of $\GLnn$. In fact for $k\leq 2\rkn$,
$\acutee\LL^k=\sigma_k^{-1}N_\GLnn\sigma_k$ where $\sigma_k$ are the permutation matrices given by
\[
\sigma_ke_i=\begin{cases} e_i & i\leq 2\rkn-k,\\
e_{2i+k-2\rkn}& 2\rkn-k< i\le 2\rkn-[\frac{k+1}2],\\ e_{6\rkn-k-2i+1} &i> 2\rkn-[\frac{k+1}2],\end{cases}\\
\]
where $e_1,\dots,e_{2\rkn}$ is the standard basis and $l=1,\dots,\rkn$.

\begin{remark}\label{r: Lshape}
For $k=0,\dots,2\rkn$, consider the isomorphism $\Sigma_k: m\mapsto \sigma_k m\sigma_k^{-1}$ from $\acutee\LL^k$ to $N_\GLnn$.
Then we have
\[
\Sigma_k(N^\GLnn_{i,4\rkn-k-i})=N_{\alpha_{2i-2\rkn+k}}, \ \ 4\rkn-2k\leq 2i<4\rkn-k,
\]
and
\[
\Sigma_k(N^\GLnn_{2\rkn-i,i+1+2\rkn-k})=N_{\alpha_{2i+1+2\rkn-k}}, \ \ 0\leq 2i<k-1,
\]
so that,
\[
\Sigma_k(\LL^{4\rkn-k-1})=\prod_{i=0}^{[\frac{k-1}{2}]} N_{\alpha_{2i+2\rkn-k}},\,\,
\Sigma_k(\overline\LL^k)=\prod_{i=0}^{[\frac{k}{2}]-1}N_{\alpha_{2i+1+2\rkn-k}}.
\]
If we denote by $c_i(u)=u_{i,i+1}$, $i=1,\dots,2\rkn-1$ the entries of $u\in N_\GLnn$ in the simple roots then
\begin{gather*}
\Sigma_k(\hat\LL^k)=\{u\in N_\GLnn:c_{2i+1+2\rkn-k}(u)=0, 0\leq 2i<k-1\},\\
\Sigma_k(\check\LL^k)=\{u\in N_\GLnn:c_{i+2\rkn-k}(u)=0, 0\leq i<k\}.
\end{gather*}
\end{remark}

\begin{lemma}
For $k=1,\dots,2\rkn$, $\chixmt\rest_{\check\LL^k}$ is a character on $\check\LL^k$.
\end{lemma}
\begin{proof}
For $k=1,\dots,2\rkn$, $N_{\GLnn, k, I}:=\Sigma_k(N_\GLnn^t\cap \acutee\LL^k)$ is the subgroup of $N_\GLnn$ consisting of the matrices whose
off-diagonal entries are zero at rows $1,\dots,2\rkn-k$ and $2\rkn-k+2i$, $i=1,\dots,\frac{[k]}{2}$.
In particular, for the negative simple root groups $N_{-\alpha_{2\rkn-i}}\subset N_\GLnn^t\cap \acutee\LL^k$, $i=1,\dots,[\frac{k}{2}]$
we have $\Sigma_k(N_{-\alpha_{2\rkn-i}})=N^\GLnn_{2i+2\rkn-k-1,2i+2\rkn-k+1}$, which are \emph{almost simple} root group.
Hence, $\chixmt\circ\Sigma_k^{-1}\rest_{N_{\GLnn, k, I}}$ is a character of $N_{\GLnn, k, I}$ that is supported on the almost simple root groups.

Similarly, $N_{\GLnn, k, II}:=\Sigma_k(N_\GLnn\cap \acutee\LL^k)$ is the subgroup of $N_\GLnn$ consisting of the matrices whose
off-diagonal entries are zero at rows $2\rkn-k+2i+1$, $i=0,\dots,\frac{[k-1]}{2}$.
The simple root groups $N_{\alpha_i}$ are stable under $\Sigma_k$ when $i<2\rkn-k$ and $\Sigma_k(N_{\alpha_i})=N^\GLnn_{2i+k-2\rkn,2i+k-2\rkn+2}$
(almost simple root group) when $2\rkn-k\leq i<2\rkn-[\frac{k+1}{2}]$.
Thus, $\chixmt\circ\Sigma_k^{-1}\rest_{N_{\GLnn, k, II}}$ is a character supported on the simple root groups $N_{\alpha_i}$ ($i<2\rkn-k$) and the almost simple root groups $N^\GLnn_{2\rkn-k+2i,2\rkn-k+2i+2}$ ($i\geq 0$).

It is now clear that $\chixmt\circ\Sigma_k^{-1}$ is a character on $\Sigma_k(\check\LL^k)$ which is supported on the simple root groups
$N_{\alpha_i}$ ($i<2\rkn-k$) and the almost simple root groups $N^\GLnn_{2\rkn-k+i,2\rkn-k+i+2}$ ($i\geq 0$).
\end{proof}

\subsection{Key facts used in the proof}

\begin{lemma}\label{lem: Lproperty}
\begin{enumerate}
\item \label{claim: normalL} For $k=0,\dots,4\rkn-2$,
$\hat\LL^k= \LL^{4\rkn-k-1}\ltimes\check\LL^k$ and $\hat\LL^{k+1}=\overline\LL^k\ltimes \check\LL^k$.
Also $\acutee\LL^k=\overline\LL^k\ltimes\hat\LL^k=\LL^{4\rkn-k-1}\ltimes\hat\LL^{k+1}$.
\item \label{claim: equivL} For any $m\in X_\GLnn$ and $\bar n\in N_\GLnn^t$ we have $\chixmt(m\bar n)=\chixmt(m)\chixmt(\bar n)$.
\item \label{claim: charL} For $k=0,\dots,4\rkn-1$, $\chixmt\rest_{\hat\LL^k}$ is a character of $\hat\LL^k$.
\item \label{claim: chitomd} For $k=1,\ldots,2\rkn$ we have $\chixmt\rest_{\hat\LL^k}=\altpsi_{tt'}\rest_{\hat\LL^k}$ for any
$t\in T_{\GLnn}$ and $t'=\diag(t'_1,\dots,t'_{2\rkn})\in T_{\GLnn}$ such that $t'_i=t'_{i+1}$ for all $i<2\rkn-\frac k2$.
In particular, when $t\in \soment$, $\chixmt\rest_{\zigzag}=\charzigzag$.
\end{enumerate}
\end{lemma}

\begin{proof}
By \eqref{eq: k4n-k} it suffices to prove parts \ref{claim: normalL} and \ref{claim: charL} for $k<2\rkn$.
Part~\ref{claim: normalL} follow from Remark~\ref{r: Lshape}.
For part~\ref{claim: equivL}, write $m=n\bar n_1$ with $n\in N_\GLnn$ and $\bar n_1\in N_\GLnn^t$, then
$\chixmt(m\bar n)=\chixmt(n)\psi_{N_\GLnn}^{-1} (\bar n_1^t\bar n^t)=\chixmt(m)\chixmt(\bar n)$.

Let $x_1,x_2\in \hat\LL^{k+1}$.
Write $x_j=y_jz_j$ with $y_j\in \check\LL^k$ and $z_j\in \overline\LL^k$ for $j=1,2$. Then by part~\ref{claim: equivL}
\begin{multline*}
\chixmt(x_1x_2)=\chixmt(y_1z_1y_2z_2)=\chixmt(y_1[z_1,y_2]y_2z_1z_2)\\=\chixmt(y_1[z_1,y_2]y_2)\chixmt(z_1)\chixmt(z_2),
\end{multline*}
where $[z_1,y_2]=z_1y_2z_1^{-1}y_2^{-1}\in\check\LL^k$ by part~\ref{claim: normalL}.
Recall that $\chixmt$ is a character on $\check\LL^k$. By Remark~\ref{r: Lshape}, $\Sigma_k([z_1,y_2])$ is not supported
on any simple root or almost simple root of $N_\GLnn$. Thus  $\chixmt([z_1,y_2])=1$. We get from part~\ref{claim: equivL}:
\begin{multline*}
\chixmt(y_1[z_1,y_2]y_2)\chixmt(z_1)\chixmt(z_2)=\chixmt(y_1)\chixmt(y_2)\chixmt(z_1)\chixmt(z_2)\\=
\chixmt(x_1)\chixmt(x_2).
\end{multline*}
This shows that $\chixmt$ is a character on $\hat\LL^{k+1}$. It is also clear that $\chixmt$ is a character on $\hat\LL^0=N_\GLnn$.
Part \ref{claim: chitomd} is clear from the definition of $\chixmt$.
\end{proof}

Let $\acutee\M^k=\acutee\LL^k\cap \GLnn^\fix$ and define $\check\M^k$, $\M^k$ and $\overline\M^k$ similarly.
Note that for $k<2\rkn$
\begin{equation}\label{eq: compML}
\M^k= \LL^k\,; \,\overline\M^k=\overline\LL^k.
\end{equation}

Let $k<\rkn$. If $n\in\M^{4\rkn-2k-2}$ and $\bar n\in\oddL^k=\overline\M^{2k+1}$, then by Lemma~\ref{lem: Lproperty}
$$[\bar n,n]\in [\overline\M^{2k+1},\hat\M^{2k+1}]\cap[\hat\M^{2k+2},\M^{4\rkn-2k-2}]=\hat\M^{2k+1}\cap\hat\M^{2k+2}=\check\M^{2k+1}.$$
A simple calculation shows that
\begin{equation}\label{eq: chixmtequiv}
\chixmt([\bar n,n])=\psi(\sprod{\bar n}{n}_{2k+1,t})
\end{equation}
where $\sprod{\cdot}{\cdot}_{2k+1,t}$ are non-degenerate pairings
\[
\sprod{\bar n}{n}_{2k+1,t}:=-\sum_{j=1}^k\bar n_{2\rkn+1-j, 2\rkn-2k+j-1}(n_{2\rkn-2k+j-1,2\rkn-j}+d_jn_{2\rkn-2k+j-2,2\rkn+1-j}).
\]
Here $d_1=0$ and $d_2,\dots,d_k$ are immaterial constants (depending continuously on $t$).
Note that with respect to the standard bases of $\M^{4\rkn-2k-2}$ and $\oddL^k$, the determinant of $\sprod{\cdot}{\cdot}_{2k+1,t}$ is $\pm1$.

Similarly for $\bar n\in\evenL^k_{\circ}$ and $n\in  \M^{4\rkn-2k-1}$,
by Lemma~\ref{lem: Lproperty} $$[\bar n,n]\in [\overline\M^{2k},\hat\M^{2k}]\cap[\hat\M^{2k+1},\M^{4\rkn-2k-1}]\subset\hat\M^{2k}\cap\hat\M^{2k+1}=\check\M^{2k},$$ and
\begin{equation}\label{eq: chixmtequiv2}
\chixmt([\bar n,n])=\psi(\sprod{\bar n}{n}_{2k,t})
\end{equation}
where $\sprod{\cdot}{\cdot}_{2k,t}$ are non-degenerate pairings
\[
\sprod{\bar n}{n}_{2k,t}:=
-\sum_{j=1}^{k-1}\bar n_{2\rkn+1-j, 2\rkn-2k+j}(n_{2\rkn-2k+j,2\rkn-j}+d'_jn_{2\rkn-2k+j-1,2\rkn-j+1})
\]
where again $d'_1=0$ and $d'_2,\dots,d'_{k-1}$ are some immaterial constants (depending continuously on $t$).

\subsection{Proof of Lemma~\ref{L: Mreg2}}
First note that the condition \eqref{p: K0inv} is clearly preserved under
the operations $\phi\rightarrow\avg^{reg}_{\evenL^k, \psi_{\vs}}\circ\phi$, $\phi\rightarrow\avgt{\Imk_k}[f_k]\circ\phi$,
$\phi\rightarrow\avg_{\oddL^k}\circ\phi$, $k=1,\dots,\rkn-1$, if defined.

\begin{enumerate}
\item Suppose that $\phi(t,m)\in\evencls_k(T_\GLnn\times \GLnn)^{K_0}$ with $k=1,\dots,\rkn-1$.
\begin{enumerate}

\item
It follows from property \eqref{p: switcht} that
\begin{equation}\label{eq: switchavg}
\avgt{\Imk_k}[f_k] \circ \phi(t,m)=\int_{\Imk_k} \phi(t,t'm)\chi_k(t')f_k(t')\ dt'
\end{equation}
for a certain (unimportant) character $\chi_k$ of $\Imk_k$.
Thus by \eqref{p: K0inv}, the integrand on the right hand side is in $\cc{\Imk_k}$. In particular,
$\avgt{\Imk_k}[f_k]\circ \phi(t,m)$ is a convergent integral.

Property \eqref{p: switcht} for $\avgt{\Imk_k}[f_k] \circ \phi$ follows immediately.
On the other hand, property \eqref{p: equim} for $\avgt{\Imk_k}[f_k] \circ \phi$ follows immediately from
Lemma~\ref{lem: Lproperty} part \eqref{claim: chitomd}.
Thus $\avgt{\Imk_k}[f_k] \circ \phi\in\evencls_k(T_\GLnn\times \GLnn)^{K_0}$.

\item Let $\phi_1(t,m)=\avgt{\Imk_k}[f_k] \circ \phi(t,m)\in\evencls_k(T_\GLnn\times \GLnn)^{K_0}$.
By \eqref{eq: chixmtequiv2} and \eqref{p: equim} for any $c\in\evenL^k_{\circ}$ and $d\in\M^{4\rkn-2k-1}\subset\hat\M^{2k}$ we have
\[
\phi_1(t, cd)=\phi_1(t,[c,d]dc)=\psi(\sprod{c}{d}_{2k,t})\chixmt(d)\phi_1(t,c).
\]
Applying Lemma~\ref{L: elemC} we conclude that $c\mapsto\phi_1(t,c)$ is compactly supported
on $\evenL^k_{\circ}$ uniformly in $t\in\Omega_{T_\GLnn}$.
To show that the conditions of Lemma~\ref{lem: genstab} are satisfied for the definition of
$\avg^{reg}_{\evenL^k, \psi_{\vs}}\circ\avgt{\Imk_k}[f_k]\circ\phi$, we only need to observe that by \eqref{eq: switchavg},
$\avgt{\Imk_k}[f_k]\circ\phi(t;\cdot)$ is left invariant under some $T_c\in \csgr(\Imk_k)$.
Lemma~\ref{lem: genstab} also gives the existence of $\Omega\in \csgr(\evenL^k)$ with the required properties.

We show that $\phi_2:=\avg^{reg}_{\evenL^k,\psi_{\vs}}\circ \phi_1\in\oddcls_k(T_\GLnn\times \GLnn)^{K_0}$.
Recall that $\hat\M^{2k+1}=\evenL^k\ltimes\check\M^{2k}$.
We check property \eqref{p: equim} for $\phi_2$ separately for $n\in\evenL^k$ and  $n\in \check\M^{2k}$.
In the first case, \eqref{p: equim} follows from the fact that $\psi_{\vs}=\chixmt^{-1}$ on $\evenL^k$.
In the second case note that for any $c\in\evenL^k$ we have $[c,n]\in \check\M^{2k}$ and $\chixmt([c,n])=1$
(as $c,n\in\hat\M^{2k+1}$) by Lemma~\ref{lem: Lproperty}.
Thus $\phi_1(t,cnm)= \phi_1(t,ncm)$ and
\[
\phi_2(t,nm)=\regint_{\evenL^k}\phi_1(t,cnm)\psi_{\vs}(c)\,dc=
\regint_{\evenL^k}\phi_1(t,ncm)\psi_{\vs}(c)\,dc=\chixmt(n)\phi_2(t,m).
\]
Finally, the property \eqref{p: switcht} follows from the fact that
$\prod_{i=k+1}^{\rkn-1}\Imk_i$ stabilizes the character $\psi_{\vs}$ on $\evenL^k$.
\end{enumerate}

\item Suppose that $\phi(t,m)\in\oddcls_k(T_\GLnn\times \GLnn)^{K_0}$ with $k=1,\dots,\rkn-1$.

Recall that for $c\in\oddL^k$ and $d\in  \M^{4\rkn-2k-2}\subset\hat\M^{2k+1}$, we have
$[c,d]\in \hat\M^{2k+1}$ and by \eqref{eq: chixmtequiv}, $\chixmt([c,d])=\psi(\sprod{c}{d}_{2k+1,t})$.
Thus for such $c$ and $d$ we have $\phi(t, cd)=\psi(\sprod{c}{d}_{2k+1,t})\chixmt(d)\phi(t,c)$.
It follows from Lemma~\ref{L: elemC} that $c\mapsto\phi(t,c)$ is compactly supported uniformly for $t\in\Omega_{T_\GLnn}$.

We check that $\phi'(t,m):=\avg_{\oddL^k}\circ \phi(t,m)\in\evencls_{k+1}(T_\GLnn\times \GLnn)^{K_0}$.
By Lemma~\ref{lem: Lproperty}, $\hat\M^{2k+2}=\oddL^k\ltimes\check\M^{2k+1}$.
Since $\Psi_t$ is trivial on $\oddL^k$, we are left to show that $\phi'(t,nm)=\chixmt(n)\phi'(t,m)$
for $n\in \check\M^{2k+1}$. For $c\in\oddL^k$ and $n\in \check\M^{2k+1}$, as both $c$ and $n$ lie in
$\hat\M^{2k+2}$, we have $[c,n]\in \check\M^{2k+1}$ and $\chixmt([c,n])=1$ by Lemma~\ref{lem: Lproperty}.
Thus $\phi(t,cnm)= \phi(t,ncm)$ and
\[
\phi'(t,nm)=\int_{\oddL^k}\phi(t,cnm)\,dc=\int_{\oddL^k}\phi(t,ncm)\,dc=\chixmt(n)\phi'(t,m).
\]
Condition \eqref{p: switcht} is clear because $T_\GLnn$ normalizes $\oddL^k$.
\end{enumerate}
This concludes the proof of Lemma~\ref{L: Mreg2}.
\section{Model transition for Langlands quotient: III} \label{sec: whitdesc}
Before proving Proposition~\ref{prop: main}, we first turn our attention to model transition from
$\model^{(N,\psi_N)}\LQ{\pi}$ to $\model^{(\alltriangle, \charalltriangle)}\LQ{\pi}$.
We will state the analog of Proposition~\ref{prop: main} in this setting (Proposition~\ref{prop: 14}).
The proofs of the two Propositions are repeated application of Fourier inversion. As the proof of Proposition~\ref{prop: 14} is simpler, we present it first.

For $W\in C^{\smth}(N\bs G,\psi_N)$ define
\[
\tranWAu(W)=\avg_{\subsubUbar}\circ W=\int_{\subsubUbar}W(\bar u\cdot)\ d\bar u.
\]
\index{$\tranWAu$} Recall that by Corollary \ref{cor: suppW} the integrand is compactly supported, hence the integral converges.
Moreover, by $\tranWAu$ is non-vanishing on $\model^{(N,\psi_N)}\LQ{\pi}$
(and in fact on any non-trivial subrepresentation of $\Ind(\WhitML(\pi'))$ for any $\pi'\in\Irr_{\gen}\GLnn$)
\cite[Theorem in \S1.3]{MR1671452} -- cf.~\cite[Remark 4.8]{1401.0198}.

\begin{proposition} \label{prop: 14}
For any $L\in\tclass(H\bs G)$ we have $\tranWAu\circ\tranLW=\tranLA$, namely
\begin{equation} \label{eq: LHS93}
\avg_{\subsubUbar}\circ\avg^{H,reg}_{N,\psi_{N}^{-1}}\circ L=
\avg_{\alltriangle,\charalltriangle^{-1}}^{H,reg}\circ L.
\end{equation}
In particular by Theorem \ref{thm: MWL}, for $\pi\in\Irr_{\temp,\meta}\GLnn$ we have
\[
\tranWAu=\tranLA\circ\tranMWL\text{ on }\model^{(N,\psi_N)}\LQ{\pi}
\]
and hence, $\tranLA$, $\tranLE$ and $\tranLV$ are non-zero on $\model^{(H,1)}\LQ{\pi}$.
\end{proposition}

\begin{proof}
By the remark above and Lemma \ref{L: stablemodel} the left-hand side of \eqref{eq: LHS93}
is a continuous functional on $\tclass(H\bs G)$
and it is enough to prove the proposition for $L\in \cc{H\bs G}$.

Define for $k=0,\dots,4\rkn-1$
\[
\upub^k:=\{\toU(u):u\in \symspace_{2\rkn}, u_{i,j}=0\text{ unless }2\rkn+i-j= k\}\simeq F^{\min\{[(k+1)/2],[(4\rkn-k+1)/2]\}}
\]
and
\[
\lowub^k:=w_U\upub^k w_U^{-1}\cap \bar U^{\dag}.
\]
\index{$\upub^k$, $\lowub^k$}
Thus, $\subsubUbar=\prod_{i=0}^{2\rkn-1}\lowub^i\subset U^{\dag}$.
For example when $\rkn=4$, the space $\subsubUbar$ consists of the image under $\toUbar$ of the matrices of the form
\[
\left(\begin{smallmatrix}
 &7&6&5&4&3&2&1\\
 & &7&6&5&4&3&2\\
 & & &7&6&5&4&3\\
 & & & &7&6&5&4\\
 & & & & &7&6&5\\
 & & & & & &7&6\\
 & & & & & & &7\\
 & & & & & & &
\end{smallmatrix}\right)
\]
where the digit $i$ corresponds to $\lowub^i$.
For this section, only the groups $\lowub^0,\dots,\lowub^{2\rkn-1}$ are relevant.
The other groups will be relevant for the next section.

For $k=0,\dots,4\rkn-1$ let
\begin{gather*}
\check\N^k=\langle N_\Levi, \upub^1,\dots,\upub^{4\rkn-k-2},\lowub^1,\dots,\lowub^{k-1}\rangle,\\
\hat\N^k=\langle N_\Levi, \upub^1,\dots,\upub^{4\rkn-k-1},\lowub^1,\dots,\lowub^{k-1}\rangle.
\end{gather*}
\index{$\check\N^k$, $\hat\N^k$}
These are unipotent subgroups of $G$.
In fact,
they are subgroups of
\[
\langle N_\Levi, \upub^1,\dots,\upub^{4\rkn-1-k},w_U\upub^1w_U^{-1},\dots,w_U\upub^kw_U^{-1}\rangle
\]
which is the unipotent radical of a semistandard Borel subgroup.
In particular, $\hat\N^{2\rkn}=\alltriangle$.

Let $\alg{X_0}$ be the closed subvariety $\alg{N}\alg{\bar U}$ of $G$. \index{$X_0$}
Thus $X_0=\{n\bar u\,:\, n\in N, \bar u\in\bar U\}$ and $\hat\N^k\subset X_0$ for all $k$.
Define a function $\chix$ on $X_0$ by
\[
\chix(n\bar u):=\psi_N(n)\newchar^{-1}(u),\,n\in N,\bar u\in \bar U.
\]
\index{$\chix$}
Note that $\chix\rest_{\alltriangle}=\charalltriangle$.
(For this section it is sufficient to consider $N\bar U^{\dag}$ instead of $X_0$.)

We write the left-hand side of \eqref{eq: LHS93} as
\begin{equation}\label{eq: iterwhit}
\avg_{\lowub^{2\rkn-1}}\circ
\avg_{\lowub^{2\rkn-2}}\circ\cdots\circ\avg_{\lowub^0}\circ\avg^H_{N,\psi_{N}^{-1}}\circ L.
\end{equation}

We apply Lemmas \ref{L: Fourier} and \ref{C: Fourier2} according to the parity of $k$ with $G_0=G$ , $H_0=H$ and $\chi_0=1$.
We claim that the data $A=\hat\N^k$, $A'=\hat\N^{k+1}$, $B=\check\N^k$, $C=\lowub^k$, $D=\upub^{4\rkn-1-k}$ and $\Psi=\chix^{-1}$
satisfy the conditions of Lemma \ref{L: Fourier} (resp., Lemma \ref{C: Fourier2}) if $k$ is odd (resp., even).
For that we need the following properties.

\begin{enumerate}
\item \label{claim: normalG} For $k=0,\dots,4\rkn-2$, $\hat\N^k=\upub^{4\rkn-1-k}\ltimes\check\N^k$,
$\hat\N^{k+1}=\lowub^k\ltimes\check\N^k$; both are subgroups of $\lowub^k\ltimes\hat\N^k$.
\item \label{claim: charG} For $k=0,\dots,4\rkn-1$, $\chix\rest_{\hat\N^k}$ is a character of $\hat\N^k$.
\item \label{claim: equivG} For $g\in X_0$, $c\in \bar U^{\dag}$, $\chix(gc)=\chix(g)$.
\item \label{claim: HinterG} For $k=0,\dots,2\rkn-1$, $\upub^{2k+1},\lowub^{2k+1}\subset H$ while
$\upub^{2k}\cap H=\lowub^{2k}\cap H=1$. $\chix$  is trivial on $H\cap\hat\N^k$.
\item \label{claim: prop5} For $k=1,\dots,2\rkn-1$ we have $[\lowub^{2k-1},\upub^{4\rkn-2k}]\subset \check\N^{2k-1}$.
Moreover, if $\bar u=(\bar u_{i,j})\in \lowub^{2k-1}$ and  $v=(v_{i,j})\in \upub^{4\rkn-2k}$ then
$\chix([\bar u,v])^{-1}=\psi(\sprod{\bar u}{v}'_{2k-1})$ where
\begin{equation}\label{eq: property5}
\sprod{\bar u}{v}'_{2k-1}=\begin{cases}
\sum_{l=1}^{2k-1} \bar u_{2\rkn+l,2\rkn-2k+l+1}v_{2\rkn-2k+l,2\rkn+l} &k=1,\dots,\rkn,\\
\sum_{l=1}^{4\rkn-2k-1} \bar u_{2k+l,l+1}v_{l,2k+l} &k=\rkn+1,\dots,2\rkn-1.
\end{cases}
 \end{equation}
\end{enumerate}
Identifying $\lowub^{2k-1}$ and $\upub^{4\rkn-2k}$ with $F^m$ where $m=\min\{k,2\rkn-k\}$ using the standard coordinates,
the determinant of the pairing $\sprod{\cdot}{\cdot}'_{2k-1}$ is $\pm 1$.

The proof of the first three properties is similar to that of Lemma~\ref{lem: Lproperty}.
The fourth property is clear from the definitions of $\upub^k$, $\lowub^k$ and $\chix$. An explicit calculation gives the last property.

It follows that when $k$ is even, $\lowub^k\cap H=1$ and $\upub^{4\rkn-1-k}\subset H$,
and the conditions of Lemma \ref{C: Fourier2} are satisfied.
When $k$ is odd,
the conditions of Lemma \ref{L: Fourier} are satisfied, with $\iota:C\rightarrow\widehat{D}$ given by
$\bar u\mapsto \psi^{-1}(\sprod{\bar u}{\cdot}'_{k})$ (see Remark \ref{rem: condFourier}). In both cases we get the relations
\begin{equation} \label{eq: iter2}
\avg_{\lowub^{k}}\circ\avg^H_{\hat{\N}^{k},\chix^{-1}} \circ L=\avg^H_{\hat{\N}^{k+1},\chix^{-1}}\circ L, \ \ k=0,\dots,4\rkn-2
\end{equation}
and by induction (since $N=\hat\N^0$ and $\chix\rest_N=\psi_N$)
\begin{equation} \label{eq: iter22}
\avg_{\lowub^{k}}\circ\dots\circ\avg_{\lowub^1}\circ\avg^H_{N,\psi_N^{-1}} \circ L=\avg^H_{\hat{\N}^{k+1},\chix^{-1}}\circ L.
\end{equation}
For $k=2\rkn-1$ we obtain the statement of the proposition.
\end{proof}

By \eqref{eq: relationetwon} we also conclude:

\begin{corollary}
For any $L\in\tclass^{\pm1}(H\bs G)$ we have
\[
\tranWAu\circ\tranLW(L)=
(\pm1)^{\rkn}\regint_{H_\Etwon\bs \Etwon}L(u\few   )\charEtwon^{-1}(u)\,du.
\]
\end{corollary}

\section{Proof of Proposition~\ref{prop: main}} \label{sec: prf of prop: main}

\subsection{Functional identity: I}
For $L\in\tclass(H\bs G)$ let $\MTF(L;t,m)$ be the function on $T_{\GLnn}\times \GLnn$ given by
\begin{equation}\label{eq: defMTF}
\MTF(L;t,m):=\FW_{\tranLW(L)}(t,m)
=\regint_{\bar{U}}\big(\regint_{H_N\bs N}L(n\levi(t)\bar{v}\levi(m))\psi_N^{-1}(n)
\,dn\big)\newchar(\bar v)\ d\bar v.
\end{equation}
\index{$\MTF$}
Recall from Lemma~\ref{L: Ureg} that the above regularized integral over $\bar U$ is defined.

Define the character $\psi_{N_\GLnn,t}$ on $N_\GLnn$ by
\[
\psi_{N_\GLnn,t}(n):=\psi_{N_\GLnn}(t n t^{-1})
\]
and let $\chinu$ be the function on $N_\Levi\ltimes\bar U$ given by
\[
\chinu(\levi(n)\bar u):=\psi_{N_\GLnn,t}(n)\newchar(\bar u)^{-1},\ \ n\in N_\GLnn,\ \bar u\in \bar U.
\]
\index{$\psi_{N_\GLnn,t}$, $\chinu$}
Recall the subgroup $\GLnn^\fix$ defined in \eqref{eq: defM1} and $N_\GLnn^\fix=N_\GLnn\cap \GLnn^\fix$.
Set $T_\Levi^\fix=\levi(T_\GLnn^\fix)$ and $N_\Levi^\fix=\levi(N_\GLnn^\fix)$.
Note that the restriction of $\chinu$ to $N_\Levi^\fix\ltimes\bar U$ is a character which is
$(T,N_\Levi^\fix\ltimes\bar U,(N_\Levi^\fix\cap N_\Levi^{\der})\ltimes\bar U^\dag)$-generic (or alternatively,
$(T_\Levi^\fix,N_\Levi^\fix\ltimes\bar U,(N_\Levi^\fix\cap N_\Levi^{\der})\ltimes\bar U)$-generic)
and trivial on $(N_\Levi^\fix\ltimes\bar U)\cap H$.

We can therefore define (using Remark~\ref{rem: partint} and Lemma \ref{lem: genstab})
\[
\regint_{(N_\Levi^\fix\ltimes\bar U)\cap H\bs (N_\Levi^\fix\ltimes\bar U)}L(ng)\chinu^{-1}(n)\ dn
\]
for $L\in\tclass(H\bs G)$.

By Lemma~\ref{lem: genstab}, the functionals
$\regint_{(N_\Levi^\fix\ltimes\bar U)\cap H\bs (N_\Levi^\fix\ltimes\bar U)}L(n)\chinu^{-1}(n)\ dn$
on $\tclass(H\bs G)$ are uniformly continuous for $t$ in a compact set.

\begin{lemma}\label{L: NUresult}
For any $L\in\tclass(H\bs G)$ we have for $t\in T_\GLnn$:
\begin{equation}\label{eq: NUresult}
\MTF(L;t,m)=\factor(t)
(\avg^{H,reg}_{N_\Levi^\fix\ltimes\bar U,\chinu^{-1}}\circ L)(\levi(m)).
\end{equation}
\end{lemma}

\begin{proof}
Once again, it is enough to prove it for $L\in \cc{H\bs G}$.
It is convenient to introduce the notation $\eva{x}\circ f(\cdot):=f(x\cdot)$. \index{$\eva{x}$}
From the definition \eqref{def: regint} of $\regint_{\bar{U}}$ we get
\begin{equation} \label{eq: MTFe}
\MTF(L;t,m)= (\avg^{\bar U_{1,1}\bar U^{\dag}}_{\bar U,\newchar}\circ\avg^{reg, \bar U^{\dag}}_{\bar U^{\dag}\bar U_{1,1},\newchar}\circ
\avg_{\bar{U}^{\dag}}\circ\eva{\levi(t)}\circ\avg^H_{N, \psi_N^{-1}}\circ L)(\levi(m)).
\end{equation}
For $k=1,\dots,2\rkn$, let $\bar U_{k,1}$ be the one-parameter subgroup $\{\toUbar(\lambda\cdot\one_{k,1}^{\symspace_{2\rkn}})\}$
of $\bar U$. \index{$\bar U_{k,1}$}
The integration over $(\bar U_{1,1}\bar U^{\dag})\bs \bar U$ can be written as an
iterated integral $\avg_{\bar U_{2\rkn,1}}\circ\ldots\circ\avg_{\bar U_{2,1}}$.
Recall the one-parameter root subgroups $N^\GLnn_{i,j}$, $N^\Levi_{i,j}=\levi(N^\GLnn_{i,j})$, $1\le i<j\le 2\rkn$.
For convenience we also set $N^\GLnn_{0,2\rkn}=N^\GLnn_{2\rkn,2\rkn}=N^\Levi_{0,2\rkn}=N^\Levi_{2\rkn,2\rkn}=1$.
For $k=0,\dots,2\rkn$, let
\begin{gather*}
\check\Nc^k=\langle \bar U^{\dag},\, N_\Levi^\fix,\, \bar U_{1,1},\dots,\bar U_{k,1},\,\,N^\Levi_{1,2\rkn},\dots,N^\Levi_{2\rkn-k-1,2\rkn}\rangle,\\
\hat\Nc^k=\langle \bar U^{\dag},\, N_\Levi^\fix,\, \bar U_{1,1},\dots,\bar U_{k,1},\,\,N^\Levi_{1,2\rkn},\dots,N^\Levi_{2\rkn-k,2\rkn}\rangle.
\end{gather*}
In particular,
\[
\hat\Nc^0=N_\Levi\ltimes\bar U^{\dag}=\hat\N^{4\rkn-1},\,\hat\Nc^{2\rkn}=N_\Levi^\fix\ltimes\bar U.
\]
We will prove the following statements (for $t=\diag(t_1,\dots,t_{2\rkn})$):
\begin{subequations}
\begin{gather}
\label{eq: barU_0}
\avg_{\bar{U}^{\dag}}\circ\eva{\levi(t)}\circ\avg^H_{N,\psi_N^{-1}}\circ
L=\abs{t_1}^{-2\rkn}\abs{\det t}^{2\rkn}\eva{\levi(t)}\circ\avg^H_{\hat\N^{4\rkn-1}, \chix^{-1}}\circ L,\\
\label{eq: barU_11}
\eva{\levi(t)}\circ\avg^H_{\hat\N^{4\rkn-1}, \chix^{-1}}\circ L=
\abs{t_1}^{\rkn}\abs{\det t}^{-2\rkn}\modulus_B^{\frac12}(\levi(t))\avg^H_{\hat\N^{4\rkn-1}, \chinu^{-1}}\circ L,\\
\label{eq: barU_k1}
\begin{split}&\avg_{\bar U_{k,1},\chinu^{-1}}\circ\dots\circ\avg_{\bar U_{2,1},\chinu^{-1}}\circ
\avg^{reg,\bar U^{\dag}}_{\bar U^{\dag}\bar U_{1,1},\newchar}\circ \avg^H_{\hat\N^{4\rkn-1},\chinu^{-1}}\circ L=\\
&\avg_{\bar U_{k,1},\chinu^{-1}}\circ\dots\circ
\avg_{\bar U_{1,1},\chinu^{-1}}\circ \avg^H_{\hat\N^{4\rkn-1},\chinu^{-1}}\circ L=
\avg^H_{\hat{\Nc}^k,\chinu^{-1}}\circ L,\ \ k=1,\dots,2\rkn.
\end{split}
\end{gather}
\end{subequations}
The Lemma would follow by combining the relations \eqref{eq: MTFe}, \eqref{eq: barU_0}, \eqref{eq: barU_11} and \eqref{eq: barU_k1} for $k=2\rkn$.

To prove \eqref{eq: barU_0} note that $\levi(t)$ normalizes $\bar U^{\dag}$, and the modulus function of $\levi(t)$
acting on $\bar U^{\dag}$ is $\abs{t_1}^{2\rkn}\abs{\det t}^{-2\rkn}$. Thus,
\[
\avg_{\bar{U}^{\dag}}\circ\eva{\levi(t)}=\abs{t_1}^{-2\rkn}\abs{\det t}^{2\rkn}\eva{\levi(t)}\circ\avg_{\bar{U}^{\dag}}.
\]
Recall the groups $\lowub^k$ defined in the proof of Proposition \ref{prop: 14}. Then
$\bar U^{\dag}=\prod_{k=0}^{4\rkn-2}\lowub^k$. Thus $\avg_{\bar{U}^{\dag}}=\avg_{\lowub^{4\rkn-2}}\circ\cdots\circ\avg_{\lowub^0}$.
As in the proof of Proposition \ref{prop: 14} we apply \eqref{eq: iter22}, this time for $k=4\rkn-2$, to get:
\[
\avg_{\bar U^{\dag}} \circ\avg^H_{N,\psi_N^{-1}}\circ L=\avg^H_{\hat\N^{4\rkn-1},\chix^{-1}}\circ L.
\]
The relation \eqref{eq: barU_0} follows.

To prove \eqref{eq: barU_11} we use the change of variable $n\mapsto \levi(t) n \levi(t)^{-1}$
and the fact that $L(\levi(t) g)=L(g)$.
An explicit computation shows that the Jacobian is $\abs{t_1}^{\rkn}\abs{\det t}^{-2\rkn}\modulus_B^{\frac12}(\levi(t))$.
We obtain
\begin{multline*}
\eva{\levi(t)}\circ\avg^H_{\hat\N^{4\rkn-1}, \chix^{-1}}\circ L(g)=
\int_{\hat\N^{4\rkn-1}\cap H\bs\hat\N^{4\rkn-1}}L(n\levi(t)g)\chix^{-1}(n)\,dn
\\=\abs{t_1}^{\rkn}\abs{\det t}^{-2\rkn}\modulus_B^{\frac12}(\levi(t))\int_{\hat\N^{4\rkn-1}\cap H\bs\hat\N^{4\rkn-1}}
L(ng)\chix^{-1}(\levi(t) n\levi(t)^{-1})\,dn.
\end{multline*}
Note that $\chinu\rest_{\hat\N^{4\rkn-1}}=\chix(\levi(t)\cdot \levi(t)^{-1})$.

To prove \eqref{eq: barU_k1} we will apply Lemma \ref{L: Fourier} or Lemma \ref{C: Fourier2},
according to the parity of $k$ with $G_0=G$, $H_0=H$ and $\chi_0=1$.
We claim the data $A=\hat\Nc^k$, $A'=\hat\Nc^{k+1}$, $B=\check\Nc^k$, $C=\bar U_{k+1,1}$, $D=N^\Levi_{2\rkn-k,2\rkn}$ and
$\Psi=\chinu^{-1}$ satisfy the conditions of Lemma \ref{L: Fourier} (resp. \ref{C: Fourier2})
if $k$ is odd (resp. even). This follows from the following properties.

\begin{enumerate}
\item \label{claim: normalNU} For $k=0,\dots,2\rkn-1$, $\hat\Nc^k=N^\Levi_{2\rkn-k,2\rkn}\ltimes\check\Nc^k$ and
$\hat\Nc^{k+1}=\bar U_{k+1,1}\ltimes \check\Nc^k$, both being subgroups of $\bar U_{k+1,1}\ltimes\hat\Nc^k$.
\item \label{claim: equivNU} For $g\in \bar N_\Levi\ltimes U$, $\bar u\in \bar U$, $\chinu(g\bar u)=\chinu(g)\chinu(\bar u)$.
\item \label{claim: charNU} For $k=0,\dots,2\rkn$, $\chinu\rest_{\hat\Nc^k}$ is a character of $\hat\Nc^k$ which is trivial on $H\cap \hat\Nc^k$.
\item \label{claim: HNU} When $k$ is odd, $\bar U_{k,1}\cap H=N^\Levi_{2\rkn-k,2\rkn}\cap H=1$.
When $k$ is even, $\bar U_{k,1}\subset H$ and $N^\Levi_{2\rkn-k,2\rkn}\subset H$.
\item For $k=1,\dots,\rkn$ we have $[\bar U_{2k,1},N^\Levi_{2\rkn-2k+1,2\rkn}]=\bar U_{1,1}\subset \check\Nc^{2k-1}$.
Moreover, if $x=\toUbar(\alpha\one_{2k,1}^{\symspace_{2\rkn}})\in \bar U_{2k,1}$ and
$y=\levi(I_{2\rkn}+\beta\epsilon_{2\rkn-2k+1,2\rkn})\in N^\Levi_{2\rkn-2k+1,2\rkn}$ then $\chinu([x,y])^{-1}=\psi(\alpha\beta)$.
\end{enumerate}

Indeed, the proof of the first three properties is similar to that of Lemma~\ref{lem: Lproperty}.
We omit the details.
The fourth property is clear, while the last one follows from a simple computation.

Applying Lemma \ref{L: Fourier} (resp. \ref{C: Fourier2}) if $k$ is odd (resp. even) we obtain the relation
\[
\avg_{\bar U_{k+1,1},\chinu^{-1}}\circ\avg^H_{\hat{\Nc}^k,\chinu^{-1}}\circ f=\avg^H_{\hat{\Nc}^{k+1},\chinu^{-1}}\circ f,\ \ k=0,\dots,2\rkn-1
\]
for any $f\in C_c^{\infty}(H\bs G)$. In particular, for $k=0$,
$\avg_{\bar U_{1,1},\chinu^{-1}}\circ\phi$ is absolutely convergent where $\phi=\avg^H_{\hat\N^{4\rkn-1},\chinu^{-1}}\circ f$. Thus by Lemma~\ref{L: stableprop} and the fact  $\chinu\rest_{\bar U}=\newchar^{-1}$, we get the first equation in \eqref{eq: barU_k1}.
The second relation in \eqref{eq: barU_k1} follows from the above identity by induction on $k$.
\end{proof}

\subsection{Functional identity: II}
We now prove Proposition~\ref{prop: main}.
Recall that for $t\in T_\GLnn$:
\begin{multline}\label{eq: defMTEE}
\tranWAv^{t,\auxf}\circ\tranLW(L)(e)=
\\\factor(t)^{-1}\avg_{\oddL^{\rkn-1}}\circ\avg^{reg}_{\evenL^{\rkn-1},\psi_{\vs}}
\circ \avgt{\Imk_{\rkn-1}}[f_{\rkn-1}]\circ\cdots
\circ\avg_{\oddL^1}\circ\avg^{reg}_{\evenL^1, \psi_{\vs}}\circ\avgt{\Imk_1}[f_1]\circ\MTF(L;t,e).
\end{multline}

We continue to assume that $L\in \cc{H\bs G}$.
Observe $\chix(\levi(tnt^{-1}))=\chixmt(n)$ for $n\in N_\GLnn$ and $\chix(\bar u)=\newchar^{-1}(\bar u)$ for $\bar u\in \bar U$.
Thus by Lemma~\ref{L: NUresult},
$$
\MTF(L;t,m)=\factor(t)\int_{(H\cap N_\Levi^\fix)\bs N_\Levi^\fix}\int_{(\bar U\cap H)\bs\bar U}
L(\bar u n\levi(m))\newchar(\bar u)\chixmt^{-1}( n)\ d\bar u\ dn.
$$

Let $H_\GLnn^\fix=H_\GLnn\cap \GLnn^\fix$. \index{$H_\GLnn^\fix$} Define
\begin{equation}\label{eq: defLt}
\Lambda(L;m):= \avg^H_{\bar U,\newchar}\circ L(\levi(m))=
\int_{H_{\bar U}\bs \bar U}L(\bar u \levi(m))\newchar( \bar u)\,d\bar u
\end{equation}
so that
\begin{equation}\label{eq: useLt}
\MTF(L;t,m)= \factor(t)\avg^{H_\GLnn^\fix}_{N_{\GLnn}^\fix,\chixmt^{-1}}\circ \Lambda(L;m).
\end{equation}

\begin{lemma} \label{lem: fixcompact}
The restriction of $\Lambda(L;\cdot)$ to $\GLnn^\fix$ belongs to $\cc{H_\GLnn^\fix\bs\GLnn^\fix,\chi_0}$
where $\chi_0=(\modulus_{H_P}\modulus_P^{-1})\circ\levi\rest_{H_\GLnn^\fix}$.
\end{lemma}

\begin{proof}
 Clearly, $\Lambda(L;\cdot)$ is $H_\GLnn^\fix$-equivariant.
Since $H_{\bar P}\bs H$ is compact, the restriction of $L$ to $\bar P$ is compactly supported modulo $H_{\bar P}$.
It follows that $\Lambda(L;\cdot)$ is compactly supported modulo $H_\Levi$.
It remains to observe that $H_\GLnn^\fix\bs\GLnn^\fix$ is closed in $H_\GLnn\bs\GLnn$, or equivalently,
that the orbit of $E=\diag(1,-1,\dots,1,-1)$ under conjugation by $\GLnn^\fix$ is closed. However, it is easy to see that this orbit
consists of the matrices of the form
$\left(\begin{smallmatrix}1&x&0\\0&A&0\\0&y&-1\end{smallmatrix}\right)$
where $A\in\GL_{2(\rkn-1)}$ is an involution with characteristic polynomial $(\lambda^2-1)^{\rkn-1}$ and $x,y$ are row vectors
of size $2(\rkn-1)$ such that $x(A+I_{2(\rkn-1)})=y(A-I_{2(\rkn-1)})=0$. This is clearly a closed subset of $\GLnn$.
In particular, $H_\GLnn^\fix\bs\GLnn^\fix$ is an affine variety.
\end{proof}

For $f_k\in \cc{\Imk_k}$, we let $f'_k(t)=f_k(t)\factor(t)$. From \eqref{eq: defMTEE} and \eqref{eq: useLt},
$\tranWAv^{t,\auxf}\circ\tranLW(L)(e)$ is equal to
\[
\avg_{\oddL^{\rkn-1}}\circ\avg^{reg}_{\evenL^{\rkn-1},\psi_{\vs}}
\circ \avgt{\Imk_{\rkn-1}}[f'_{\rkn-1}]\circ\cdots
\circ\avg_{\oddL^1}\circ\avg^{reg}_{\evenL^1, \psi_{\vs}}\circ\avgt{\Imk_1}[f'_1]\circ
\avg^{H_\GLnn^\fix}_{N_{\GLnn}^\fix,\chixmt^{-1}}\circ\Lambda(L;e).
\]
(Note that $\avg^{H_\GLnn^\fix}_{N_{\GLnn}^\fix,\chixmt^{-1}}\circ\Lambda(L;m)$ is a function of two variables $(t,m)$
through the dependence of $\chixmt$ on $t$.)
We will show that for any $L\in \cc{H\bs G}$ and $t\in\soment$ we have
\begin{multline}\label{eq: Efinal}
\avg_{\oddL^{\rkn-1}}\circ\avg_{\evenL^{\rkn-1},\psi_{\vs}}\circ\avgt{\Imk_{\rkn-1}}[f'_{\rkn-1}]\circ\cdots
\circ\avg_{\oddL^1}\circ\avg_{\evenL^1, \psi_{\vs}}\circ\avgt{\Imk_1}[f'_1]\circ
\avg^{H_\GLnn^\fix}_{N_{\GLnn}^\fix,\psi_{N_\GLnn,t}^{-1}}\circ\Lambda(L;e)\\=c(\auxf)
 \avg^{H_\GLnn^\fix}_{\zigzag, \charzigzag^{-1}}\circ \Lambda(L;e).
\end{multline}
Taking into account the definition \eqref{eq: defLt} of $\Lambda(L;\cdot)$ and \eqref{eq:charEtwon} we get
\[
\tranWAv^{t,\auxf}\circ\tranLW(L)= c(\auxf)\avg^H_{\Etwon,\charEtwon^{-1}}\circ L( e)
=c(\auxf)\int_{H_\Etwon\bs \Etwon}L(u)\charEtwon^{-1}(u)\,du
\]
which is the statement of Proposition~\ref{prop: main}.

It remains to prove \eqref{eq: Efinal}.
Recall the definition of $\hat\M^k$ in \S\ref{sec: defLk}.
By Lemma~\ref{lem: Lproperty}, $\chixmt\rest_{\hat\M^k}$ is a character of $\hat\M^k$.
Also recall that $\hat\M^{2\rkn}=\zigzag$ and $\chixmt\rest_{\zigzag}=\charzigzag$ for $t\in\soment$.
Thus from Lemma \ref{lem: fixcompact}, \eqref{eq: Efinal} is the case $k=\rkn-1$ of the last part of the Lemma below.

\begin{lemma}\label{L: iter4}
Assume $f\in \cc{H_{\GLnn}^\fix\bs\GLnn^\fix,\chi_0}$. Then for any $k=1,\dots,\rkn-1$ and $t\in T_\GLnn$:
\begin{enumerate}
\item $\avg^{H_\GLnn^\fix}_{\hat{\M}^{2k},\chixmt^{-1}}\circ f$ is invariant under $t\mapsto tt'$ where $t'\in\Imk_k$.
\item $\avg_{\evenL^k,\psi_{\vs}}\circ \avg^{H_\GLnn^\fix}_{\hat{\M}^{2k},\chixmt^{-1}}\circ f=
\avg^{H_\GLnn^\fix}_{\hat{\M}^{2k+1},\chixmt^{-1}}\circ f$.
\item $\avg_{\oddL^k}\circ \avg^{H_\GLnn^\fix}_{\hat{\M}^{2k+1},\chixmt^{-1}}\circ f=
\avg^{H_\GLnn^\fix}_{\hat{\M}^{2k+2},\chixmt^{-1}}\circ f$.
\end{enumerate}
Thus,
\begin{multline*}
\avg_{\oddL^k}\circ\avg_{\evenL^k,\psi_{\vs}}\circ\avgt{\Imk_k}[f'_k]\circ\cdots
\circ\avg_{\oddL^1}\circ\avg_{\evenL^1, \psi_{\vs}}\circ\avgt{\Imk_1}[f'_1]\circ
\avg^{H_\GLnn^\fix}_{N_{\GLnn}^\fix,\chixmt^{-1}}\circ f\\=\big(\prod_{i=1}^k\int_{\Imk_i}f'_i(t')\ dt'\big)
\avg^{H_\GLnn^\fix}_{\hat{\M}^{2k+2},\chixmt^{-1}}\circ f.
\end{multline*}
\end{lemma}

\begin{proof}
The first part follows from Lemma \ref{lem: Lproperty} part \ref{claim: chitomd}.
For the second and third parts we will apply Lemmas \ref{C: Fourier2} and \ref{L: Fourier} respectively
with $G_0=\GLnn^\fix$, $H_0=H_\GLnn^\fix$ and $\chi_0$ as above.
Observe that for $k=1,\dots,\rkn-1$
\begin{equation}\label{eq: HMcap}
\oddL^{k-1},\M^{2k-1}\subset H_{\GLnn}^\fix,\,\,\,
\evenL^k\cap H_{\GLnn}^\fix=\M^{2k}\cap H_{\GLnn}^\fix=1.
\end{equation}
We claim that the data  $A=\hat\M^{2k}$, $A'=\hat\M^{2k+1}$, $B=\check\M^{2k}$, $C=\evenL^k$, $D= \M^{4\rkn-1-2k}$ and $\Psi=\chixmt$
satisfy the conditions in Lemma~\ref{C: Fourier2}.
This follows from Lemma~\ref{lem: Lproperty}.
Thus, the second part follows from Lemma~\ref{C: Fourier2}.

Similarly, for the third part, we apply Lemma~\ref{L: Fourier} to the data
$A=\hat\M^{2k+1}$, $A'=\hat\M^{2k+2}$, $B=\check\M^{2k+1}$, $C=\oddL^k$, $D= \M^{4\rkn-2-2k}$ and $\Psi=\chixmt$.
In this case by \eqref{eq: chixmtequiv}, $\bar n\mapsto \psi(\sprod{\bar n}{\cdot}_{2k+1,t})$ gives an isomorphism between $C$ and
$\widehat{D}$ preserving Haar measure.

Finally, the last part follows immediately  by induction on $k$.
Note that $\hat{\M}^2=N_{\GLnn}^\fix$.
\end{proof}

This concludes the proof of Proposition~\ref{prop: main}.

\section{Variants and complementary results} \label{sec: Mfactor}

We finish the paper with some auxiliary results
(Lemma \ref{lem: spcltrans} and Corollaries \ref{cor: Mreg} and \ref{cor: fornextpaper} below) which will be used in a sequel of this paper.
They are variants of the results of the previous sections and are proved in a similar way.
The results do not seem to fit a general pattern and are hard to motivate by themselves.
The only reason to include them here is that they are simple modifications of earlier results.

\subsection{}
Recall that $\mira_{\GLnn}'$ is the second mirabolic subgroup of $\GLnn$ consisting of matrices whose first column is $(1,0,\dots,0)$.
Let $T'_{\GLnn}=\{\diag(I_\rkn,t):t\in\GL_{\rkn}\text{ diagonal}\}$.

Consider the subspace $\spclWM$ of $\cc{N_\GLnn\bs\mira_{\GLnn}',\psi_{N_\GLnn}}$
consisting of the $W$'s such that\footnote{This condition is not automatic since the set $N_\GLnn\cdot T'_\GLnn\cdot\vs$ is not closed.}
$W\rest_{T'_\GLnn\ltimes \vs}\in \cc{T'_\GLnn\ltimes \vs}$.
For $W\in \spclWM$ define
\[
\tranwa(W)= \avg_{\vs,\psi_{\vs}}\circ\avg_{T'_\GLnn,\factor^{-1}\abs{\det}^{\rkn}}\circ W.
\]
Note that $\tranwa(W)\in C^{\smth}(\zigzag\bs\mira_{\GLnn}',\charzigzag)$ since $T'_{\GLnn}$ stabilizes the character
$\charzigzag\rest_{\zigzag\cap N_\GLnn}=\psi_{N_\GLnn}\rest_{\zigzag\cap N_\GLnn}$.
Recall that by Lemma~\ref{L: cuspidal} (or rather, its analogue for $\mira_{\GLnn}'$) if $W\in \cc{N_\GLnn\bs\mira_{\GLnn}',\psi_{N_\GLnn}}$
then $\avg^{N_\GLnn}_{\mira_{\GLnn}'\cap H_\GLnn}(W)\in \cc{\mira_{\GLnn}'\cap H_\GLnn\bs \mira_{\GLnn}'}$.
Consider
\[
\avg_{N_\GLnn,\psi_{N_\GLnn}^{-1}}^{H_\GLnn}:\cc{\mira_{\GLnn}'\cap H_\GLnn\bs \mira_{\GLnn}'}\rightarrow
C^{\smth}(N_\GLnn\bs\mira_{\GLnn}',\psi_{N_\GLnn})
\]
and
\[
\avg^{N_\GLnn}_{\mira_{\GLnn}'\cap H_\GLnn}:\cc{N_\GLnn\bs\mira_{\GLnn}',\psi_{N_\GLnn}}\rightarrow \cc{\mira_{\GLnn}'\cap H_\GLnn\bs \mira_{\GLnn}'}.
\]
Then $\avg_{N_\GLnn,\psi_{N_\GLnn}^{-1}}^{H_\GLnn}\circ\avg^{N_\GLnn}_{\mira_{\GLnn}'\cap H_\GLnn}=\id$ by
(the analogue of) Lemma \ref{L: feglnn}.

In this subsection we prove the following functional identity.

\begin{proposition}\label{prop: factor}
We have
\[
\tranwa=\avg^{H_\GLnn}_{\zigzag,\charzigzag^{-1}}\circ\avg^{N_\GLnn}_{\mira_{\GLnn}'\cap H_\GLnn}\text{ on }\spclWM,
\]
or equivalently
\[
\tranwa\circ\avg_{N_\GLnn,\psi_{N_\GLnn}^{-1}}^{H_\GLnn}=\avg^{H_\GLnn}_{\zigzag,\charzigzag^{-1}}\text{ on }
\avg^{N_\GLnn}_{\mira_{\GLnn}'\cap H_\GLnn}(\spclWM).
\]
\end{proposition}

We first write $\tranwa$ as an iterated integral.
Let $T^k$ be the image of the co-character
\[
\tau_k(x):=\diag(1,\ldots,1,\overbrace{x,\ldots,x}^{k\text{ times}})\in T_\GLnn,
\]
so that $T'_\GLnn=\prod_{k=1}^\rkn T^k$.

\begin{lemma} \label{lem: manipIM}
For any $W\in\spclWM$,
\begin{equation*}
\tranwa(W)=\avg_{\oddL^{\rkn-1}}\circ\avg_{T^\rkn,\tilde\Xi_\rkn}\circ\avg_{\evenL^{\rkn-1},\psi_{\vs}}\circ
\dots\circ\avg_{\oddL^0}\circ\avg_{T^1,\tilde\Xi_1}\circ\avg_{\evenL^0,\psi_{\vs}}\circ W
\end{equation*}
where $\tilde\Xi_k(\tau_k(x))=\abs{x}^{\rkn k-k^2+k-1}$, $k=1,\dots,\rkn$.
\end{lemma}

\begin{proof}
Using \eqref{eq: defvs} we write the (absolutely convergent) integration $\tranwa(W)$ over $\vs$  as an iterated integral over
 $\evenL^0,\oddL^0,\dots,\evenL^{\rkn-1},\oddL^{\rkn-1}$, and
write the integration over $T'_{\GLnn}$ as an iterated integral over $T^k$ ($k=1,\dots,\rkn$).
Thus, $\tranwa(W)$ equals
\[\avg_{\oddL^{\rkn-1}}\circ\avg_{\evenL^{\rkn-1},\psi_{\vs}}\circ\cdots\circ\avg_{\oddL^0}
\circ\avg_{\evenL^0,\psi_{\vs}}\circ\avg_{T^\rkn,\Xi_\rkn}\circ\cdots\circ\avg_{T^1,\Xi_1}\circ W
\]
where for $t=\tau_k(x)\in T^k$
\[
\Xi_k(t):=\factor^{-1}(t)\abs{\det t}^{\rkn}=\abs{x}^{\rkn k-\frac{k+k^2}2},\ \ k=1,\dots,\rkn.
\]
(We recall that $\psi_{\vs}$ is trivial on $\oddL^i$, $i=0,\dots,\rkn-1$.)

It is easy to see that for $i=0,\dots,k-1$, $T^k$ stabilizes $(\evenL^i,\psi_{\vs}\rest_{\evenL^i})$ under conjugation.
Of course $T^k$ normalizes $\oddL^i$ for all $i$.
For $t\in T^k$ let $\Xi'_{k,i}(t)=\modulus_{T^k;\oddL^i}(t^{-1})$ and $\Xi''_{k,i}(t)=\modulus_{T^k;\evenL^i}(t^{-1})$.
By a change of variable $r\mapsto t^{-1}rt$ we have
\begin{multline*}
\avg_{\oddL^i}\circ\avg_{T^k,\Xi}\circ f(g)=
\int_{\oddL^i}\int_{T^k} f(trg)\Xi(t)\,dt\,dr
\\=\int_{\oddL^i}\int_{T^k} f(rtg)\Xi(t)\Xi'_{k,i}(t)\,dt\,dr=
\avg_{T^k,\Xi\Xi'_{k,i}}\circ\avg_{\oddL^i}\circ f(g)
\end{multline*}
and (for $i<k$)
\begin{multline*}
\avg_{\evenL^i,\psi_{\vs}}\circ\avg_{T^k,\Xi}\circ f(g)=
\int_{\evenL^i}\int_{T^k} f(trg)\psi_{\vs}(r)\Xi(t)\,dt\,dr
\\=\int_{\evenL^i}\int_{T^k} f(rtg)\psi_{\vs}(r)\Xi(t)\Xi''_{k,i}(t)\,dt\,dr=
\avg_{T^k,\Xi\Xi''_{k,i}}\circ\avg_{\evenL^i,\psi_{\vs}}\circ f(g)
\end{multline*}
for any function $f$ on $G$ and $\Xi$ on $T^k$, provided that the double integral converges absolutely.

The lemma follows by making the above changes of variables for all pairs $(i,k)$ with $i<k-1$
(first type) and $i<k$ (second type) and noting that
\[
\tilde\Xi_k=\Xi_k\prod_{i=0}^{k-2}\Xi'_{k,i}\prod_{i=0}^{k-1}\Xi''_{k,i}
\]
on $T^k$. Indeed, $\prod_{i=0}^{k-2}\Xi'_{k,i}(t)\prod_{i=0}^{k-1}\Xi''_{k,i}(t)$ is the modulus function
of $t^{-1}$ acting on the group $\hat\LL^{2k-1}\cap N_{\GLnn}^t$, which is easy to compute.
\end{proof}

Recall from definition \eqref{eq: defchixmt} (with $t=e$) that
\[
\altpsi_e(m\overline m)=\psi_{N_{\GLnn}}(m)\psi_{N_{\GLnn}}(\overline m^t)^{-1},\ \ m\in N_\GLnn,\ \overline m\in N_\GLnn^t.
\]
In particular, $\altpsi_e\rest_{N_\GLnn}=\psi_{N_{\GLnn}}$ and $\altpsi_e\rest_{\vs}=\psi_{\vs}^{-1}$.

We apply the Lemmas in Appendix \ref{a: int} for $G_0=\mira_{\GLnn}'$, $H_0=H_{\GLnn}\cap\mira_{\GLnn}'$ and $\chi_0=1$.
Observe that the modulus function of $\tau_k(x)^{-1}$ acting on $\hat\LL^{2k-1}\cap H_\GLnn\bs \hat\LL^{2k-1}$ by conjugation is
$\abs{x}^{\rkn k-k^2+k}$, i.e.
\begin{equation}\label{eq: modmatch}
\tilde\Xi_k(\tau_k(x)^{-1})=\abs{x}\modulus_{T^k; \hat\LL^{2k-1}\cap H_\GLnn\bs\hat\LL^{2k-1}}(\tau_k(x)).
\end{equation}

Write $ \LL^{2k}= \LL_\circ^{2k}\times N^{\GLnn}_{k,k+1}$ where
\[
\LL_\circ^{2k}=\{I_{2\rkn}+u:\,  u_{i,j}=0\text{   if   }i+j\not=2k+1 \text{  or  }j\leq i+1\}.
\]
\begin{lemma} \label{L: iterIM}
For any $f\in C_c^{\infty}(\mira_{\GLnn}'\cap H_\GLnn\bs \mira_{\GLnn}')$ and $k=1,\dots,\rkn$ we have
\begin{enumerate}
\item \label{claim: evencase}
$\avg_{\evenL^{k-1}, \altpsi_e^{-1}}\circ\avg^{H_\GLnn}_{\hat\LL^{2k-2},\altpsi_e^{-1}} \circ f
=\avg^{H_\GLnn}_{\hat\LL^{2k-1},\altpsi_e^{-1}} \circ f$.
\item \label{claim: tavg}  $\avg_{T^k,\tilde\Xi_k} \circ \avg^{H_\GLnn}_{\hat\LL^{2k-1},\altpsi_e^{-1}} \circ f=
\avg^{H_\GLnn}_{\LL^{4\rkn-2k}_\circ\ltimes\check\LL^{2k-1},\altpsi_e^{-1}} \circ f$.
\item \label{claim: rkodd}
$\avg_{\oddL^{k-1}}\circ \avg^{H_\GLnn}_{\LL^{4\rkn-2k}_\circ\ltimes\check\LL^{2k-1},\altpsi_e^{-1}} \circ f
=\avg^{H_\GLnn}_{\hat\LL^{2k},\altpsi_e^{-1}} \circ f$.
\end{enumerate}
Therefore for $k=0,\dots,\rkn$,
\begin{equation}\label{eq: factoriter}
\avg_{\oddL^{k-1}}\circ\avg_{T^k,\tilde\Xi_k}\circ\avg_{\evenL^{k-1},\psi_{\vs}}\circ
\dots\circ\avg_{\oddL^0}\circ\avg_{T^1,\tilde\Xi_1}\circ\avg_{\evenL^0,\psi_{\vs}}\circ \avg^{H_\GLnn}_{N_\GLnn,\psi_{N_\GLnn}^{-1}}\circ f
=\avg^{H_\GLnn}_{\hat\LL^{2k},\altpsi_e^{-1}} \circ f.
\end{equation}
\end{lemma}

\begin{proof}
We first recall (similarly to \eqref{eq: HMcap}) that for all $k$
\begin{equation}\label{eq: HMcap2}
\oddL^k,\LL^{2k-1}\subset H_{\GLnn}\cap\mira_{\GLnn}',\,\, \evenL^k\cap H_{\GLnn}=\LL^{2k}\cap H_{\GLnn}=1.
\end{equation}

For part \ref{claim: evencase}, use Lemma \ref{C: Fourier2}, with the data $A=\hat\LL^{2k-2}$, $A'=\hat\LL^{2k-1}$,
$\tilde A=\acutee\LL^{2k-2}$, $B=\check\LL^{2k-2}$,
$C=\evenL^{k-1}$, $D= \LL^{4\rkn-2k+1}$ and $\Psi=\altpsi_e^{-1}$. From \eqref{eq: HMcap2} and Lemma \ref{lem: Lproperty},
the conditions of Lemma \ref{C: Fourier2} are satisfied.

For part \ref{claim: tavg}, use Lemma \ref{L: Fourier2} with the data $A=\hat\LL^{2k-1}$,
$B=\LL^{4\rkn-2k}_\circ\ltimes\check\LL^{2k-1}$, $C=T^k$,
$D=N^{\GLnn}_{2\rkn-k, 2\rkn-k+1}$ and $\Psi=\altpsi_e^{-1}$.
The conditions in Lemma \ref{L: Fourier2} are easy to verify in our case.
In particular, for $c=\tau_k(x^{-1})\in C$, $d=I_{2\rkn}+y\one_{2\rkn-k,2\rkn-k+1}\in D$ and $b\in B$, we have $\altpsi_e^{-1}(cbdc^{-1})=\altpsi_e^{-1}(bd)\psi((1-x)y)$.
Thus we apply Lemma \ref{L: Fourier2} with $\tau(x)=\tau_k(x^{-1})$ and $\iota(d)=y$. Taking into account \eqref{eq: modmatch}, we get part \ref{claim: tavg}.

For part \ref{claim: rkodd}, use Lemma~\ref{L: Fourier},
with the data $A= \LL^{4\rkn-2k}_\circ\ltimes\check\LL^{2k-1}$,
$A'=\hat\LL^{2k}$, $B=\check\LL^{2k-1}$, $C=\oddL^{k-1}$, $D=\LL^{4\rkn-2k}_\circ$ and $\Psi=\altpsi_e^{-1}$.
It is easy to check that $C$ normalizes $A$.
From \eqref{eq: HMcap2} and Lemma~\ref{lem: Lproperty}, to check the remaining conditions of Lemma~\ref{L: Fourier}
we  observe that for $c\in C$ and $d\in\LL^{4\rkn-2k}_\circ$ we have $[c,d]\in \check\LL^{2k-1}$ with
\[
\altpsi_e^{-1}([c,d])=\psi(\sum_{i=1}^{k-1} c_{2\rkn-k+1+i,2\rkn-k-i+1}(d_{2\rkn-k-i,2\rkn-k+1+i}+d_{2\rkn-k-i+1,2\rkn-k+i}))
\]
(where $d_{2\rkn-k,2\rkn-k+1}=0$). Thus $c\mapsto \altpsi_e^{-1}([c,\cdot])$ gives an isomorphism $C\rightarrow\widehat{D}$ preserving Haar measure.
Applying Lemma~\ref{L: Fourier}, we get part \ref{claim: rkodd}.

Finally, the last part follows by induction since $\hat\LL^0=N_{\GLnn}$.
\end{proof}

To finish the proof of Proposition~\ref{prop: factor}  we apply the above Lemma to $f=\avg^{N_\GLnn}_{\mira_{\GLnn}'\cap H_\GLnn}\circ W$ with $W\in \spclWM$. Then $W=\avg_{N_\GLnn,\psi_{N_\GLnn}^{-1}}^{H_\GLnn}\circ f$.
Since $\hat\LL^{2\rkn}=\zigzag$ and  $\altpsi_e\rest_{\zigzag}=\charzigzag$, the proposition follows from the $k=\rkn$ case of
\eqref{eq: factoriter} and Lemma~\ref{lem: manipIM}.

\label{sec: box}

\subsection{}
Next we note the following simple lemma.
\begin{lemma}\label{lem: switch}
Let $G_1$ and $ G_2$ be two subgroups of $G_3$, $\chi$ a character of $G_1$, $\Xi$ a function of $G_2$.
Let $c\in G_3$ be such that $[b,c]\in G_1$ for all $b\in G_2$. Let $\chi_c$ be the function of $G_2$
given by $\chi_c(b):=\chi([b,c])$. (It is not necessarily a character.)
Then for any $f\in C(G_1\bs G_3, \chi)$ we have $\eva{c}\circ \avg_{G_2,\Xi}\circ f=\avg_{G_2,\Xi\chi_c}\circ\eva{c}\circ f$.
\end{lemma}

\begin{proof}
This is just a restatement of the identity
\[
\int_{G_2}f(bc\cdot)\Xi(b)\ db=\int_{G_2} f([b,c] cb\cdot)\Xi(b)\ db=\int_{G_2}f(cb\cdot)\Xi(b)\chi_c(b)\ db.\qedhere
\]
\end{proof}

Let $\diagelemGLn=\diag((-1)^{i-1})\in \GL_n$. Set $E_1:=\diag(\diagelemGLn^*,\diagelemGLn)\in\GLnn$ and $\elemu=\toU(\num E_1)$ with some parameter $\num\in F$.
Let $\elemv=\sm{I_\rkn}{\num\diagelemGLn w_\rkn}{}{I_\rkn}\in\GLnn$ so that $\few\elemu \few^{-1}=\levi(\elemv)$.

\begin{lemma}\label{lem: elemv}
For any $k<2\rkn-1$, $t\in T_{\GLnn}$ and $x\in\overline\LL^k$ we have $[x,\elemv]\in\check\LL^k$ and $\chixmt([x,\elemv])=1$.
Moreover, if $k=2\rkn-1$ then for any $t\in\soment$ and $x\in\overline\LL^k$ we have $[x,\elemv]\in\check\LL^k$ and $\chixmt([x,\elemv])=1$.
\end{lemma}

\begin{proof}
If $k<2\rkn-1$ then $x,\elemv\in\hat\LL^{k+1}$ and therefore $[x,\elemv]\in\check\LL^k$.
Also, $\chixmt([x,\elemv])=1$ since $\chixmt\rest_{\hat\LL^{k+1}}$ is a character.
In the case $k=2\rkn-1$ it is no longer true that $\elemv\in\hat\LL^{k+1}$. However,
$\elemv\in\hat\LL^k$ and thus $[\overline\LL^k,\elemv]\subset [\overline\LL^k,\hat\LL^k]\subset\hat\LL^k$
by Lemma~\ref{lem: Lproperty}.
Similarly, $\elemv\in\LL^{2\rkn}=\LL^{4\rkn-k-1}$ so that $[\overline\LL^k,\elemv]\subset[\hat\LL^{k+1},\LL^{4\rkn-k-1}]\subset\hat\LL^{k+1}$.
Thus $[\overline\LL^k,\elemv]\subset\hat\LL^k\cap\hat\LL^{k+1}=\check\LL^k$.
Moreover, explicit computation (using the particular form of $\elemv$ and $\chixmt$) shows that $\chixmt([x,\elemv])=1$.
\end{proof}

Let $\GLnn^\fixx\subset\GLnn^\fix$ be the kernel of the character $m\mapsto m_{1,1}$ of $\GLnn^\fix$.
Note that $\GLnn^\fixx$ contains any unipotent subgroup of $\GLnn^\fix$ and that
\begin{equation} \label{eq:eps2normfixx}
\elemv\text{ normalizes }\GLnn^\fixx.
\end{equation}

Let $\Xi_\elemv$ be the function on $T'_{\GLnn}$ defined by
\[
\Xi_\elemv(t)=\psi_{N_\GLnn}(t\elemv t^{-1}).
\]
Explicitly if $t=\diag(1,\ldots,1,t_1,\ldots,t_\rkn)$ then $\Xi_\elemv(t)=\psi((-1)^{\rkn+1}\num t_1^{-1})$.
Let
\[
\tran^{\elemv}(W)= \avg_{\vs,\psi_{\vs}}\circ\avg_{T'_\GLnn,\Xi_\elemv\factor^{-1}\abs{\det}^{\rkn}}\circ W
\]
for $W\in \spclWM$.

\begin{lemma} \label{lem: spcltrans}
For any $W\in \spclWM$ we have
\begin{equation}\label{eq: boxfactor}
\tran^{\elemv}(W)=\eva{\elemv}\circ\avg^{H_\GLnn}_{\zigzag,\charzigzag^{-1}}\circ \avg^{N_\GLnn}_{\mira_{\GLnn}'\cap H_\GLnn}\circ W.
\end{equation}
Thus, if $\pi\in\Irr_{\gen,\meta}\GLnn$ is unitarizable, for any $W\in \WhitM(\pi)$ such that $W\rest_{\mira'_{\GLnn}}\in\spclWM$ we have
\[
\tran^{\elemv}(W)=\eva{\elemv}\circ\avg^{H_\GLnn}_{\zigzag,\charzigzag^{-1}}\circ \avg^{N_\GLnn}_{\mira_{\GLnn}\cap H_\GLnn}\circ W.
\]
\end{lemma}

\begin{proof}
As in the proof of Proposition~\ref{prop: factor}, let $f=\avg^{N_\GLnn}_{\mira_{\GLnn}'\cap H_\GLnn}\circ W$; then $W=\avg_{N_\GLnn,\psi_{N_\GLnn}^{-1}}^{H_\GLnn}\circ f$.
As in Lemma~\ref{lem: manipIM} we get the following iterated integral expression of $\tran^{\elemv}(W)$:
$$\avg_{\oddL^{\rkn-1}}\circ
\avg_{T^\rkn,\Xi_{\elemv}\tilde\Xi_\rkn}\circ\avg_{\evenL^{\rkn-1},\psi_{\vs}}\circ
\dots\circ\avg_{\oddL^0}\circ
\avg_{T^1,\tilde\Xi_1}\circ\avg_{\evenL^0,\psi_{\vs}}\circ \avg^{H_\GLnn}_{N_\GLnn,\psi_{N_\GLnn}^{-1}}\circ f.
$$
Here we use the fact that $\Xi_{\elemv}(\prod_{i=1}^\rkn t_i)=\Xi_{\elemv}(t_\rkn)$ when $t_i\in T^i$ for $1\leq i\leq \rkn$.

Now use \eqref{eq: factoriter} for $k=\rkn-1$ and then part \ref{claim: evencase} of Lemma~\ref{L: iterIM} (for $k=\rkn$).
Recalling that $\altpsi_e\rest_{\vs}=\psi_{\vs}^{-1}$ we get
$$\tranwa\circ\avg_{N_\GLnn,\psi_{N_\GLnn}^{-1}}^{H_\GLnn}\circ f=\avg_{\oddL^{\rkn-1}}\circ
\avg_{T^\rkn,\Xi_{\elemv}\tilde\Xi_\rkn}\circ \avg^{H_\GLnn}_{\hat\LL^{2\rkn-1},\altpsi_e^{-1}} \circ f.$$
Next note $\elemv\in \hat\LL^{2\rkn-1}$. Moreover  for $t\in T^\rkn$, we have $[t,\elemv ]\in \hat\LL^{2\rkn-1}$
with $\altpsi_e([t,\elemv ])=\Xi_{\elemv}(t)\altpsi_e(\elemv)^{-1}$. From Lemma~\ref{lem: switch} we get
\[
\avg_{T^\rkn,\Xi_{\elemv}\tilde\Xi_\rkn}\circ\avg^{H_\GLnn}_{\hat\LL^{2\rkn-1},\altpsi_e^{-1}}
=\altpsi_e(\elemv)^{-1}\avg_{T^\rkn,\Xi_{\elemv}\tilde\Xi_\rkn}\circ\eva{\elemv}\circ\avg^{H_\GLnn}_{\hat\LL^{2\rkn-1},\altpsi_e^{-1}}
=\eva{\elemv}\circ\avg_{T^\rkn, \tilde\Xi_\rkn}\circ\avg^{H_\GLnn}_{\hat\LL^{2\rkn-1},\altpsi_e^{-1}}.
\]
From parts \ref{claim: tavg} and \ref{claim: rkodd} of Lemma~\ref{L: iterIM} for $k=\rkn$, to prove \eqref{eq: boxfactor} we are left to show that
\[
\avg_{\oddL^{\rkn-1}}\circ\eva{\elemv}\circ\avg^{H_\GLnn}_{\LL^{2\rkn}_\circ\ltimes\check\LL^{2\rkn-1},\altpsi_e^{-1}}=
\eva{\elemv}\circ\avg_{\oddL^{\rkn-1}}\circ\avg^{H_\GLnn}_{\LL^{2\rkn}_\circ\ltimes\check\LL^{2\rkn-1},\altpsi_e^{-1}}.
\]
From Lemma~\ref{lem: elemv}  (for $k=2\rkn-1$) we get that for any $r\in\oddL^{\rkn-1}$,
$[r,\elemv^{-1}]\in\check\LL^{2\rkn-1}$ with $\altpsi_e([r,\elemv^{-1}])=1$. The above relation follows from Lemma~\ref{lem: switch}.

When $W\in \WhitM(\pi)$, we also have $f=\avg^{N_\GLnn}_{\mira_{\GLnn}\cap H_\GLnn}\circ W$
by Proposition~\ref{prop: sametran}. The second assertion follows.
\end{proof}

\subsection{}
Define a character $\oldsubsubUbar$ on $\subsubUbar$ by $\oldsubsubUbar(\bar u)=\psi(-\num \bar u_{3\rkn,\rkn+1})$.
We show the following variant of Proposition~\ref{prop: 14}.
\begin{lemma} \label{lem: 14var}
For any $L\in\tclass(H\bs G)$ we have
\begin{equation} \label{eq: varbarU}
\avg_{\subsubUbar,\oldsubsubUbar}\circ\avg^{H,reg}_{N,\psi_{N}^{-1}}\circ L=
\eva{\elemu}\circ\avg_{\alltriangle,\charalltriangle^{-1}}^{H,reg}\circ L.
\end{equation}
\end{lemma}

\begin{proof}
Again by continuity, we only need to prove the identity for $L\in \cc{H\bs G}$.
By Corollary \ref{cor: suppW} we can write the left-hand side as an iterated integral:
$$
\avg_{\lowub^{2\rkn-1},\oldsubsubUbar}\circ\avg_{\lowub^{2\rkn-2}}\circ\dots\circ\avg_{\lowub^1}\circ\avg^H_{N,\psi_N^{-1}} \circ L.
$$
Here we note that $\oldsubsubUbar$ is trivial on all $\lowub^k$ with $k<2\rkn-1$.
Applying \eqref{eq: iter22} for $k=2\rkn-2$ we get that the above is
$$\avg_{\lowub^{2\rkn-1},\oldsubsubUbar}\circ\avg^H_{\hat{\N}^{2\rkn-1},\chix^{-1}}\circ L.$$
Note  that $\elemu\in \hat\N^{2\rkn-1}$ and $\chix(\elemu)=1$.
From property \eqref{eq: property5} and the statements immediately preceding it, (applied in case $k=\rkn$) we get for any
$\bar v \in\lowub^{2\rkn-1}$, $[\bar v, \elemu]\in \hat{\N}^{2\rkn-1}$ with $\chix([\bar v, \elemu])=\oldsubsubUbar(\bar v)$.
Thus from Lemma~\ref{lem: switch} the above equals:
$$\avg_{\lowub^{2\rkn-1},\oldsubsubUbar}\circ\eva{\elemu}\circ\avg^H_{\hat{\N}^{2\rkn-1},\chix^{-1}}\circ L=\eva{\elemu}\circ\avg_{\lowub^{2\rkn-1}}\circ\avg^H_{\hat{\N}^{2\rkn-1},\chix^{-1}}\circ L.$$
The Lemma follows now from \eqref{eq: iter2} (with $k=2\rkn-1$).
\end{proof}

\subsection{} Let $\remn\subset\GLnn$ be the direct product of the commuting one-parameter root subgroups
$\langle N^\GLnn_{i,j}: \rkn<i\leq 2\rkn, 1<j\leq 2\rkn+1-i\rangle$.
We supplement Lemma~\ref{L: Mreg2} with the following:
\begin{lemma}\label{L: regcomp}
Let $K_0\in\csgr(G)$ and $k=1,\dots,\rkn-1$. Then
for any $\phi(t,m)\in\evencls_k (T_\GLnn\times \GLnn)^{K_0}$ we have
\begin{equation} \label{eq: eps2}
\avg^{reg}_{\evenL^k,\psi_{\vs}}\circ\eva{\elemv}\circ \phi(t,\cdot)=
\eva{\elemv}\circ \avg^{reg}_{\evenL^k, \psi_{\vs}}\circ \phi(t,\cdot)
\end{equation}
if either side is well defined.
Similarly, for any $\phi(t,m)\in\oddcls_k (T_\GLnn\times \GLnn)^{K_0}$ (and if $k=\rkn-1$ when $t\in\soment$) we have
\begin{equation} \label{eq: eps22}
\avg_{\oddL^k}\circ\eva{\elemv}\circ \phi(t,\cdot)=\eva{\elemv}\circ \avg_{\oddL^k}\circ \phi(t,\cdot).
\end{equation}
Finally, suppose that $\phi(t,m)\in\evencls_\rkn(T_\GLnn\times \GLnn)^{K_0}$.
Then for all $t\in\Omega_{T_\GLnn}$ a compact subset of $T_\GLnn$, as a function of $m\in\remn$, $\eva{\elemv}\circ\phi(t, m)$ is supported on some
$\Omega\in \csgr(\remn)$ determined by $\Omega_{T_\GLnn}$ and $K_0$.
\end{lemma}

\begin{proof}
To prove \eqref{eq: eps2}, we show that for $x\in\evenL^{k}$ and  $\phi(t,m)\in\evencls_k (T_\GLnn\times \GLnn)^{K_0}$,
we have $\phi(t,\elemv xg)=\phi(t,x\elemv g)$.
This follows from \eqref{p: equim} since by  Lemma \ref{lem: elemv} and \eqref{eq:eps2normfixx},
$[x,\elemv]\in \hat\LL^{2k}\cap \GLnn^\fixx=\hat\M^{2k}$ and $\chixmt([x,\elemv])=1$.
Similarly, for $\phi(t,m)\in\oddcls_k (T_\GLnn\times \GLnn)^{K_0}$ we have $\phi(t,\elemv xg)=\phi(t,x\elemv g)$
for $x\in\oddL^{k}$, hence \eqref{eq: eps22}.

Now suppose that $\phi(t,m)\in\evencls_\rkn(T_\GLnn\times \GLnn)^{K_0}$. Note that for $n\in \hat\M^{2\rkn}$,
$[n,\elemv]\in \hat\M^{2\rkn}$ by Lemma~\ref{lem: Lproperty} part \ref{claim: normalL} and \eqref{eq:eps2normfixx}.
Thus $\eva{\elemv}\circ\phi(t,nm)=\phi(t,\elemv nm)=\chixmt'(n)\eva{\elemv}\circ\phi(t,m)$ where $\chixmt'(n)=\chixmt(\elemv n(\elemv)^{-1})$.

For fixed $j$, let $Y_j$ be  the product of $\langle N^\GLnn_{i,j}: n< i\leq 2\rkn+1-j\rangle$.
Let $N_{\triangle}^k=\prod_{j=k}^\rkn Y_j$. Then
\begin{equation}\label{eq: Ndelta}
N_{\triangle}^j=N_{\triangle}^{j+1}\times Y_j,\,\,\remn=N_{\triangle}^2,\,\,N_{\triangle}^{\rkn+1}=1.
\end{equation}

Fix $j$ with $1<j\leq \rkn$. Let $Y'_j$ be the product of $\langle N^\GLnn_{j-1,i}: \rkn< i\leq 2\rkn+1-j\rangle$.
Then $Y'_j\subset \hat\M^{2\rkn}$ and $cdc^{-1}\in \hat\M^{2\rkn}$ when $c\in N_{\triangle}^j$ and $d\in Y_j'$.
Moreover,  $\chixmt'([c,d])=\psi(\sprod{\overline c}{d}_{Y_j,t})$ where $\overline{c}$
is the projection of $c$ to $Y_j$ and $\sprod{}{}_{Y_j,t}$ is a non-degenerate pairing between $Y_j$ and $Y'_j$
depending continuously on $t$.
Thus for any  $r\in N_{\triangle}^{j+1}$, $c\in Y_j$ and $d\in Y_j'$ we have
\[
\phi(t,rcd)=\psi(\sprod{c}{d}_{Y_j,t})\chixmt'(d)\phi(t, rc)
\]
by condition \eqref{p: equim}. Lemma~\ref{L: elemC} implies that for any $t\in\Omega_{T_\GLnn}$,
$c\mapsto\phi(t, rc)$ is supported on some
$\Omega_j\in\csgr(Y_j)$ which is determined by $\Omega_{T_\GLnn}$, and $K_0$ and independent of $r\in N_{\triangle}^{j+1}$.
It follows from \eqref{eq: Ndelta}, $m\mapsto\phi(t,m)$ is compactly supported on $N_{\triangle}^j$ if and only if it is so on
$N_{\triangle}^{j+1}$. Since $N_{\triangle}^{\rkn+1}=1$, we conclude that $\phi(t,m)$ is compactly supported on
$\remn  =N_{\triangle}^2$, uniformly for $t\in\Omega_{T_\GLnn}$.
\end{proof}

\begin{corollary} \label{cor: Mreg}
Let $W\in C^{\smth}(Z_{\Levi}'N\bs G,\psi_{N})$, $t\in\soment$ and $\auxf\in\cc{\spcltrs}=\auxF$.
Then
\begin{multline}\label{eq: entire}
\tran_*^{t,\auxf}(W_s):=\factor(t)^{-1}
\avg_{\remn  }\circ
\avg_{\oddL^{\rkn-1}}\circ\avg^{reg}_{\evenL^{\rkn-1},\psi_{\vs}}\circ\avgt{\Imk_{\rkn-1}}[f_{\rkn-1}]\circ\cdots\\
\circ\avg_{\oddL^1}\circ\avg^{reg}_{\evenL^1,\psi_{\vs}}\circ\avgt{\Imk_1}[f_1]\circ\eva{\elemv}\circ\FW_{W_s}(t,e)
\end{multline}
is analytic in $s$ and locally constant in $t\in\soment$, uniformly in $s$. It equals
\begin{multline*}
\factor(t)^{-1}\avg_{\remn  }\circ\eva{\elemv}\circ
\avg_{\oddL^{\rkn-1}}\circ\avg^{reg}_{\evenL^{\rkn-1},\psi_{\vs}}\circ\avgt{\Imk_{\rkn-1}}[f_{\rkn-1}]\circ\cdots\\
\circ\avg_{\oddL^1}\circ\avg^{reg}_{\evenL^1,\psi_{\vs}}\circ\avgt{\Imk_1}[f_1]\circ\FW_{W_s}(t,e).
\end{multline*}
Moreover it equals
\begin{equation}\label{eq: defLEmult}
\factor(t)^{-1}\int_{\spcltrs}\int_{\levi(\vs)\ltimes\remn}\int_{\bar U}
\auxf(t')W\left(\levi(tt')\bar u \levi(\elemv v)\cdot \right)\psi_{\vs}(r)\newchar(\bar u)\ d\bar u\ d r \ dt'
\end{equation}
when the integrals converge. Here the character $\psi_{\vs}$ on $\vs\ltimes \remn $ factors through $\vs$.
\end{corollary}

\begin{proof}
It follows from Lemma \ref{L: Ureg}, Corollary \ref{cor: abcde}
and Lemma \ref{L: regcomp}  that, for all $t$ in a fixed compact subset
$\Omega_{\soment}$, the iterated integral \eqref{eq: entire} equals the integration \eqref{eq: defLEmult} with the domain of integration
for $\bar u$ and $r$ replaced by some  open compact subgroups $\Omega_1\in \csgr(\bar U)$ and $\Omega_2\in \csgr(\vs\ltimes \remn)$,
dependent on $W$ and $\Omega_{\soment}$, but independent of $s$. Thus we get that
\eqref{eq: entire} is analytic in $s$ and locally constant in $t\in\soment$, uniformly in $s$.
The second claim follows from  Lemmas~\ref{L: regcomp}, as clearly the operators $\avgt{\Imk_i}[f_i]$ and $\eva{\elemv}$ commute.
The last claim  follows from Lemma~\ref{L: stableprop}.
\end{proof}

From Proposition~\ref{prop: main}, we get for $L\in\tclass(H\bs G)$ and $W=\tranLW(L)$, for any $t\in\soment$ and
$\auxf\in\auxF$,
\begin{equation}\label{eq: lastcomp}
\tran_*^{t,\auxf}(W)=c(\auxf)\avg_{\levi(\remn ) }\circ\eva{\levi(\elemv)}\circ\tranLE(L).
\end{equation}

\subsection{}
Let $\remu=\few^{-1}\levi(\remn  )\few$. Then $\remu\subset \bar U$ and $\remu\cap\subsubUbar=1$. Let $\subUbar=\remu\times \subsubUbar$.
It consists of $\bar u$ with $\bar u_{2\rkn+i,j}=0$ whenever $i\geq \rkn$ and $j\leq \rkn$. We extend the character $\oldsubsubUbar$ trivially to $\subUbar$.
It is proved in \cite[Lemma 4.5]{1401.0198} that for any $W\in C^{\smth}(N\bs G,\psi_N)$,
$W\rest_{\subUbar}$ is compactly supported. In particular, the integral
\[
\tran'(W):=\avg_{\subUbar,\oldsubsubUbar }\circ W
\]
is well defined.

We have the following corollary of the previous discussion.
\begin{corollary} \label{cor: fornextpaper}
Let $\pi\in \Irr_{\temp,\meta}\GLnn$, let $W\in \Ind(\WhitM(\pi),\frac12)$, then for any $t\in\soment$ and
$\auxf\in\auxF$, $\tran_*^{t,\auxf}(\W)(\few\cdot)=c(\auxf)\fel^\rkn \tran'(\W)(\cdot)$.
Moreover $\tran'(\W)=\fel^{\rkn+1}\eva{\few}\circ \avg_{\levi(\remn ) }\circ\eva{\levi(\elemv)}\circ\tranLE(L)$.
\end{corollary}

\begin{proof}
By Proposition~\ref{P: L}, we see
$$\tran_*^{t,\auxf}(\W)=\fel \tran_*^{t,\auxf}(\tranLW(L)),\ \
\tran'(\W)=\fel \tran'(\tranLW(L))$$
 where $L=\tranWL(W)\in \tclass^{\fel}(H\bs G)$.  By \eqref{eq: lastcomp},
\[
\tran_*^{t,\auxf}(\W)=c(\auxf)\fel\avg_{\levi(\remn ) }\circ\eva{\levi(\elemv)}\circ\tranLE(L).
\]
On the other hand,
\[
\tran'(\W)=\fel
\avg_{\remu}\circ\eva{\elemu}\circ\tranLA(L)
\]
by Lemma \ref{lem: 14var}. Using \eqref{eq: relationetwon}, we get
\[
\tran'(\W)=\fel^{\rkn+1}\avg_{\remu}\circ\eva{\elemu}\circ\eva{\few}\circ\tranLE(L).
\]
The Corollary follows from the relation
\[
\avg_{\remu}\circ\eva{\elemu}\circ\eva{\few}=\avg_{\remu}\circ\eva{\few}\circ\eva{\levi(\elemv)}=
\eva{\few}\circ\avg_{\levi(\remn)}\circ\eva{\levi(\elemv)}.
\qedhere
\]
\end{proof}

\printindex

\appendix
\section{Some integration lemmas}\label{a: int}

In this appendix all groups are algebraic over a local field $F$ (or their groups of $F$-points) and all isomorphisms are
defined over $F$.
We fix a group $\alg{G_0}$, a closed subgroup $\alg{H_0}$ such that $\alg{H_0}\bs\alg{G_0}$ is affine,
and a continuous character $\chi_0:H_0\rightarrow\R_{>0}$.
The following statements are in the spirit of \cite[\S7.1]{MR2848523}. (See also \cite[Lemma 2.2]{MR1671452}.)

We fix a non-trivial character $\psi$ of $F$.
It determines a self-dual Haar measure on $F$, as well as a Haar measure on $F^*$.
Throughout this section we assume that Haar measures on groups and their semidirect products are chosen
compatibly, and all isomorphisms are compatible with Haar measures.

\begin{lemma}\label{L: Fourier}
Let $A,A',\tilde A,B,C,D$ be closed unipotent subgroups of $G_0$ satisfying the following properties.
\begin{enumerate}
\item $A=D\cdot B$, $D$ normalizes $B$ and $B\cap D\subset H_0$.
\item $\tilde A=C\ltimes A$, $A'=C\ltimes B$.
\item $A\cap H_0=(D\cap H_0)\cdot(B\cap H_0)$.
\item $D\cap H_0$ is normal in $D$ and $D\cap H_0\bs D$ is abelian.
\item $C\subset H_0$ and there is a homeomorphism $\iota$ between $C$ and the Pontryagin dual of $D\cap H_0\bs D$,
preserving Haar measures.
\end{enumerate}
Let $\Psi\in C^\infty(A\cap H_0\bs\tilde A)$ be such that for any $a\in A$, $c\in C$:
\[
\Psi(ca)=\Psi(cac^{-1})=\Psi(a)\sprod{\iota(c)}{\overline a},
\]
where
$\overline{a}$ is the image of $a$ under the projections $A\rightarrow D\cap H_0\bs D$.
Then we have $\avg_{C}\circ\avg^{H_0}_{A,\Psi}\circ f=\avg^{H_0}_{B,\Psi}\circ f=\avg^{H_0}_{A',\Psi}\circ f$
for any $f\in C_c^{\infty}(H_0\bs G_0,\chi_0)$, i.e.
\begin{equation}\label{eq: Fourier1}
\int_C\big(\int_{A\cap H_0\bs A} f(ac)\Psi(a)\,da\big)\,dc=\int_{B\cap H_0\bs B} f(b)\Psi(b)\,db=
\int_{A'\cap H_0\bs A'}f(a)\Psi(a)\ da.
\end{equation}
\end{lemma}

\begin{proof}
First note that the restriction of $f$ to $A$ is compactly supported modulo $A\cap H_0$
since $A\cap H_0\bs A$ is closed in $H_0\bs G_0$, as a unipotent orbit on an affine variety.
Similarly for $A'$ and $B$.
Make a change of variable $a\mapsto cac^{-1}$ on the left-hand side of \eqref{eq: Fourier1}.
By left $(H_0,\chi_0)$-equivariance of $f$, it is equal to
\[
\int_C\big(\int_{A\cap H_0\bs A} f(a)\Psi(a)\sprod{\iota(c)}{\overline a}\,da\big)\,dc.
\]
This integral equals
\[
\int_C\big(\int_{D\cap H_0\bs D}\int_{B\cap H_0\bs B} f(bd)\Psi(bd)\sprod{\iota(c)}{\bar d}\,db\,dd\big)\,dc.
\]
Let $\phi\in C_c^{\infty}(D\cap H_0\bs D)$ be the function $\phi(\bar d)=\int_{B\cap H_0\bs B} f(bd)\Psi(bd)\,db$.
Then the above is just
\[
\int_C \hat{\phi}(\iota(c))\,dc=\phi(e)=\int_{B\cap H_0\bs B} f(b)\Psi(b)\,db.
\]
The second equality of \eqref{eq: Fourier1} is clear.
\end{proof}

\begin{remark}
In most cases we will apply the lemma when $B\cap D=1$, i.e. $A=D\ltimes B$.
\end{remark}

\begin{lemma}\label{C: Fourier2}
Let $A,A',\tilde A,B,C,D$ be closed unipotent subgroups of $G_0$ satisfying the following properties.
\begin{enumerate}
\item $D\subset H_0$.
\item $A=D\ltimes B$, $A'=C\ltimes B$, $\tilde A=C\ltimes A$.
\item $A'\cap H_0=B\cap H_0$.
\end{enumerate}
Let $\Psi\in C^\infty(A\cap H_0\bs\tilde A)$ be such that $\Psi(ac)=\Psi(a)\Psi(c)$ for all $a\in A$, $c\in C$.
Then we have $\avg_{C,\Psi}\circ\avg^{H_0}_{A,\Psi}\circ f=\avg^{H_0}_{A',\Psi}\circ f$
for any $f\in C_c^{\infty}(H_0\bs G_0,\chi_0)$, i.e.
\begin{equation}\label{eq: Fourier2}
\int_C\big(\int_{A\cap H_0\bs A} f(ac)\Psi(ac)\,da\big)\,dc=
\int_{A'\cap H_0\bs A'} f(a')\Psi(a')\,da'.
\end{equation}
\end{lemma}

\begin{proof}
Since $A=D\ltimes B$ and $D\subset H_0$ we have $\avg^{H_0}_{A,\Psi}=\avg^{H_0}_{B,\Psi}$.
Since $A'\cap H_0=B\cap H_0$ we have $\avg_{C,\Psi}\circ\avg^{H_0}_{B,\Psi}=\avg^{H_0}_{A',\Psi}$.
\end{proof}

We will also need the following variant of Lemma \ref{L: Fourier}.
This will only be used in \S\ref{sec: Mfactor}.

\begin{lemma}\label{L: Fourier2}
Let $A,B,C,D$ be closed subgroups of $G_0$ satisfying the following properties.
\begin{enumerate}
\item $A$ is a unipotent group.
\item $A=D\ltimes B$ and $A\cap H_0=B\cap H_0$.
\item There is an isomorphism $\iota:D\rightarrow F$.
\item There is an isomorphism $\tau:F^*\rightarrow C$.
\item $C$ normalizes $A$.
\item $C\subset H_0$.
\item $\chi_0\circ\tau=1$.
\end{enumerate}
Let $\Psi\in C^\infty(A\cap H_0\bs A)$ be such that for any $a\in A$, $t\in F^*$
\[
\Psi(\tau(t)a\tau(t^{-1}))=\Psi(a)\psi((1-t)\iota(\bar a)),
\]
where $\overline{a}$ is the image of $a$ under the projections $A\rightarrow D$.
Let $\Xi(c)$, $c\in C$ be the modulus function of $c^{-1}$ acting on $A\cap H_0\bs A$ and set $\Xi'(\tau(t))=\Xi(\tau(t))\abs{t}$, $t\in F^*$.
Then we have $\avg_{C,\Xi'}\circ\avg^{H_0}_{A,\Psi}=\avg^{H_0}_{B,\Psi}$
for any $f\in C_c^{\infty}(H_0\bs G_0,\chi_0)$. In other words,
\begin{equation}\label{eq: Fourier3}
\int_C\big(\int_{A\cap H_0\bs A} f(ac)\Psi(a)\,da\big)\Xi'(c)\ dc=\int_{B\cap H_0\bs B} f(b)\Psi(b)\,db.
\end{equation}
\end{lemma}

\begin{proof}
We make a change of variable $a\mapsto cac^{-1}$ on the left-hand side of \eqref{eq: Fourier3}.
By left $(H_0,\chi_0)$-equivariance of $f$ we get
\[
\int_{F^*}\big(\int_{A\cap H_0\bs A} f(a)\Psi(a)\psi((1-t)\iota(\overline a))\,da\big)\abs{t}\ dt.
\]
This integral equals
\[
\int_{F^*}\big(\int_D\int_{B\cap H_0\bs B} f(bd)\Psi(bd)\psi((1-t)\iota(d))\,db\,dd\big)\abs{t}\,dt.
\]
Let $\phi(\iota(d))=\int_{B\cap H_0\bs B} f(bd)\Psi(bd)\,db$. Then $\phi\in C_c^{\infty}(F)$ and the above integral is just
\[
\int_{F^*}\hat{\phi}(1-t)\abs{t}\ dt=\int_F\hat{\phi}(x)\ dx=\phi(0)=\int_{B\cap H_0\bs B} f(b)\Psi(b)\,db
\]
as required.
\end{proof}

\begin{remark} \label{rem: condFourier}
We will only use the Lemmas above in the case where
$H_0$ is the fixed point subgroup of an involution $\theta$ of $G_0$
which stabilizes $A$, $B$, $C$, $D$.
Thus, the conditions $A\cap H_0=(D\cap H_0)\cdot(B\cap H_0)$
(resp, $A'\cap H_0=B\cap H_0$, $A\cap H_0=B\cap H_0$) are automatically satisfied.

Also, in the cases at hand the restriction of $\Psi$ to $A$ and to $A'$ will be a character (and $\Psi\rest_{A\cap H_0}\equiv1$).
(Note however that $\Psi$ is not a character on $\tilde A$.)
Thus, the condition on $\Psi$ in Lemma \ref{L: Fourier} is that it is right $C$-invariant and
$\Psi([c,d])=\sprod{\iota(c)}{\overline d}$ for all $c\in C$, $d\in D$.
\end{remark}

\section{A majorization lemma}
For this section let $G=\GL_m$ and let $B=T\rtimes N$ be the standard Borel subgroup of $G$ where $N$ is the subgroup of upper
unitriangular matrices and $T$ is the diagonal torus. Let also $K$ be the standard maximal compact subgroup of $G$.
Let $w_0\in G$ be the permutation matrix corresponding to the longest Weyl element.
We denote by $\|h\|$ be the maximum of the norm of the entries of $h$ and $h^{-1}$, and let $\sigma(h)=\max(1,\log_q(\|h\|))$.
We have $\sigma(g)\ge1$, $\sigma(gh)\le\sigma(g)+\sigma(h)$ and $\sigma$ is bi-$K$-invariant.

Let $\Pi_0$ be the representation parabolically induced from the trivial representation of the $B$.
Let $\phi_0$ be the unramified vector in $\Pi_0$ such that $\phi_0(e)=1$.
For any $c\ge0$ let $N_c=\{u\in N:\val(u_{i,i+1})\ge -c, i=1,\dots,m-1\}$.
Similarly for $\bar N_c\subset \bar N=N^t$. We show the following variant of \cite[Proposition 3.2]{WaldAst3461}:
\begin{lemma} \label{lem: mainbnd018}
For any $A\ge 0$ there exists ${A'}\ge0$ such that
\[
\int_{N_c}\phi_0(w_0ug)\sigma(u)^A\ du\ll_A(\sigma(g)+c)^{A'}\phi_0(g)
\]
for all $c\ge1$ and $g\in\GL_m$.
\end{lemma}

\begin{proof}
First, we reduce to the case $g=1$, i.e., to the inequality
\begin{equation} \label{eq: g=1bnd}
\int_{N_c}\phi_0(w_0u)\sigma(u)^A\ du=\int_{\bar N_c}\phi_0(\bar u)\sigma(\bar u)^A\ d\bar u\ll_Ac^{A'}, \ \ \ c\ge1.
\end{equation}
Indeed, write $g=nak$ where $n\in N$, $a\in T$ and $k\in K$. Then by a change of variable we get
\begin{multline*}
\int_{N_c}\phi_0(w_0ug)\sigma(u)^A\ du=
\int_{N_cn}\phi_0(w_0uak)\sigma(un^{-1})^A\ du\\=
\delta_0(a)^{\frac12}\int_{a^{-1}N_cna}\phi_0(w_0uk)\sigma(aua^{-1}n^{-1})^A\ du=
\delta_0(a)^{\frac12}\int_{a^{-1}N_cna}\phi_0(w_0u)\sigma(aua^{-1}n^{-1})^A\ du.
\end{multline*}
By matrix multiplication $\sigma(g)=\sigma(na)\ge\sigma(a)$. Therefore, $\sigma(g)\ge\frac13(\sigma(n)+\sigma(a))$ and
the above is
\[
\ll_A\delta_0(a)^{\frac12}\int_{a^{-1}N_cna}\phi_0(w_0u)(\sigma(u)^A+\sigma(g)^A)\ du.
\]
Note that $a^{-1}N_cna\subset N_{c+\sigma(n)+\max_i\log_q\abs{a_{i+1}/a_i}}\subset N_{c+4\sigma(g)}$.
Thus, the estimate \eqref{eq: g=1bnd} will give
\[
\ll_A\delta_0(a)^{\frac12}(c+\sigma(g))^{A'}=\phi_0(g)(c+\sigma(g))^{A'}
\]
(possibly for a larger ${A'}$).

Let $\bar P=M\rtimes\bar U$ be the parabolic subgroup stabilizing the line $(0,\dots,0,*)^t$.
Let $\bar U_c=\bar U\cap \bar N_c$.
We will prove \eqref{eq: g=1bnd} by induction on $m$. For the induction step it suffices to prove that for any $A\ge0$ there exists ${A'}\ge0$ such that
\begin{equation} \label{eq: indbnd24}
\int_{\bar U_c}\phi_0(\bar u_1\bar u_2)\sigma(\bar u_1\bar u_2)^A\ d\bar u_1\ll_A\phi_0(\bar u_2)(\sigma(\bar u_2)+c)^{A'}
\end{equation}
for all $u_2\in\bar N\cap M$ and $c\ge1$.
(Note that $\phi_0(\bar u_2)$ and $\sigma(\bar u_2)$ are unchanged if we view $\bar u_2$
as an element of $\GL_{m-1}$.)

Clearly there exists $A_1>0$ such that
\begin{equation}\label{bound: fact}
\phi_0(gh)\ll \phi_0(g) q^{A_1\sigma(h)}
\end{equation}
for any $g,h\in\GL_m$.
Let $\bar U_\natural$ denote the subgroup of $\bar U$ consisting of elements with $u_{m,m-1}=0$.
On \cite[p. 196-7]{WaldAst3461} it is proved that for any $C'$ there exists $A_2>0$ such that
\begin{equation} \label{ref: Wald2bnd}
\int_{\bar U_\natural} 1_{\sigma\ge C}(u)\phi_0(u)\sigma(u)^{C'}\,du\ll_{C'} q^{-A_2C}
\end{equation}
for all $C\ge 0$.

We will use the following fact (\cite[Lemme II.4.2]{MR1989693}): there exists $A_3>0$ such that
\[
\int_{\bar U}\phi_0(\bar u)\sigma(\bar u)^{-A_3}\,d\bar u<\infty.
\]
This implies for all $C$ and $m\in M$
\begin{equation} \label{eq: bndshell}
\int_{\bar U}1_{\sigma=C}(\bar u_1)\phi_0(m\bar u_1)\ d\bar u_1\ll C^{A_3} \phi_0(m).
\end{equation}
Indeed, writing $m=bk$ where $b\in B\cap M$ and $k\in K\cap M$ then
by a change of variable $\bar u_1\mapsto k^{-1}\bar u_1 k$ the left-hand side of \eqref{eq: bndshell} is
\begin{multline*}
\int_{\bar U}1_{\sigma=C}(\bar u_1)\phi_0(b\bar u_1)d\bar u_1=\int_{\bar U}1_{\sigma=C}(\bar u_1)\phi_0(b)\phi_0(\bar u_1)d\bar u_1
\\\le C^{A_3}\phi_0(m)\int_{\bar U}\phi_0(\bar u_1)\sigma(\bar u_1)^{-A_3}d\bar u_1 \ll C^{A_3} \phi_0(m).
\end{multline*}
By a similar reasoning
\begin{equation} \label{eq: bnd192}
\int_{\bar U}\phi_0(m\bar u_1)\sigma(\bar u_1)^{-A_3}\ d\bar u_1\ll \phi_0(m).
\end{equation}

We break the integral in \eqref{eq: indbnd24} according to the regions
$\sigma(\bar u_1\bar u_2)\le C_1\sigma(\bar u_2)$ and
$\sigma(\bar u_1\bar u_2)>C_1\sigma(\bar u_2)$ where $C_1>1$ is a constant which will be determined below.
We write $I_1$ and $I_2$ for the corresponding integrals so that
\begin{equation} \label{eq: I1I2}
\int_{\bar U_c}\phi_0(\bar u_1\bar u_2)\sigma(\bar u_1\bar u_2)^A\ d\bar u_1=I_1+I_2.
\end{equation}

Write $\bar u_1\bar u_2=\bar u_2\bar u_1'$.
In the domain of $I_1$ we have $\sigma(\bar u_1')\leq (C_1+1)\sigma(\bar u_2)$.
Thus
\begin{multline} \label{eq: bndI1}
I_1\le \int_{\bar U} (C_1\sigma(\bar u_2))^A 1_{\sigma\le(C_1+1)\sigma(\bar u_2)}(\bar u_1')\phi_0(\bar u_2\bar u_1')\,d\bar u_1'\\
\le \int_{\bar U} ((C_1+1)\sigma(\bar u_2))^{A+A_3} \phi_0(\bar u_2\bar u_1')\sigma(\bar u_1')^{-A_3}\,d\bar u_1
\ll((C_1+1)\sigma(\bar u_2))^{A+A_3}\phi_0(\bar u_2)
\end{multline}
by \eqref{eq: bnd192}.

Consider now $I_2$. On the domain of $I_2$ we have $\sigma(\bar u_1)> (C_1-1)\sigma(\bar u_2)$.
Also $\sigma(\bar u_1\bar u_2)\le\sigma(\bar u_1)+\sigma(\bar u_2)<\sigma(\bar u_1)+C_1^{-1}\sigma(\bar u_1\bar u_2)$
so that $\sigma(\bar u_1\bar u_2)<(1-C_1^{-1})^{-1}\sigma(\bar u_1)$.
We write $\bar u_1=\bar v_1\bar v_2$ where $\bar v_1\in U_\natural$ and $\bar v_2=I+y\epsilon_{m,m-1}$.
Once again, we break the domain of integration of $I_2$ to $\val y\ge -C_2\sigma(\bar u_1)$ and
$-c\le\val y<-C_2\sigma(\bar u_1)$ where $0<C_2<1$ is a small constant which will be fixed later.
We write $I_3$ and $I_4$ for the corresponding integrals so that
\begin{equation} \label{eq: I2I3I4}
I_2=I_3+I_4.
\end{equation}
In domain of $I_3$ we have $\sigma(\bar u_1)=\sigma(\bar v_1)$ and $\sigma(\bar v_2)\le 1+C_2\sigma(\bar u_1)$.
Thus, by \eqref{bound: fact}
\begin{multline*}
\phi_0(\bar u_1\bar u_2)=\phi_0(\bar v_1\bar v_2\bar u_2)\ll \phi_0(\bar v_1)q^{A_1(\sigma(\bar v_2)+\sigma(\bar u_2))}\\\le
\phi_0(\bar v_1)q^{A_1(C_2\sigma(\bar u_1)+\sigma(\bar u_2))}=
\phi_0(\bar v_1)q^{A_1(C_2\sigma(\bar v_1)+\sigma(\bar u_2)+1)}.
\end{multline*}
Together with \eqref{ref: Wald2bnd} it follows that (here $V(x)=\vol(\{y\in F:\val y\ge x\})$)
\begin{multline*}
I_3\ll \sum_{C>(C_1-1)\sigma(\bar u_2)}q^{A_1(\sigma(\bar u_2)+C_2C+1)}\int_{\bar U_\natural}(1-C_1^{-1})^{-A} 1_{\sigma=
C}(\bar v_1)\phi_0(\bar v_1)\sigma(\bar v_1)^A V(-C_2\sigma(\bar v_1))\,d\bar v_1\\
\ll_A \sum_{C>(C_1-1)\sigma(\bar u_2)}q^{A_1\sigma(\bar u_2)+A_1+C_2C(A_1+1)}(1-C_1^{-1})^{-A}q^{-CA_2}.
\end{multline*}
Taking $C_2=\frac12 A_2/(A_1+1)$ and $C_1=1+4A_1/A_2$ we get
\begin{equation} \label{eq: bndI3}
I_3\ll_A q^{-A_1\sigma(\bar u_2)}\ll\phi_0(\bar u_2)
\end{equation}
by \eqref{bound: fact}.

For $I_4$, once again write $\bar u_1\bar u_2=\bar u_2\bar u_1'$.
By the same reasoning as above, $\sigma(\bar u_2\bar u_1')<(1-C_1^{-1})^{-1}\sigma(\bar u_1')$.
On the domain of integration of $I_4$ we have $\sigma(\bar u_1)<c/C_2$ and hence
$\sigma(\bar u_1')\le\sigma(\bar u_1)+2\sigma(\bar u_2)<cC_2^{-1}+2\sigma(\bar u_2)$.
Therefore for suitable $A_4$ we have
\begin{equation} \label{eq: bndI4}
I_4\ll \int_{\bar u_1'\in\bar U:\sigma(\bar u_1')<cC_2^{-1}+2\sigma(\bar u_2)} \phi_0(\bar u_2\bar u_1')(1-C_1^{-1})^{-A} \sigma(\bar u_1')^A\,d\bar u_1'
\ll_A
\phi_0(\bar u_2)(c+\sigma(\bar u_2))^{A_4}
\end{equation}
by \eqref{eq: bndshell}.
Combining \eqref{eq: I1I2}, \eqref{eq: bndI1}, \eqref{eq: I2I3I4}, \eqref{eq: bndI3} and \eqref{eq: bndI4} we get \eqref{eq: indbnd24}.
The lemma follows.
\end{proof}

\section{Uniqueness of degenerate Whittaker models for $\LQ{\pi}$\\by Marko Tadi\'c} \label{sec: Tadic}
\newcommand{\JGL}{J_{\GL}}
\newcommand{\sGL}{s_{\GL}}
\newcommand{\sGLgen}{s_{\GL,\gen}}
\newcommand{\MGL}{M_{\GL}^*}
\newcommand{\MGLgen}{M_{\GL,\gen}^*}
\newcommand{\Speh}{\mathcal{SP}}
\newcommand{\Groth}{R}
\newcommand{\reps}{\mathcal{R}}

As in the body of the paper for $\pi\in\Irr_{\temp}\GL_{2\rkn}$ we denote by $\LQ{\pi}$ the Langlands quotient
of the induced representation $I(\pi,\frac12)=\Ind_P^G(\pi\abs{\det}^{\frac12})$ where $G=\Sp_{2\rkn}$ and
$P=M\rtimes U$ is the (upper triangular) Siegel parabolic subgroup of $G$ (with $M=\GL_{2\rkn}$).
Also recall that $N$ is the group of upper unitriangular matrices in $G$ and $\psi_N$ is a character on $N$
which is trivial on $U$ and non-degenerate on $M\cap N$.
The goal of this section is to prove the following uniqueness result.

\begin{proposition} \label{prop: Tadic}
Let $\pi\in\Irr_{\temp,\meta}\GL_{2\rkn}$. Then
\[
\dim\Hom_N(\LQ{\pi},\psi_N)=1.
\]
In other words, (since $\Hom_N(\LQ{\pi},\psi_N)=\Hom_{N_M}(\JGL(\LQ{\pi}),\psi_N)$) the Jacquet module $\JGL(\LQ{\pi})$ of $\LQ{\pi}$
with respect to $P$ (a finite length representation of $\Levi$) has a unique irreducible generic subquotient.
\end{proposition}

In the case where $\pi$ is parabolically induced from distinct supercuspidal representations,
the proposition was proved in \cite[Proposition 5.12]{MR1954940}.

We first recall some standard facts and notation about Jacquet modules of general linear groups and symplectic groups.
\begin{itemize}
\item For any $\ell$-group $Q$ we denote by $\Groth(Q)$ the Grothendieck group of the category $\reps(Q)$ of smooth
(complex) representations of finite length of $Q$. It is an ordered abelian group.
\item We denote by $1$ the one-dimensional representation of the trivial group ($\GL_0$ or $\Sp_0$ depending on the context).
\item We write $\reps(\GL)=\oplus_{m\ge0}\reps(\GL_m)$ and $\Irr\GL=\cup_{m\ge0}\Irr\GL_m$.
\item For $\delta_i\in\reps(\GL_{m_i})$, $i=1,2$ we denote by $\delta_1\times\delta_2\in\reps(\GL_{m_1+m_2})$
the representation parabolically induced from $\delta_1\otimes\delta_2$ (with respect to the upper triangular parabolic subgroup).
\item For $\delta\in\reps(\GL_k)$ and $\sigma\in\reps(\Sp_m)$ we denote by $\delta\rtimes\sigma\in\reps(\Sp_{k+m})$
the representation parabolically induced from $\delta\otimes\sigma$ (once again, with respect to the upper triangular parabolic subgroup).
\item If $\delta\in\reps(\GL_m)$ we denote by $\delta\nu^s$, $s\in\C$ the twist of $\delta$ by the character $\abs{\det\cdot}^s$ of $\GL_m$.
\item If $a,b\in\R$ with $b-a\in\Z_{\ge0}$ we write $[a,b]=\{a,\dots,b\}$. By a segment $[a,b]_\rho$ we mean a set
of the form $\{\rho\nu^a,\dots,\rho\nu^b\}$ where $\rho$ is an irreducible supercuspidal representation of some $\GL_m$.
To such a segment one associates an essentially square-integrable representation $\Delta([a,b]_\rho)$ of $\GL_{m(b-a+1)}$.
All essentially square-integrable representations of $\GL_n$ are obtained this way, for a unique segment.
By convention the empty segment corresponds to $1$, i.e. $\Delta([a,a-1]_{\rho})=1$.
\item We write $[a_1,b_1]_{\rho_1}\mapsto[a_2,b_2]_{\rho_2}$ if
\begin{enumerate}
\item $[a_1,b_1]_{\rho_1}\cup[a_2,b_2]_{\rho_2}$
forms a segment which properly contains $[a_1,b_1]_{\rho_1}$ and $[a_2,b_2]_{\rho_2}$, and,
\item  $\rho_2\nu^{a_2}=\rho_1\nu^{a_1+l}$ with $l\in\Z_{>0}$.
\end{enumerate}
\item More generally, for a multisegment $[a_1,b_1]_{\rho_1},\dots,[a_k,b_k]_{\rho_k}$ ordered such that
$[a_i,b_i]_{\rho_i}\not\mapsto[a_j,b_j]_{\rho_j}$ for all $i<j$, we denote by
\[
\Delta([a_1,b_1]_{\rho_1},\dots,[a_k,b_k]_{\rho_k})
\]
the Langlands quotient (i.e., the unique irreducible quotient) of $\Delta([a_1,b_1]_{\rho_1})\times\dots\times\Delta([a_k,b_k]_{\rho_k})$.
This gives a bijection between irreducible representations of $\GL_n$ and multisegments.
The supercuspidal support of the representation above is $\sum_{i=1}^k[a_i,b_i]_{\rho_i}$ (multiset union).
\item For $\delta\in\Irr_{\temp}\GL_m$ let $\Speh(\delta)\in\Irr\GL_{2m}$ be the Langlands quotient of $\delta\nu^{\frac12}\times\delta\nu^{-\frac12}$.
\item We denote by $\sGL(\sigma)\in\Groth(\GL_m)$ semisimplification of the Jacquet module $\JGL(\sigma)$ of $\sigma\in\reps(\Sp_m)$
with respect to the Siegel parabolic
and by $\sGLgen(\sigma)\in\Groth(\GL_m)$ the generic part of $\sGL(\sigma)$ (that is, we retain only the generic
representations in the composition series of $\JGL(\sigma)$, with the same multiplicities).
\item We have
\[
\sGL(\delta\rtimes\sigma)=\MGL(\delta)\times\sGL(\sigma)
\]
where $\MGL:\Groth(\GL_m)\rightarrow\Groth(\GL_m)$ is defined in \cite[\S2]{MR2504024}.
It satisfies $\MGL(\delta_1\times\delta_2)=\MGL(\delta_1)\times\MGL(\delta_2)$.
\item We denote by $\MGLgen(\delta)$ the generic part of $\MGL(\delta)$.
\item For our purposes we only need to know $\MGL(\delta)$ and $\MGLgen(\Speh(\delta))$ for essentially square-integrable $\delta$.
We have
\begin{equation} \label{eq: MGLsqr}
\MGL(\Delta([a,b]_\rho))=\sum_{c\in [a-1,b]}\Delta([c+1,b]_\rho)\times\Delta([-c,-a]_{\dual\rho})
\end{equation}
while it follows from \cite{MR2996769} that if $\delta=\Delta([-a,a]_\rho)$ then $\MGL(\Speh(\delta))$ equals
\[
\sum\Delta([c+\frac12,a+\frac12]_\rho,[d+\frac12,a-\frac12]_\rho)\times
\Delta([\frac12-d,a+\frac12]_{\dual\rho},[\frac12-c,a-\frac12]_{\dual\rho}).
\]
where the sum is over $c\in [-a,a+1]$ and $d\in [-a-1,a]$ such that $d<c$. In particular,
\begin{equation} \label{eq: JGLSpeh}
\MGLgen(\Speh(\delta))=\delta\nu^{-\frac12}\times\dual\delta\nu^{-\frac12}.
\end{equation}
\item If $\delta\in\Irr_{\temp}\GL_m$ then we can write $\delta=\delta_1\times\dots\times\delta_k$
with $\delta_i$ square-integrable. Then $\Speh(\delta)=\Speh(\delta_1)\times\dots\times\Speh(\delta_k)$.
It follows that \eqref{eq: JGLSpeh} holds for $\delta\in\Irr_{\temp}\GL_m$ as well.
\end{itemize}

We are now ready to prove Proposition \ref{prop: Tadic}.
Since $\LQ{\pi}\hookrightarrow I(\pi,-\frac12)$ it follows from Frobenius reciprocity that $\pi\nu^{-\frac12}$
is a quotient of $\JGL(\LQ{\pi})$. In particular,
\[
\Hom_N(\LQ{\pi},\psi_N)\ne0.
\]
We will show that $\pi\nu^{-\frac12}$ is the only generic irreducible subquotient of $\JGL(\LQ{\pi})$, i.e.
that $\sGLgen(\LQ{\pi})=\pi\nu^{-\frac12}$.
Consider first the case where $\pi=\pi_1\times\dots\times\pi_k$ where $\pi_i\in\Irr_{\meta}\GL$
are square-integrable and distinct. We will prove the statement by induction on $k$.
The case $k=1$ follows from the results of \cite{MR2127173}.
More precisely, suppose that $\pi=\Delta([-a,a]_\rho)$ with $\dual\rho=\rho$.
Let $\epsilon\ge0$ be such that $\rho\nu^\epsilon\rtimes1$ is reducible.
Since $\pi\in\Irr_{\meta}\GL_{2\rkn}$, $I(\pi,\frac12)$ is reducible and therefore $\epsilon\in[\frac12-a,a+\frac12]$
\cite{MR1658535}.
Also $\epsilon\in\{0,\frac12,1\}$ by the results of Shahidi \cite{MR1070599}.
By \cite{MR2127173} there exist two inequivalent subrepresentations $\delta_1$, $\delta_2$ of $I(\pi,\frac12)$
(with $\delta_2=0$ if $a=\epsilon-\frac12$) such that
\begin{gather*}
\sGL(\delta_1)=\sum_{l\in [-a-\frac12,\epsilon-1]}\Delta([-l,a-\frac12]_\rho)\times\Delta([l+1,a+\frac12]_\rho),\\
\sGL(\delta_2)=\sum_{l\in [\epsilon,a-\frac12]}\Delta([-l,a+\frac12]_\rho)\times\Delta([l+1,a-\frac12]_\rho).
\end{gather*}
Since
\[
\sGL(I(\pi,\frac12))=\MGL(\pi)=\sum_{l\in [-a-\frac12,a+\frac12]}\Delta([-l,a-\frac12]_\rho)\times\Delta([l+1,a+\frac12]_\rho)
\]
we infer that\footnote{In fact, one can easily prove equality although
this is not explicitly mentioned in [ibid.]. We will not need it.}
\[
\sGL(\LQ{\pi})\le\oplus_{l\in [\epsilon,a+\frac12]}\Delta([l+1,a+\frac12]_\rho,[-l,a-\frac12]_\rho).
\]
In particular,
\[
\sGLgen(\LQ{\pi})\le\pi\nu^{-\frac12}.
\]

For the induction step, assuming that the statement holds for $k$ and we will prove it for $k+1$.
Write $\pi_i=\Delta([-a_i,a_i]_{\rho_i})$, $i=1,\dots,k+1$ with $\dual{\rho_i}=\rho_i$.
Let $\pi'=\pi_1\times\dots\times\pi_k$. Since $\LQ{\pi}$ is the unique irreducible quotient of $\pi\nu^{\frac12}\rtimes1$,
it is also a quotient of $\pi_{k+1}\nu^{\frac12}\rtimes\LQ{\pi'}$ (which is a quotient of $\pi\nu^{\frac12}\rtimes1=
\pi_{k+1}\nu^{\frac12}\rtimes(\pi'\nu^{\frac12}\rtimes1)$).
Taking Jacquet module we infer that
\[
\sGL(\LQ{\pi})\le\MGL(\pi_{k+1}\nu^{\frac12})\times\sGL(\LQ{\pi'}).
\]
In particular, by induction hypothesis and \eqref{eq: MGLsqr}
\begin{multline} \label{eq: sglbnd}
\sGLgen(\LQ{\pi})\le \MGL(\pi_{k+1}\nu^{\frac12})\times\pi'\nu^{-\frac12}=
\\\sum_{c\in [-a_{k+1},a_{k+1}+1]}\Delta([c+\frac12,a_{k+1}+\frac12]_{\rho_{k+1}})\times\Delta([\frac12-c,a_{k+1}-\frac12]_{\rho_{k+1}})
\times\pi'\nu^{-\frac12}.
\end{multline}
In a similar vein,
\[
\sGLgen(\LQ{\pi})\le \MGL(\pi_1)\nu^{\frac12}\times\pi''\nu^{-\frac12}
\]
where $\pi''=\pi_2\times\dots\times\pi_{k+1}$.
We may assume without loss of generality that $[-a_{k+1},a_{k+1}]_{\rho_{k+1}}\not\subseteq[-a_i,a_i]_{\rho_i}$
(and hence, $\rho_{k+1}\nu^{a_{k+1}}\notin [-a_i,a_i]_{\rho_i}$) for all $i=1,\dots,k$.
Comparing supercuspidal supports, we observe that there are no common irreducible subrepresentations of
\[
\sum_{c\in [-a_{k+1},a_{k+1}]}\Delta([c+\frac12,a_{k+1}+\frac12]_{\rho_{k+1}})\times\Delta([\frac12-c,a_{k+1}-\frac12]_{\rho_{k+1}})
\times\pi'\nu^{-\frac12}
\]
and $\MGL(\pi_1)\nu^{\frac12}\times\pi''\nu^{-\frac12}$ since $\rho_{k+1}\nu^{a_{k+1}+\frac12}$ occurs in the support of the supercuspidal support
of all summands in the former expression, but not in the latter.
Therefore, by \eqref{eq: sglbnd} $\sGLgen(\LQ{\pi})\le \pi_{k+1}\nu^{-\frac12}\times\pi'\nu^{-\frac12}=\pi\nu^{-\frac12}$ which gives
the induction step.

Consider the general case.
By the result of Matringe \cite{1301.0350}, we can write $\pi=\tau\times\delta\times\dual\tau$
where $\tau$ is tempered and $\delta$ is of the form considered above.
(In fact, $\delta$ is uniquely determined by $\pi$.)
The intertwining operator $\pi\nu^{\frac12}\rtimes1\rightarrow\pi\nu^{-\frac12}\times1$ is the composition of the intertwining operators
\begin{align*}
\tau\nu^{\frac12}\times\delta\nu^{\frac12}\times\dual\tau\nu^{\frac12}\rtimes1\rightarrow\\
\tau\nu^{\frac12}\times\delta\nu^{\frac12}\times\tau\nu^{-\frac12}\rtimes1\rightarrow\\
\tau\nu^{\frac12}\times\tau\nu^{-\frac12}\times\delta\nu^{\frac12}\rtimes1\xrightarrow*\\
\tau\nu^{-\frac12}\times\tau\nu^{\frac12}\times\delta\nu^{-\frac12}\rtimes1\rightarrow\\
\tau\nu^{-\frac12}\times\delta\nu^{-\frac12}\times\tau\nu^{\frac12}\rtimes1\rightarrow\\
\tau\nu^{-\frac12}\times\delta\nu^{-\frac12}\times\dual\tau\nu^{-\frac12}\rtimes1
\end{align*}
Thus, the image $\LQ{\pi}$ is a subquotient of the image of the middle one (denoted $*$) which is
$\Speh(\tau)\rtimes\LQ{\delta}$. (In fact, it is a direct summand since the latter is unitarziable.)
Thus,
\[
\sGL(\LQ{\pi})\le\MGL(\Speh(\tau))\times\sGL(\LQ{\delta})
\]
and hence by the previous part and \eqref{eq: JGLSpeh}
\[
\sGLgen(\LQ{\pi})\le\MGLgen(\Speh(\tau))\times\delta\nu^{-\frac12}=\pi\nu^{-\frac12}
\]
as required. This concludes the proof of Proposition \ref{prop: Tadic}.

\providecommand{\bysame}{\leavevmode\hbox to3em{\hrulefill}\thinspace}
\providecommand{\MR}{\relax\ifhmode\unskip\space\fi MR }
\providecommand{\MRhref}[2]{%
  \href{http://www.ams.org/mathscinet-getitem?mr=#1}{#2}
}
\providecommand{\href}[2]{#2}

\end{document}